\documentclass[11pt,reqno]{amsart}    
\usepackage[margin=1in]{geometry}
\usepackage{latexsym,amsfonts,amsmath,amssymb,enumerate}    
\usepackage{amsthm}    
\usepackage{mathrsfs}  
\usepackage{setspace}
\onehalfspace
\usepackage{tikz}    
\usetikzlibrary{arrows,matrix,decorations.markings,shapes.geometric}       
\numberwithin{equation}{section}    

\theoremstyle{plain}
\newtheorem{thm}{Theorem}[section]
\newtheorem{lem}[thm]{Lemma}
\newtheorem{prop}[thm]{Proposition}
\newtheorem{cor}[thm]{Corollary}
\theoremstyle{definition}
\newtheorem{defn}[thm]{Definition}

\newtheorem{exmp}[thm]{Example}
\theoremstyle{remark}
\newtheorem{rem}[thm]{Remark}

\newtheorem*{rem*}{Remark}
\newtheorem*{ack}{Acknowledgements}

\newcommand{\bs}{\boldsymbol}
    
\newcommand{\be}{\begin{equation}}    
\newcommand{\ee}{\end{equation}}    
\newcommand{\beu}{\begin{equation*}}    
\newcommand{\eeu}{\end{equation*}}    
\newcommand{\bea}{\begin{eqnarray}}    
\newcommand{\eea}{\end{eqnarray}}    
\newcommand{\beaa}{\begin{eqnarray*}}    
\newcommand{\eeaa}{\end{eqnarray*}}    
\newcommand{\bmx}{\begin{pmatrix}}    
\newcommand{\emx}{\end{pmatrix}}

\newcommand{\g}{{\mathfrak g}}    
\newcommand{\h}{{\mathfrak h}}

\newcommand{\Pp}{\mathcal P^{+}}

\newcommand{\mf}{\mathfrak}
\newcommand{\mc}{\mathcal}    
\newcommand{\al}{{\alpha}}    
\newcommand{\gh}{{\widehat \g}}

\newcommand{\alf}{{\textstyle{\frac{1}{2}}}}

\newcommand{\nn}{\nonumber}

\newcommand{\8}{{\infty}}

\newcommand{\eps}{\epsilon}

\newcommand{\Z}{{\mathbb Z}}

\newcommand{\C}{{\mathbb C}}

\renewcommand{\P}{{\mathcal P}}
\newcommand{\Q}{{\mathcal Q}}    
\newcommand{\R}{{\mathbb R}}

\newcommand{\ket}[1]{{\,\left|#1\right>}\,}

\newcommand{\uq}{{U_q}}
\newcommand{\uqsl}[1]{{U_{q^{r_{#1}}}(\widehat{\mathfrak{sl}}_2{}^{(#1)})}}    
\newcommand{\uqslp}[1]{{U_{q^{r_{#1}}}(\widehat{\mathfrak{sl}}_2)}}    
\newcommand{\uqslt}{{\uq(\widehat{\mathfrak{sl}}_2})}    
\newcommand{\uqgh}{{\uq(\widehat\g)}}    

\newcommand{\uqg}{{\uq(\g)}}

\newcommand{\YY}[2]{Y^{}_{#1, #2}}    
\newcommand{\MM}[2]{Y^{-1}_{#1, #2}}

\newcommand{\Alp}{\mathsf A}

\newcommand{\goi}[2]{=}

\newcommand{\It}{\mathcal X}
\newcommand{\Iw}{\mathcal W}
\newcommand{\Iy}{\mathcal Y}

\newcommand{\on}{}    

\newcommand{\groth}[1]{{\mathrm{Rep}(#1)}}    
\newcommand{\Cx}{\mathbb C^*}
\newcommand{\qnum}[1]{\left[ #1\right]_q}
\renewcommand{\binom}[2]{\begin{bmatrix} #1 \\ #2 \end{bmatrix}}
\newcommand{\qbinom}[2]{\binom{#1}{#2}_q}
    
\newcommand{\btp}{\begin{tikzpicture}[baseline=0pt,scale=0.9,line width=0.25pt]}    
\newcommand{\etp}{\end{tikzpicture}}

\newcommand{\Zys}{\Z\!\left[ Y_{i,a}^{\pm 1} \right]_{i\in I, a \in \Cx}}
\newcommand{\ma}[1]{\mathbb D_{#1}}

\renewcommand{\L}{L}

\newcommand{\scr}{\mathscr}

\newcommand{\atp}[1]{}

\newcommand{\mon}{\mathsf m}

\newcommand{\path}{\longrightarrow}
\newcommand{\confer}{c.f. } 
\newcommand{\mchiq}{\scr M}
\newcommand{\nops}{\overline{\scr P}_{(i_t,k_t)_{1\leq t\leq T}}}

\newcommand{\ps}{{(p_1,\dots,p_T)}}
\newcommand{\pps}{{(p'_1,\dots,p'_T)}}
\newcommand{\sd}{(\lambda/\mu)}
\newcommand{\ptop}{{\left(\phigh_{i_1,k_1},\dots,\phigh_{i_T,k_T}\right)}}
\newcommand{\pbot}{{\left(\plow_{i_1,k_1},\dots,\plow_{i_T,k_T}\right)}}
\newcommand{\phigh}{p^{+}}
\newcommand{\plow}{p^{-}}

\DeclareMathOperator{\wt}{wt}

\DeclareMathOperator{\res}{res}
\DeclareMathOperator{\Span}{span}

\DeclareMathOperator{\Tab}{Tab}
\DeclareMathOperator{\topp}{top}
\newcommand{\T}{\mathcal T}

\newcommand{\sn}{{(i_t,k_t)_{1\leq t\leq T}}}
\newcommand{\bb}{\begin{tikzpicture}[scale=.25] \draw (0,0) rectangle (1,1); \filldraw (.5,.5) circle (.3);\end{tikzpicture}}
\newcommand{\wb}{\begin{tikzpicture}[scale=.25] \draw (0,0) rectangle (1,1); \filldraw[fill=white] (.5,.5) circle (.3);\end{tikzpicture}}
\newcommand{\eb}{\begin{tikzpicture}[scale=.25] \draw (0,0) rectangle (1,1); \end{tikzpicture}}

\begin{document}

\title{Path description of type B $q$-characters}

\author{E. Mukhin} 
\address{Department of Mathematical Sciences, 402
  N. Blackford St, LD 270, IUPUI, Indianapolis, IN 46202, USA.  }
\email{mukhin@math.iupui.edu}

\author{C. A. S. Young}

\address{Department of Mathematics, University of York, Heslington, York YO105DD, UK. (present address)
and Yukawa Institute for Theoretical Physics, Kyoto University,
  Kyoto, 606-8502, Japan.}  \email{charlesyoung@cantab.net}
    
\begin{abstract}
We give a set of sufficient conditions for a Laurent polynomial to be the $q$-character of a 
finite-dimensional irreducible representation of a quantum affine group. We use this result 
to obtain an explicit path description of $q$-characters for a class of modules in type B.
In particular, this proves a conjecture of Kuniba-Ohta-Suzuki.
\end{abstract}

\maketitle
\setcounter{tocdepth}{1}

\section{Introduction}
The category of finite-dimensional representations of quantum affine groups has received a lot of attention in the past several decades. Its fascinating structure has been studied by many authors who used a variety of methods: algebraic e.g. \cite{CPbook,CHbeyondKR,HernandezMinAff,HernandezKR,FR,FM}, geometric e.g. \cite{Nakajima1,VV,Nakajima3}, combinatorial e.g. \cite{SchillingE6, LSS}, and analytic e.g. \cite{BRrsos, KOS}. The full list of the literature on this subject is immense. However, in sharp contrast to the case of affine Lie algebras, where irreducible finite-dimensional representations are just tensor products of evaluation modules, the quantum case is far from being understood. In particular, there is still no formula for dimensions of simple modules, no description of the restriction functor to the quantum groups of finite type, and no tensor product decomposition rules.

Our approach uses a combination of representation-theoretical and combinatorial arguments. The main tool is the theory of $q$-characters which was set up by \cite{FR, FM}. Let $\g$ be a simple Lie algebra, $I$ the set of nodes of the Dynkin diagram of $\g$, and $q\in \Cx$ not a root of unity. Similar to the usual weight theory, finite-dimensional $\uqgh$-modules have the $l$-weight decomposition according to the action of the commutative algebra of Cartan current generators. The $q$-character encodes this decomposition  in terms of formal variables 
$(Y_{i,a})_{i\in I,a\in \Cx}$.  The $q$-characters give an injective ring homomorphism 
from the Grothendieck ring of finite-dimensional $\uqgh$-modules \cite{FR} to the ring of Laurent polynomials in the $(Y_{i,a})_{i\in I,a\in \Cx}$.

The $q$-characters are in some sense more rigid than the usual characters. In particular, in many cases the $q$-character of a simple $\uqgh$-module can be computed recursively  by  an algorithm \cite{FM} starting from its highest $l$-weight. The algorithm works for the class of Kirillov-Reshetikhin modules,  see \cite{Nakajima3,NakajimaKR} for ADE types, \cite{HernandezKR} for other types. It also works for some minimal affinizations \cite{HernandezMinAff}. 
However, the straightforward application of the algorithm is known to  fail in some cases -- see for instance example 5.6 of \cite{HernandezLeclerc}, and  \cite{NN4}. 

In this paper, we prove that a certain finite set of conditions guarantees that the algorithm of \cite{FM} produces the correct $q$-character; see Theorem \ref{thmA}. To the best of our knowledge, it is the first criterion which uses no representation-theoretical information about the module at all, thus making the problem strictly combinatorial. 

Of course, the conditions of Theorem \ref{thmA} do not apply in general. In particular, $q$-characters appearing this way are thin and special, meaning that all $l$-weight spaces are one-dimensional and there is a unique dominant $l$-weight. 
For example, in type A, by \cite{NT} the theorem handles exactly the modules corresponding to skew Young diagrams, where the $q$-characters are known explicitly. We reproduce these results using our methods to illustrate and motivate the main ideas.
However, there are many examples in other types where Theorem \ref{thmA} can be used to obtain new results. In this paper we apply it to a class of simple modules in type B which we call \emph{snake modules}. The snake modules include all Kirillov-Reshetikhin modules, minimal affinizations and modules parameterized by skew Young diagrams, and many others. 

The definition of snake modules (\S\ref{sec:snakes}) is a natural generalization of the definition of minimal affinization in the following sense.
Recall that a tensor product of two fundamental modules, $\L(Y_{i,a}) \otimes \L(Y_{j,b})$
is irreducible for all but a finite set of values of $a/b\in \Cx$. In particular there is one value, of the form $q^{s_{ij}}$, $s_{ij}\in\Z_{>1}$,  
for which the simple quotient $\L( Y_{i,a}Y_{j,aq^{s_{ij}}})$ is the minimal affinization. The integers $s_{ij}$ are known in all types \cite{CPminaffBCFG,CPminaffireg,CPminaffADE}.
In the regular cases, the highest monomials of general minimal affinizations have the form
$\prod_{t}Y_{i_t,a_t}$ such that $a_{t+1}/a_{t}=q^{s_{i_t,i_{t+1}}}$ and $i_t$ form a monotonic sequence with respect to the standard ordering of the Dynkin diagram.
\emph{Minimal snake} modules are the simple modules whose highest monomials have the form
$\prod_{t}Y_{i_t,a_t}$ such that $a_{t+1}/a_{t}=q^{s_{i_t,i_{t+1}}}$ (without the monotonicity conditions). 
In type B we also can weaken the condition $a_{t+1}/a_{t}=q^{s_{i_t,i_{t+1}}}$ in such a way that Theorem \ref{thmA} still applies; we call the resulting simple modules snake modules.

The second main result of the present paper is Theorem \ref{snakechar}, which gives a closed form for the $q$-character of any snake module in type B. Theorem \ref{snakechar} states, in particular, that snake modules are thin and special. The $q$-character is given in terms of a system of \emph{non-overlapping paths}, introduced in \S\ref{sec:paths}. 

The $q$-characters for a subset of snake modules corresponding to skew diagrams have been conjectured in \cite{KOS}, see also \cite{NN1}, motivated by their study of the spectra of solvable vertex models. Their conjectured answer is given in three equivalent forms: skew Young tableaux, paths (different from the ones in this paper) and determinants of Jacobi-Trudi type. 
We use Theorem \ref{snakechar} to settle this conjecture.
  
This paper is structured as follows. The necessary details about quantum affine algebras and their finite-dimensional representations are recalled in Section \ref{qcharsec}.  Section \ref{sec:thinspecial} is devoted to proving Theorem \ref{thmA}, on the $q$-characters of thin special $\uqgh$-modules. In the remainder of the paper we specialize to types A and B. In Section \ref{sec:snakes} we define snakes and snake modules. Section \ref{sec:paths} sets up the definitions of paths, and establishes a series of lemmas about these paths. These allow us, in Section \ref{sec:snakechar}, to prove Theorem \ref{snakechar} which gives a closed form for the $q$-character of any snake module. Finally, in Section \ref{sec:tableaux} we relate our paths to the skew Young tableaux and use Theorem \ref{snakechar} to
prove the Kuniba-Ohta-Suzuki 
conjecture.

\begin{ack} The research of CASY was funded by a Postdoctoral Research Fellowship (grant number P09771) from the Japan Society for the Promotion of Science (until end Oct 2010) and by the EPSRC (from Nov 2010, grant number EP/H000054/1). CASY is grateful to the Department of Mathematical Sciences, IUPUI, for their hospitality during a visit when part of this research was carried out.
The research of EM is supported by the NSF, grant number DMS-0900984.
Computer programs to calculate $q$-characters were written in FORM \cite{FORM}.
\end{ack}

\section{Background}\label{qcharsec}
\subsection{Cartan data}
Let $\g$ be a complex simple Lie algebra of rank $N$ and $\h$ a Cartan subalgebra of $\g$. We identify $\h$ and $\h^*$ by means of the invariant inner product $\left<\cdot,\cdot\right>$  on $\g$ normalized such that the square length of the maximal root equals 2.  Let $I=\{1,\dots,N\}$ and let $\{\alpha_i\}_{i\in I}$ be a set of simple roots, with $\{\alpha^\vee_i\}_{i\in I}$ and $\{\omega_i\}_{i\in   I}$,
the sets of, respectively, simple coroots and fundamental weights. 
Let $C=(C_{ij})_{i,j\in I}$ denote the Cartan matrix. We have 
\be\nn 2 \left< \alpha_i, \alpha_j\right> = C_{ij} \left<   \alpha_i,\alpha_i\right>,\quad 2 \left< \alpha_i, \omega_j\right> = \delta_{ij}\left<\alpha_i,\alpha_i\right> .\ee Let $r^\vee$ be the maximal number of edges connecting two vertices of the Dynkin diagram of $\g$. Thus $r^\vee=1$ if $\g$ is of types A, D or E, $r^\vee = 2$ for types B, C and F and $r^\vee=3$ for $\mathrm G_2$.  Let $r_i= \alf r^\vee \left<\alpha_i,\alpha_i\right>$. The numbers $(r_i)_{i\in   I}$ are relatively prime integers. We set \be\nn D:= \mathrm{diag}(r_1,\dots,r_N),\qquad B := DC;\ee the latter is the symmetrized Cartan matrix, $B_{ij} = r^\vee \left<\alpha_i,\alpha_j\right>$.
 
Let $Q$ (resp. $Q^+$) and $P$ (resp. $P^+$) denote the $\Z$-span (resp. $\Z_{\geq 0}$-span) of the simple roots and fundamental weights respectively. Let $\leq$ be the partial order on $P$ in which $\lambda\leq \lambda'$ if and only if $\lambda'-\lambda\in Q^+$. 
 
Let $\gh$ denote the untwisted affine algebra corresponding to $\g$. 

Fix a $q \in \Cx$, not a root of unity. 
Define 
the $q$-numbers, $q$-factorial and $q$-binomial: \be\nn \qnum n := \frac{q^n-q^{-n}}{q-q^{-1}},\quad \qnum n ! := \qnum n \qnum{n-1} \dots \qnum 1,\quad \qbinom n m := \frac{\qnum n !}{\qnum{n-m} ! \qnum m !}.\ee

\subsection{Quantum Affine Algebras}
The \emph{quantum affine algebra} $\uqgh$ in Drinfeld's new realization, \cite{Drinfeld} is generated by $x_{i,n}^{\pm}$ ($i\in I$, $n\in\Z$), $k_i^{\pm 1}$ ($i\in I$), $h_{i,n}$ ($i\in I$, $n\in \Z\setminus\{0\}$) and central elements $c^{\pm 1/2}$, subject to the following relations:
\begin{align}
  k_ik_j = k_jk_i,\quad & k_ih_{j,n} =h_{j,n}k_i,\nn\\
  k_ix^\pm_{j,n}k_i^{-1} &= q^{\pm B_{ij}}x_{j,n}^{\pm},\nn\\
 \label{hxpm} [h_{i,n} , x_{j,m}^{\pm}] &= \pm \frac{1}{n} [n B_{ij}]_q c^{\mp
    {|n|/2}}x_{j,n+m}^{\pm},\\ 
x_{i,n+1}^{\pm}x_{j,m}^{\pm} -q^{\pm B_{ij}}x_{j,m}^{\pm}x_{i,n+1}^{\pm} &=q^{\pm
    B_{ij}}x_{i,n}^{\pm}x_{j,m+1}^{\pm}
  -x_{j,m+1}^{\pm}x_{i,n}^{\pm},\nn\\ [h_{i,n},h_{j,m}]
  &=\delta_{n,-m} \frac{1}{n} [n B_{ij}]_q \frac{c^n -
    c^{-n}}{q-q^{-1}},\nn\\ [x_{i,n}^+ , x_{j,m}^-]=\delta_{ij} & \frac{
    c^{(n-m)/2}\phi_{i,n+m}^+ - c^{-(n-m)/2} \phi_{i,n+m}^-}{q^{r_i} -
    q^{-r_i}},\nn\\
  \sum_{\pi\in\Sigma_s}\sum_{k=0}^s(-1)^k\left[\begin{array}{cc} s \nn\\
      k \end{array} \right]_{q^{r_i}} x_{i, n_{\pi(1)}}^{\pm}\ldots
  x_{i,n_{\pi(k)}}^{\pm} & x_{j,m}^{\pm} x_{i,
    n_{\pi(k+1)}}^{\pm}\ldots x_{i,n_{\pi(s)}}^{\pm} =0,\ \
  s=1-C_{ij},\nn
\end{align}
for all sequences of integers $n_1,\ldots,n_s$, and $i\ne j$, where $\Sigma_s$ is the symmetric group on $s$ letters, and $\phi_{i,n}^{\pm}$'s are determined by the formula
\begin{equation} \label{phidef} \phi_i^\pm(u) :=
  \sum_{n=0}^{\infty}\phi_{i,\pm n}^{\pm}u^{\pm n} = k_i^{\pm 1}
  \exp\left(\pm(q-q^{-1})\sum_{m=1}^{\infty}h_{i,\pm m} u^{\pm
      m}\right).
\end{equation}
There exist a coproduct, counit and antipode making $\uqgh$ into a Hopf algebra. 

The subalgebra of $\uqgh$ generated by $(k_i)_{i\in I}$, $(x^\pm_{i,0})_{i\in I}$ is a Hopf subalgebra of $\uqgh$ and is isomorphic as a Hopf algebra to $\uqg$, the quantized enveloping algebra of $\g$.
In this way, $\uqgh$-modules restrict to $\uqg$-modules. 


\subsection{Finite-dimensional representations and $q$-characters}\label{ssec:fdreps}
A representation $V$ of $\uqgh$ is \emph{of type $1$} if $c^{\pm 1/2}$ acts as the identity on $V$ and 
\be V = \bigoplus_{\lambda\in P} V_\lambda\,\,\,,\qquad\quad 
V_\lambda = \{ v \in V: k_i \on v = q^{\langle \alpha_i, \lambda \rangle} v\}.\label{gwts}\ee 
In what follows, all representations will be assumed to be of type 1 without further comment.
The decomposition (\ref{gwts}) of a finite-dimensional representation $V$ into its $\uqg$-weight spaces can be refined by decomposing it into the Jordan subspaces of the mutually commuting $\phi_{i,\pm r}^\pm$ defined in (\ref{phidef}),
\cite{FR}: 
\be V = \bigoplus_{\bs\gamma} V_{\bs\gamma}\,\,, \qquad \bs\gamma = (\gamma_{i,\pm r}^\pm)_{i\in I, r\in \Z_{\geq 0}}, \quad \gamma_{i,\pm r}^\pm \in \mathbb C\label{lwdecomp}\ee 
where \be V_{\bs\gamma} = \{ v \in V : \exists k \in \mathbb N, \,\, \forall i \in I, m\geq 0 \quad \left(
  \phi_{i,\pm m}^\pm - \gamma_{i,\pm m}^\pm\right)^k \on v = 0 \} \,. \nn\ee 
If $\dim (V_{\bs\gamma}) >0$, $\bs\gamma$ is called an \emph{$l$-weight} of $V$. For every finite-dimensional representation of $\uqgh$, the $l$-weights are known \cite{FR} to be of the form 
\be\nn \gamma_i^\pm(u) := \sum_{r =0}^\8 \gamma_{i,\pm r}^\pm u^{\pm r} 
 = q^{\deg Q_i - \deg R_i} \, \frac{Q_i(uq^{-1}) R_i(uq)}{Q_i(uq) R_i(uq^{-1})} \,,\ee 
where the right hand side is to be treated as a formal series in positive (resp. negative) integer powers of $u$, and $Q_i$ and $R_i$ are polynomials of the form 
\be\nn Q_i(u) = \prod_{a\in \Cx} \left( 1- ua\right)^{w_{i,a}}, \quad\quad 
       R_i(u) = \prod_{a\in \Cx} \left( 1- ua\right)^{x_{i,a}}, \ee 
for some $w_{i,a}, x_{i,a}\geq 0$, $i\in I,a\in \Cx$.  Let $\P$ denote the free abelian multiplicative group of  monomials in infinitely many formal variables $(Y_{i,a})_{i\in I,a\in \Cx}$. $\P$ is in bijection with the set of $l$-weights $\bs\gamma$ of the form above according to 
\be\label{lwdef} \bs\gamma=\bs\gamma(m) \quad\text{with}\quad m = \prod_{i \in I, a\in \Cx} Y_{i,a}^{w_{i,a}-x_{i,a}}.\ee
We identify elements of $\P$ with $l$-weights of finite-dimensional representations in this way, and henceforth write $V_m$ for $V_{\bs\gamma(m)}$.  Let $\Z\P=\Zys$ be the ring of Laurent polynomials in $(Y_{i,a})_{i\in I,a\in \Cx}$ with integer coefficients.
The $q$-character map $\chi_q$ \cite{FR} is then the injective homomorphism of rings 
\be \chi_q : \groth\uqgh \longrightarrow \Zys \nn\ee 
defined by 
\be \chi_q (V) = \sum_{m\in\P} \dim\left(V_m\right) m.\nn\ee 
Here $\groth\uqgh$ is the Grothendieck ring of finite-dimensional representations of $\uqgh$.

For any finite-dimensional representation $V$ of $\uqgh$, we let
\be \mchiq(V) := \left\{ m\in \P: m \text{ is a monomial of } \chi_q(V) \right\}.\nn\ee 

For each $j\in I$, a monomial $m = \prod_{i \in I, a\in \Cx} Y_{i,a}^{u_{i,a}}$ is said to be \emph{$j$-dominant} (resp. \emph{$j$-anti-dominant}) if and only if $u_{j,a}\geq 0$ (resp. $u_{j,a}\leq 0$) for all $a\in\Cx$. A monomial is (\emph{anti-})\emph{dominant} if and only if it is $i$-(anti-)dominant for all $i\in I$. Let $\Pp\subset \P$ denote the set of all dominant monomials.

If $V$ is a finite-dimensional representation of $\uqgh$ and $m\in\mchiq(V)$ is dominant, then a non-zero vector $\ket{m}\in V_m$ is called a \emph{highest $l$-weight vector}, with \emph{highest $l$-weight} $\bs \gamma(m)$, if and only if 
\be  \phi^\pm_{i,\pm t}\on\ket m= \ket m\gamma(m)^\pm_{i,\pm t} \quad \text{ and }\quad x^+_{i,r} \on \ket m = 0,\quad\text{for all } i\in I, r\in \Z, t\in \Z_{\geq 0} .\nn\ee 
A finite-dimensional representation $V$ of $\uqgh$ is said to be a \emph{highest $l$-weight representation} if $V= \uqgh \on \ket m$ for some highest $l$-weight vector $\ket m\in V$.

It is known \cite{CPbook,CP94} that for each $m\in\P^+$ there is a unique finite-dimensional irreducible representation, denoted $\L(m)$, of $\uqgh$ that is highest $l$-weight with highest $l$-weight $\bs\gamma(m)$, and moreover every finite-dimensional irreducible $\uqgh$-module is of this form for some $m\in \P^+$. 

Also for each $m\in \P^+$, there exists a highest $l$-weight representation $W(m)$, called the \emph{Weyl module}, with the property that every highest $l$-weight representation of $\uqgh$ with highest $l$-weight $\bs\gamma(m)$ is a quotient of $W(m)$ \cite{CPweyl}. 

A finite-dimensional $\uqgh$-module $V$ is said to be \emph{special} 
if and only if $\chi_q(V)$ has exactly one dominant monomial. It is \emph{anti-special} if and only if $\chi_q(V)$ has exactly one anti-dominant monomial. 
It is \emph{thin} 
if and only if no $l$-weight space of $V$ has dimension greater than 1. In other words, the module is thin if and only if the $(\phi_{i,\pm r}^\pm)_{i\in I,r\in \Z_{\geq 0}}$ are simultaneously diagonalizable with joint simple spectrum. A finite-dimensional $\uqgh$-module $V$ is said to be \emph{prime} if and only if it is not isomorphic to a tensor product of two non-trivial $\uqgh$-modules \cite{CPprime}.

Let $\chi:\groth\uqg \to \Z P$ be the $\uqg$-character homomorphism. Let $\wt:\P \to P$ be the homomorphism of abelian groups defined by $\wt: \YY i a \mapsto \omega_i$. Extend the map $\wt{}$ by linearity to $\Z\P\to\Z P$. Then the following diagram commutes \cite{FR}:
\be\begin{tikzpicture}    
\matrix (m) [matrix of math nodes, row sep=3em,    
column sep=4em, text height=2ex, text depth=1ex]    
{ \groth\uqgh &  \mathbb Z \P     \\    
\groth\uqg & \mathbb Z P     \\   };    
\path[->,font=\scriptsize]    
(m-1-1) edge node [above] {$\chi_q$} (m-1-2)    
(m-2-1) edge node [above] {$\chi$} (m-2-2)    
(m-1-1) edge node [left] {$\res$} (m-2-1)    
(m-1-2) edge node [right] {$\wt$} (m-2-2);    
\end{tikzpicture}\nn\ee    
where $\res:\groth\uqgh\to\groth\uqg$ is the restriction homomorphism.

Define $A_{i,a}\in \P$, $i\in I,a\in\Cx$, by 
\be\label{adef} A_{i,a} = Y_{i,aq^{r_i}} Y_{i,aq^{-r_i}} \prod_{C_{ji}=-1} Y_{j,a}^{-1} \prod_{C_{ji}=-2} Y_{j,aq}^{-1} Y_{j,aq^{-1}}^{-1} \prod_{C_{ji}=-3} Y_{j,aq^2}^{-1} Y_{j,a}^{-1} Y_{j,aq^{-2}}^{-1}.\ee 
Let $\Q$ be the subgroup of $\P$ generated by $A_{i,a}$, $i\in I,a\in\Cx$. Let $\Q^\pm$ be the monoid generated by $A_{i,a}^{\pm 1}$, $i\in I,a\in\Cx$. Note that $\wt A_{i,a} = \al_i$. There is a partial order $\leq$ on $\P$ in which $m\leq m'$ if and only if $m' m^{-1} \in \Q^+$. It is compatible with the partial order on $P$ in the sense that $m\leq m'$ implies $\wt m\leq \wt m'$.

We have \cite{FM} that for all $m_+\in \P^+$, 
\be \mchiq(\L(m_+)) \subset m_+ \Q^- \label{imchiq}.\ee

For all $i\in I, a\in \Cx$ let $u_{i,a}$
be the homomorphism of abelian groups $\P \to \mathbb Z$ such that
\be u_{i,a}(Y_{j,b})=\begin{cases} 1 & i=j \text{ and } a=b \\ 0 & \text{otherwise.}\end{cases}
\label{udef}\ee
Let 
$v$ be the homomorphisms of abelian groups $\Q \to \mathbb Z$ such that
\be 
  v(A_{j,b})=- 1 .\nn\ee
Note that the $(A_{i,a})_{i\in I,a\in \Cx}$ are algebraically independent, so 
$v$ is well-defined.

For each $j\in I$ we denote by $\uqsl j$ the copy of $\uqslp j$ generated by $c^{\pm1/2}, (x_{j,r}^\pm)_{r\in \Z}$, $(\phi_{j,\pm r}^\pm)_{r\in \Z_{\geq 0}}$. Let 
\be \beta_j : \mathbb Z \left[ Y^{\pm 1}_{i,a}\right]_{i\in I; a \in \Cx} \to \mathbb Z \left[ Y^{\pm 1}_{j,a}\right]_{a \in \Cx} \nn\ee     
be the ring homomorphism which  sends, for all $a\in\Cx$, $Y_{k,a}\mapsto 1$ for all $k \neq j$ and $Y_{j,a}\mapsto Y_{j,a}$.
For each $j\in I$, there exists \cite{FM} a ring homomorphism    
\be\tau_j : \Z\left[ Y^{\pm 1}_{i,a}\right]_{i\in I; a \in \Cx} \to \mathbb Z \left[ Y^{\pm 1}_{j,a}\right]_{a \in \Cx} \otimes \mathbb Z \left [ Z_{k, b}^{\pm 1}\right ]_{k \neq j; b \in \Cx},\nn\ee  
where $(Z_{k,b}^{\pm 1})_{k\neq j,b\in\Cx}$ are certain new formal variables, with the following properties:
\begin{enumerate}[i)] \item $\tau_j$ is injective. \item $\tau_j$ refines $\beta_j$ in the sense that $\beta_j$ is the composition of $\tau_j$ with the homomorphism $\mathbb Z [ Y^{\pm 1}_{j,a}]_{a \in \Cx} \otimes \mathbb Z [ Z_{k, b}^{\pm 1}]_{k \neq j; b \in \Cx}\to \mathbb Z [Y^{\pm 1}_{j,a}]_{a \in \Cx}$ which sends $Z_{k,b}\mapsto 1$ for all $k\neq j$, $b\in \Cx$.
\item In the diagram    
\be\label{tauj}\begin{tikzpicture}    
\matrix (m) [matrix of math nodes, row sep=3em,    
column sep=4em, text height=2ex, text depth=1ex]    
{     
\mathbb Z \left[ Y^{\pm 1}_{i,a}\right]_{i\in I; a \in \Cx}    
 & \mathbb Z \left[ Y^{\pm 1}_{j,a}\right]_{a \in \Cx} \otimes \mathbb Z \left [ Z_{k, b}^{\pm 1}\right ]_{k \neq j; b \in \Cx} \\    
\mathbb Z \left[ Y^{\pm 1}_{i,a}\right]_{i\in I; a \in \Cx}      
 & \mathbb Z \left[ Y^{\pm 1}_{j,a}\right]_{a \in \Cx} \otimes \mathbb Z \left [ Z_{k, b}^{\pm 1}\right ]_{k \neq j; b \in \Cx}\\    
};    
\path[->,font=\scriptsize,shorten <= 2mm,shorten >= 2mm]    
(m-1-1) edge node [above] {$\tau_j$} (m-1-2)    
(m-2-1) edge node [above] {$\tau_j$} (m-2-2)    
(m-1-1) edge (m-2-1)    
(m-1-2) edge (m-2-2);    
\end{tikzpicture}\ee    
let the right vertical arrow be multiplication by $\beta_j(A_{j,c}^{-1}) \otimes 1$; then the diagram commutes if and only if the left vertical arrow is multiplication by $A^{-1}_{j,c}$.
\end{enumerate}

\subsection{Affinizations of $\uqg$-modules} \label{sec:minaff} 
For $\mu\in P^+$, let $V(\mu)$ be the (unique up to isomorphism) simple $\uqg$-module with highest weight $\mu$.
$\L(m)$, $m\in \P^+$, is said to be an \emph{affinization} of $V(\mu)$ if $\wt m=\mu$ \cite{Cminaffrank2}.
Two affinizations are said to be equivalent if they are isomorphic as $\uqg$-modules. Let $[[\L(m)]]$ denote the equivalence class of $\L(m)$, $m\in \P^+$. For each $\lambda\in P^+$ define $\ma \lambda := \left\{[[\L(m)]] : \wt m = \lambda \right\}$, the set of equivalence classes of affinizations of $V(\lambda)$.
Any finite-dimensional $\uqg$-module $V$ is isomorphic to a direct sum of finite-dimensional simple $\uqg$-modules; for each $\lambda\in P^+$ let $[V:V(\lambda)]$ denote the multiplicity of $V(\lambda)$ in $V$. There is a partial order $\leq$ on equivalence classes of affinizations in which $[[\L(m)]] \leq [[\L(m')]]$ if and only if for all $\nu\in P^+$ either
\begin{enumerate}[(i)]
\item $[\L(m):V(\nu)] \leq [\L(m'):V(\nu)]$, or
\item there exists a $\mu\in P^+$, $\mu\geq \nu$, such that $[\L(m):V(\mu)]<[\L(m'):V(\mu)]$.
\end{enumerate}
For all $\lambda\in P^+$, $\ma\lambda$ is a finite poset \cite{Cminaffrank2}.
A \emph{minimal affinization} of $V(\lambda)$, $\lambda\in P^+$, is a minimal element of $\ma \lambda$ with respect to the partial ordering \cite{Cminaffrank2}.
Thus a minimial affinization is by definition an equivalence class of $\uqgh$-modules; but we shall refer also to the elements of such a class -- the modules themselves -- as minimal affinizations.

\section{Thin special $q$-characters}\label{sec:thinspecial}
\subsection{Thin simple $\uqslt$ modules} \label{boxpic}
Theorem \ref{thmA} below gives conditions which, if satisfied, guarantee that a given set of monomials $\mc M$ is the $q$-character of a thin special simple finite-dimensional $\uqgh$-module.  In the proof we shall need the following lemma concerning thin special simple finite-dimensional $\uqslt$-modules.

\begin{lem}\label{thinsimplesl2lem}
Let $\g=\mf{sl}_2$. Let $m\in \P$. There exists $M\in \P^+$ such that $\L(M)$ is thin and $m\in \mchiq(\L(M))$ if and only if,  for all $a\in \Cx$,  $|u_{1,a}(m)| \leq 1$  and $u_{1,a}(m) - u_{1,aq^2}(m) \neq 2$. If such an $M$ exists it is unique and exactly one of the following holds for all  $a\in \Cx$.
\begin{align*}
&(i)  &u_{1,a}(m) &=  1, &u_{1,aq^{2}}(m) &= 0, \qquad\text{ and }\, &mA_{1,aq}^{- 1}&\in\mchiq(\L(M))\\
&(ii)  &u_{1,a}(m) &=  1, &u_{1,aq^{2}}(m) &= 1, \qquad\text{ and }\, &mA_{1,aq}^{- 1}&\in\mchiq(W(M)) \setminus \mchiq(\L(M))\\ 
&(iii)  &u_{1,a}(m) &=  -1, &u_{1,aq^{-2}}(m) &= 0, \qquad\text{ and }\, &mA_{1,aq^{-1}}&\in\mchiq(\L(M))\\
&(iv)  &u_{1,a}(m) &=  -1, &u_{1,aq^{-2}}(m) &= -1, \,\quad\text{ and }\, &mA_{1,aq^{-1}}&\in\mchiq(W(M)) 
\setminus \mchiq(\L(M))\\
&(v)  &u_{1,a}(m)&=0, \quad\text{ and } & mA_{1,aq}^{-1}&\notin \mchiq(W(M)) &\text{ and }\,\,\, m A_{1,aq^{-1}} &\notin \mchiq(W(M)).  
\end{align*} 
\end{lem}
\begin{proof}
This follows from the known closed forms for the $q$-characters of all simple finite-dimensional modules, and all Weyl modules, of $\uqslt$. Let $a_k\in \Cx$, $1\leq k\leq K\in \Z_{>0}$, and consider $M=\prod_{k=1}^KY_{1,a_k}\in\P^+$. The Weyl module $W(M)$ is isomorphic to a tensor product of fundamental representations \cite{CPweyl}, and 
$\chi_q\left(W(M)\right) = \prod_{k=1}^K( Y_{1,a_k} + Y_{1,a_kq^2}^{-1})$. The simple module $\L(M)$ is isomorphic to a tensor product of evaluation modules \cite{CPsl2}, as follows. First recall that the set $S_{r}(a) := \left \{a q^{-r+1}, aq^{-r+3}, \dots, aq^{r-1} \right \}$ is called the \emph{$q$-string} of length $r\in \Z_{\geq 0}$ centred on $a\in\Cx$, and that two $q$-strings are said to be in \emph{general position} if one contains the other or their set-theoretic union is not a $q$-string. Let $m_{a}^{(r)}:= Y_{1,aq^{-r+1}} Y_{1,aq^{-r+3}} \dots Y_{1,aq^{r-1}}$. There is a unique multiset of $q$-strings in pairwise general position, say $\{S_{r_t}(b_t): 1\leq t\leq T\}$, such that $M = \prod_{t=1}^T m_{b_t}^{(r_t)}$. Then $\chi_q(\L(M))= \prod_{t=1}^T \chi_q(\L(m_{b_t}^{(r_t)}))$, where  $\L(m_a^{(r)})$  is an evaluation module and
\be \chi_q(\L(m_a^{(r)})) = \left(Y_{1,aq^{-r+1}} Y_{1,aq^{-r+3}} \dots Y_{1,aq^{r-1}}\right)      
                      \left(1+ \sum_{t=0}^{r-1}     
             A^{-1}_{1,aq^r} A^{-1}_{1,aq^{r-2}} \dots A^{-1}_{1,aq^{r-2t}}\right).\ee      
Note that \emph{thin} simple finite-dimensional $\uqslt$-modules are precisely those whose $q$-strings are, in addition to being pairwise in general position, also pairwise disjoint. 
\end{proof}
\begin{cor}\label{tscor}
 A simple finite-dimensional $\uqslt$-module is thin if and only if it is special.\qed
\end{cor}

It is helpful to picture the set of monomials $\mchiq(\L(M))$ of a thin $\uqslt$-module $\chi_q(\L(M))$ as follows. Suppose there is an $a\in \Cx$ such that $\mchiq(\L(M))\subset \Z[Y^{\pm 1}_{1,aq^{2k}}]_{k\in \Z}$. 
Consider a one-dimensional array of boxes, indexed by the integers, in which each box can be in one of three states: it can be empty ($\eb$), it can contain one black ball ($\bb$), or it can contain one white ball ($\wb$). 
Given a monomial $m\in \mchiq(\L(M))$ we place a black  ball at each site $k$ such that $u_{1,aq^{2k}}(m)=1$, a white ball at each site $k$ such that $u_{1,aq^{2k}}(m)=-1$, and leave the rest empty, to reach some configuration, for example
\be\nn\dots \eb\eb\bb\eb\eb\eb\eb\bb\bb\eb\eb\wb\eb\eb \dots.\ee 
By Lemma \ref{thinsimplesl2lem}, the pattern $\bb\wb$ never occurs, and  
we can generate all other monomials of $\chi_q(\L(M))$ by repeatedly performing moves of the following two types:
\begin{enumerate}[(a)]
\item Remove a black ball from site $k$ and fill the empty site $k+1$ with a white ball: $\bb\eb \to  \eb\wb$.
\item Remove a white ball from site $k+1$ and fill the empty site $k$ with a black ball: $ \eb\wb \to \bb\eb$. 
\end{enumerate}
In the present paper we identify, in types A and B, a class of representations, which we shall call \emph{snake modules}, \S\ref{sec:snakes}, for which a generalization of this  ball-and-box realization of monomials exists.

\subsection{Effect of $x_{i,r}^\pm$ on $l$-weight states}
In the proof of Theorem \ref{thmA} we shall also need the following proposition, which is a quantum-affine analog of the familiar statement in the representation theory of $\g$, or of $\uqg$, that if $\ket\omega$ is a vector of weight $\omega\in P$ then $x_i^\pm\on\ket\omega$ is either zero or has weight $\omega\pm\alpha_i$. 
\begin{prop}\label{Aprop}
Let $V$ be a finite-dimensional $\uqgh$-module and $m,m'\in\mchiq(V)$. Let $\ket m\in \ker\left(\phi^\pm_i(u)-\gamma(m)^\pm_i(u)\right)\subset V_m$ for all $i\in I$. Then, for all $j\in I$, at least one of the following holds: 
\begin{enumerate}
\item there is an $a\in\Cx$ such that $m'=mA_{j,a}$ (resp. $m'=mA_{j,a}^{-1}$), or
\item for all $r\in\Z$, when $x^+_{j,r}\on \ket m$ (resp. $x^-_{j,r}\on \ket m$) is decomposed into $l$-weight spaces, \confer (\ref{lwdecomp}), its component in  $V_{m'}$ is zero.
\end{enumerate}
\end{prop}
\begin{proof}
The defining relations between $x^{\pm}_{j,r}$ and $h_{i,r}$, $i\in I$, (\ref{hxpm}) are equivalent to
\bea    
(1- q^{\pm B_{ij}}uv)\phi^+_i(u)\,  x^\pm_j(v) &=&  \, (q^{\pm B_{ij}}- uv) \, x^\pm_j(v) \, \phi^+_i(u) , \nn\\    
(1/(uv)- q^{\pm B_{ij}})\phi^-_i(u)\,  x^\pm_j(v) &=&  \, (q^{\pm B_{ij}}/(uv)-1) \, x^\pm_j(v) \, \phi^-_i(u) \nn\eea    
where $\phi_i^\pm(u)$ is as in (\ref{phidef}) and $x^\pm_j(v)$ is the formal Laurent series $x^\pm_j(v) := \sum_{n \in \mathbb Z} x^{\pm}_{j,n} v^{-n}$.    
Consider for definiteness the case of $x^+_{j}(v)$ ($x^-_j(v)$ is similar).
Let $(\ket{m',k})_{1\leq k\leq \dim V_{m'}}$ be a basis of the $l$-weight space $V_{m'}$. We have 
\be x^+_j(v)\on \ket m = \sum_{k=1}^{\dim V_{m'}}\ket{m',k} \lambda_{k}(v)+\dots ,\nn\ee
where ``$\dots$'' indicates terms in other $l$-weight spaces and, for each $k$, $\lambda_{k}(v)$ is a formal Laurent series in $v$ with complex coefficients. We can suppose that the basis $(\ket{m',k})_{1\leq k\leq \dim V_{m'}}$ has been chosen so that the action of the $\phi^\pm_{i,\pm r}$ (mutually commuting on $V$)  is lower triangular, i.e. for all $(i,r)\in I\times \Z_{\geq 0}$ and $1\leq k\leq \dim V_{m'}$, \be(\phi_{i,\pm r}^\pm-\gamma(m')_{i,\pm r}^\pm)\on\ket{m',k} \in \Span_{k'> k} \{\ket{m',k'}\}.\label{lowertriang}\ee 

Suppose that statement (2) of the proposition is false. Then there is a $k$ with $1\leq k\leq \dim V_{m'}$ such that $\lambda_{k}(v) \neq 0$. By considering the smallest such $k$ we may assume that for all $1\leq \ell<k$, $\lambda_{\ell}(v)= 0$. Now, since $\ket m\in \ker\left(\phi^+_i(u)-\gamma(m)^+_i(u)\right)$, we have
\bea 0&=& (q^{B_{ij}}-uv) x^+_j(v) \left(\phi^+_i(u) - \gamma(m)^+_i(u) \right) \on \ket m \nn\\
&=& 
\left( (1-q^{B_{ij}}uv)\phi^+_i(u) -  (q^{B_{ij}}-uv) \gamma(m)^+_i(u) \right) x^+_j(v) \on \ket m .\nn\eea
By taking the $\ket{m',k}$ component of this expression and using (\ref{lowertriang}), we have
\be 0= \left( (1-q^{B_{ij}}uv)\gamma^+_i(m')(u) -  (q^{B_{ij}}-uv) \gamma^+_i(m)(u) \right) \lambda_{k}(v) 
\label{Xeqn}.\ee
This must hold for all $i\in I$. For each $i\in I$, (\ref{Xeqn}) is an equation of the form $0=\lambda_k(v) \sum_{n=0}^\8 u^n (b^{(i)}_n + c^{(i)}_nv)$  for the formal Laurent series $\lambda_k(v)$,  with $b^{(i)}_n,c^{(i)}_n\in \C$ for all $n\in\Z_{\geq 0}$. Equivalently, for each $i\in I$, it is a countably infinite set of first order recurrence relations on the series coefficients of $\lambda_k(v)$. There are non-zero solutions if and only if there is an $a\in\Cx$ such that $b^{(i)}_n/c^{(i)}_n=-a$ for all $n\in \Z_{\geq 0}$ and all $i\in I$. That is,
\be \gamma(m')^+_i(u) \left(\gamma(m)^+_i(u)\right)^{-1} = q^{B_{ij}} \frac{1-q^{-B_{ij}} u a}{1-q^{B_{ij}}u a} \equiv \gamma(A_{j,a})^+_i(u)\nn\ee
as an equality of power series in $u$, \confer (\ref{adef}) and (\ref{lwdef}). Similar reasoning concerning $\phi^-$ yields the analogous statement for $\gamma(A_{j,a})^-_i(u)$. So we have (1), as required. \end{proof}

\subsection{Criteria for a correct thin special $q$-character}
We can now state the first main result of the paper.
\begin{thm}\label{thmA}
Let $m_+\in \P^+$. Suppose that $\mc M\subset \P$ is a finite set of distinct monomials such that:
\begin{enumerate}[(i)]
\item $\{m_+\} = \P^+ \cap \mc M$; \label{monlydom} 
\item For all $m\in \mc M$ and all $(i,a)\in I \times \Cx$, if $mA_{i,a}^{-1}\notin\mc M$ then $mA_{i,a}^{-1} A_{j,b}\notin \mc M$ unless $(j,b)=(i,a)$;\label{onewayback}
\item For all $m\in \mc M$ and all $i\in I$ there exists
$M\in \mc M$, $i$-dominant, such that
\be\nn \chi_q(\L(\beta_i(M))) = \sum_{m'\in m \Z[A_{i,a}^{\pm 1}]_{a\in \Cx}\cap \mc M} \beta_i(m') .\ee\label{inthinsimple}
\end{enumerate}
Then
\be \label{snch} \chi_q(\L(m_+)) = \sum_{m\in \mc M} m \ee 
and in particular $\L(m_+)$ is thin and special.
\end{thm}
\begin{proof}
We begin by establishing that (\ref{monlydom}) and (\ref{inthinsimple}) together imply that 
\be \mc M \subset m_+ \Q^- \label{incone}.\ee
Indeed, Property (\ref{monlydom}) states that for every monomial $m\neq m_+$ in $\mc M$ there is an $i\in I$ such that $m$ is not $i$-dominant. So there is an $a\in \Cx$ such that $u_{i,aq^{r_i}}(m)<0$ and (by choice of $a$) $u_{i,aq^{-r_i}}(m)\geq 0$. Property (\ref{inthinsimple}) states that $m$ is a monomial of the $q$-character of a thin simple finite-dimensional $\uqsl i$-module. By Lemma \ref{thinsimplesl2lem}, it must therefore be that $u_{i,aq^{r_i}}(m')=-1$, $u_{i,aq^{-r_i}}(m') = 0$, and $mA_{i,a}\in \mc M$. The relation (\ref{incone}) follows by recursion since $\mc M$ is assumed finite. 

We now prove  (\ref{snch}). 
For each $n\in \Z_{\geq 0}$ let $\Z[A^{-1}_{i,a}]_{(i,a)\in I\times\Cx}^{(> n)}$ denote the ring (without identity) of polynomials in the variables $\{A_{i,a}^{-1}:(i,a)\in I \times \Cx\}$ whose monomials all have degree strictly greater than $n$.
To prove (\ref{snch}) it is sufficient to establish the following claim for all $n\in\Z_{\geq 0}$.

\emph{Claim:} The equality (\ref{snch}) holds modulo $m_+ \Z[A^{-1}_{i,a}]_{(i,a)\in I\times \Cx}^{(> n)}$.

We proceed by induction on $n$. The claim is true for $n=0$ by (\ref{imchiq}) and (\ref{incone}). 
So for the inductive step, let $n\in \Z_{>0}$ and suppose the claim is true for $n-1$. 

It follows from Proposition \ref{Aprop} that all monomials $m'$ in $\chi_q(\L (m_+))$ such that $v(m'm_+^{-1})=n$ are of the form $m A_{i,a}^{-1}$ for some $m\in\mchiq(\L (m_+))$ with $v(mm_+^{-1})=n-1$ and some $(i,a)\in I\times \Cx$. For suppose not: then by Proposition \ref{Aprop}, $\L(m_+)_{m'}$ contains a highest $l$-weight vector and so generates a proper submodule -- a contradiction. Therefore by the inductive hypothesis all monomials  $m'$ in $\chi_q(\L (m_+))$ such that $v(m'm_+^{-1})=n$ are of the form $m A_{i,a}^{-1}$ for some $m\in\mc M$ with $v(mm_+^{-1})=n-1$ and some $(i,a)\in I\times \Cx$.
As we noted above, every monomial $m'$ in $\mc M$ such that $v(m'm_+^{-1})=n$ is also of the form $mA_{i,a}^{-1}$ for some $m\in \mc M$ with $v(mm_+^{-1})=n-1$ and some $(i,a)\in I\times \Cx$. 

Consequently, it is enough to consider monomials of the form  $mA_{i,a}^{-1}$ for some $m\in \mc M$  with $v(mm_+^{-1})=n-1$ and some $(i,a)\in I\times \Cx$.
Let $M$ be the unique $i$-dominant monomial in $m\Z[A_{i,a}]_{a\in \Cx}\cap \mc M$. Property (\ref{inthinsimple}) asserts that such an $M$ exists (possibly $m=M$) and  moreover that $\chi_q(\L(\beta_i(M))) = \sum_{m'\in m \Z[A_{i,a}^{\pm 1}]_{a\in \Cx}\cap \mc M} \beta_i(m')$, so that the simple $\uqsl i$-module $\L(\beta_i (M))$ is thin. Since $\beta_i(m)$ is a monomial of a thin simple $\uqsl i$-module, Lemma \ref{thinsimplesl2lem} implies that exactly one of the following three cases applies:
\begin{enumerate}[(I)]
\item $u_{i,aq^{-r_i}}(m) = 1$,  $u_{i,aq^{r_i}}(m) = 0$, and $\beta_i(mA_{i,a}^{-1})\in \mchiq(\L(\beta_i(M)))$.
\item $u_{i,aq^{-r_i}}(m) = 1$,  $u_{i,aq^{r_i}}(m) = 1$, and $\beta_i(mA_{i,a}^{-1})\in\mchiq(W(\beta_i(M))) \setminus \mchiq(\L(\beta_i(M)))$.
\item $u_{i,aq^{-r_i}}(m) \leq 0$ and $\beta_i(mA_{i,a}^{-1})\notin\mchiq(W(\beta_i(M)))$.
\end{enumerate}
We shall now complete the inductive step by showing that in case (I), $mA_{i,a}^{-1}$ appears, with coefficient exactly 1, on both sides of (\ref{snch}), while in cases (II) and (III), $mA_{i,a}^{-1}$ does not appear on either side of (\ref{snch}). 

First observe that 
\be\nn mA_{i,a}^{-1} \Z[A_{i,b}]^{(>0)}_{b\in \Cx} \cap \mchiq(\L(m_+)) = mA_{i,a}^{-1} \Z[A_{i,b}]^{(>0)}_{b\in \Cx} \cap \mc M,\ee 
by the inductive assumption, and that $M$ is the unique $i$-dominant monomial in this set.

Now consider case (I). By Property (\ref{inthinsimple}), $mA_{i,a}^{-1}\in\mc M$ (note that $\beta_i$ is injective when restricted to $m\Z[A_{i,a}^{\pm 1}]_{(i,a)\in I\times\Cx}$). By the injectivity and property (\ref{tauj})  of the homomorphism $\tau_i$ we deduce that $m A_{i,a}^{-1}$ must be a monomial of $\chi_q(\L (m_+))$. The coefficient of $\beta_i(mA_{i,a}^{-1})$ in $\chi_q(\L(\beta_i(M)))$ is 1, so therefore $mA_{i,a}^{-1}$ must have coefficient exactly 1 in $\chi_q(\L(m_+))$, for if it appeared with a larger coefficient it would not be part of a consistent $\uqsl i$-character.

Now consider case (II). By Property (\ref{inthinsimple}), $mA_{i,a}^{-1}\notin \mc M$. Hence by Property (\ref{onewayback}), $mA_{i,a}^{-1} A_{j,b}\notin \mc M$ unless $(j,b)=(i,a)$. If $mA_{i,a}^{-1} A_{j,b}$ is not in $m_+ \Q^-$ then it is not in $\mchiq(\L(m_+))$ by (\ref{imchiq}). If $mA_{i,a}^{-1} A_{j,b}$ is in $m_+\Q^-$ then $v(mA_{i,a}^{-1} A_{j,b}m_+^{-1})=n-1$ and we can use the inductive assumption. Therefore 
\be\label{nin}mA_{i,a}^{-1} A_{j,b}\notin \mchiq(\L(m_+))\quad\text{ unless }\quad  (j,b) = (i,a).\ee  
Now suppose, for a contradiction, that $m':=mA_{i,a}^{-1}\in\mchiq(\L(m_+))$.  Then we can pick a non-zero $\ket{m'}\in \ker(\phi_{i}^\pm(u)-\gamma^\pm_i(m')(u))\subseteq \L (m_+)_{m'}$. By Proposition \ref{Aprop}, for all $r\in \Z$, $x^+_{j,r} \on \ket{m'} = 0$ for all $j\neq i$, and $x^+_{i,r}\on\ket{m'}\in (L m_+)_m$. If $x^+_{i,r}\on \ket {m'} =0$ for all $r\in \Z$ then $\ket {m'}$ generates a proper submodule in $\L (m_+)$: a contradiction since $\L (m_+)$ is simple. So for some $r\in \Z$, $x^+_{i,r} \on \ket{m'}$ is non-zero and spans the one-dimensional (by the inductive assumption) $l$-weight space $\L (m_+)_m$. 
Therefore $\ket{m'}\notin \Span_{r\in \Z} x_{i,r}^- (\L (m_+)_m)$, because $\L(\beta_i(M))$ is by definition irreducible, $\beta_i(m)$ appears in its $q$-character, and $\beta_i(m')$ does not. Now, if $m''\in \mchiq(\L(m_+))$ and $j\in I$ are such that $\ket{m'}\in \Span_{r\in \Z} x_{j,r}^{-}\on \L(m_+)_{m''}$ then $\wt(m'') = \wt(m')+\alpha_j$ and hence $v(m'' m_+^{-1})=n-1$. So, by the inductive assumption, $\dim(\L(m_+)_{m''})=1$, and thus Proposition \ref{Aprop} applies to any $\ket{m''}\in \L(m_+)_{m''}$. Hence, by (\ref{nin}),  $x_{j,r}^-\ket{m''}$ has zero component in $\L(m_+)_{m'}$ for all $m''\neq m$. So $\ket{m'}\notin \Span_{i\in I,r\in \Z} x_{i,r}^- (\L (m_+))$: a contradiction. 

Finally consider case (III). By Property (\ref{inthinsimple}), $mA_{i,a}^{-1}\notin\mc M$. Now $mA_{i,a}^{-1}$ is not $i$-dominant since $u_{i,aq^{-r_i}}(mA_{i,a}^{-1}) =-1$. But $\beta_i(mA_{i,a}^{-1})$ is not in the $q$-character of the Weyl module $W(\beta_i(M))$, and, recall, $M$ is the unique $i$-dominant monomial in $mA_{i,a}^{-1} \Z[A_{i,b}]_{b\in \Cx}\cap \mc M$. Therefore $mA_{i,a}^{-1}$ cannot appear in $\chi_q(\L(m_+))$ because it is not part of a consistent $\uqsl i$-character.

This completes the inductive step, and we have therefore established the claim above, for all $n\in\Z_{\geq 0}$, and hence the equality (\ref{snch}). Finally, it follows that $\L(m_+)$ is manifestly thin, and it is special by Property (\ref{monlydom}).
\end{proof}

It is very plausible that the conditions of Theorem \ref{thmA} could be weakened: for example, we suspect that (\ref{onewayback}) could be dropped, and that (iii) could be replaced by the demand that $m \Z[A_{i,a}^{\pm 1}]_{a\in \Cx}\cap \mc M$ should be a consistent thin $\uqsl i$-character.

Let us stress that Theorem \ref{thmA} requires no \emph{a priori} knowledge about the representation $\L(m_+)$. One way to construct the set $\mc M$ of monomials is to use the $q$-character algorithm of \cite{FM}. Starting with any dominant monomial $m_+$, the algorithm  -- provided it does not fail, in the sense explained in \cite{FM} -- produces a Laurent polynomial, $\chi_{FM}(m_+)\in \Z\P$. It is then a finite computational check to determine whether the monomials of $\chi_{FM}(m_+)$ satisfy the conditions of Theorem \ref{thmA}. If they do then 
\be\chi_{FM}(m_+)= \chi_q(\L(m_+)).\nn\ee
It they do not, then $m_+$ does not fall within the scope of Theorem \ref{thmA}.

\begin{figure} \caption{\label{snakeposfig} Each of the points $\scriptscriptstyle\square$ (resp. $\scriptscriptstyle\blacksquare$) is in snake position (resp. minimal snake position) to the point marked $\circ$.}
\be\nn\begin{tikzpicture}[scale=.5,yscale=-1]
\draw[help lines] (0,0) grid (5,6);
\foreach \y in {1,2,3,4,5,6,7} {\node at (-1,\y-1) {$\scriptstyle\y$};}
\draw (1,-1) -- (4,-1);
\foreach \x in {1,2,3,4} {\filldraw[fill=white] (\x,-1) circle (2mm) node[above=1mm] {$\scriptstyle\x$}; }
\draw[thick] (2,0) circle (1.5mm);
\foreach \x/\y in {1/3,2/2,3/3,4/4} 
{\node[regular polygon, regular polygon sides=4,draw,fill=black,inner sep=.3mm] at (\x,\y) {};}
\foreach \x/\y in {1/5,2/4,2/6,3/5,4/6} 
{\node[regular polygon, regular polygon sides=4,draw,fill=white,inner sep=.3mm] at (\x,\y) {};}
\end{tikzpicture}
\nn\ee
\be
\begin{tikzpicture}[scale=.35,yscale=-1]
\draw[help lines] (0,0) grid (14,18);
\draw (2,-1) -- (6,-1); \draw (8,-1) -- (12,-1);
\draw[double,->] (6,-1) -- (6.8,-1); \draw[double,->] (8,-1) -- (7.2,-1);
\filldraw[fill=white] (7,-1) circle (2mm) node[above=1mm] {$\scriptstyle 4$};
\foreach \x in {1,2,3} {
\filldraw[fill=white] (2*\x,-1) circle (2mm) node[above=1mm] {$\scriptstyle\x$}; 
\filldraw[fill=white] (2*7-2*\x,-1) circle (2mm) node[above=1mm] {$\scriptstyle\x$}; }
\foreach \y in {0,2,4,6,8,10,12,14,16,18} {\node at (-1,\y) {$\scriptstyle\y$};}
\draw[thick] (4,2) circle (2mm);
\foreach \x/\y in {2/8,4/6,4/10,6/8,6/12,7/9,7/13,2/12,2/16,4/14,4/18,6/16,7/17} 
{\node[regular polygon, regular polygon sides=4,draw,fill=white,inner sep=.3mm] at (\x,\y) {};}
\foreach \x/\y in {2/8,4/6,6/8,7/9} 
{\node[regular polygon, regular polygon sides=4,draw,fill=black,inner sep=.3mm] at (\x,\y) {};}
\end{tikzpicture}
\begin{tikzpicture}[scale=.35,yscale=-1]
\draw[help lines] (0,0) grid (14,18);
\draw (2,-1) -- (6,-1); \draw (8,-1) -- (12,-1);
\draw[double,->] (6,-1) -- (6.8,-1); \draw[double,->] (8,-1) -- (7.2,-1);
\filldraw[fill=white] (7,-1) circle (2mm) node[above=1mm] {$\scriptstyle 4$};
\foreach \x in {1,2,3} {
\filldraw[fill=white] (2*\x,-1) circle (2mm) node[above=1mm] {$\scriptstyle\x$}; 
\filldraw[fill=white] (2*7-2*\x,-1) circle (2mm) node[above=1mm] {$\scriptstyle\x$}; }
\foreach \y in {0,2,4,6,8,10,12,14,16,18} {\node at (-1,\y) {$\scriptstyle\y$};}
\begin{scope}[xscale=-1,xshift=-14cm,yshift=-2cm]
\draw[thick] (4,2) circle (2mm);
\foreach \x/\y in {2/8,4/6,4/10,6/8,6/12,7/9,7/13,2/12,2/16,4/14,4/18,6/16,7/17,6/20,2/20} 
{\node[regular polygon, regular polygon sides=4,draw,fill=white,inner sep=.3mm] at (\x,\y) {};}
\foreach \x/\y in {2/8,4/6,6/8,7/9} 
{\node[regular polygon, regular polygon sides=4,draw,fill=black,inner sep=.3mm] at (\x,\y) {};}
\end{scope}
\end{tikzpicture}
\nn\ee\be\begin{tikzpicture}[scale=.35,yscale=-1]
\draw[help lines] (0,0) grid (14,18);
\draw (2,-1) -- (6,-1); \draw (8,-1) -- (12,-1);
\draw[double,->] (6,-1) -- (6.8,-1); \draw[double,->] (8,-1) -- (7.2,-1);
\filldraw[fill=white] (7,-1) circle (2mm) node[above=1mm] {$\scriptstyle 4$};
\foreach \x in {1,2,3} {
\filldraw[fill=white] (2*\x,-1) circle (2mm) node[above=1mm] {$\scriptstyle\x$}; 
\filldraw[fill=white] (2*7-2*\x,-1) circle (2mm) node[above=1mm] {$\scriptstyle\x$}; }
\foreach \y in {0,2,4,6,8,10,12,14,16,18} {\node at (-1,\y) {$\scriptstyle\y$};}
\begin{scope}[yshift=1cm]
\draw[thick] (7,0) circle (2mm);
\foreach \x/\y in {7/2,7/6,7/10,7/14,8/5,8/9,8/13,10/7,10/11,12/9,8/17,10/15,12/13,12/17} 
{\node[regular polygon, regular polygon sides=4,draw,fill=white,inner sep=.3mm] at (\x,\y) {};}
\foreach \x/\y in {7/2,8/5,10/7,12/9} 
{\node[regular polygon, regular polygon sides=4,draw,fill=black,inner sep=.3mm] at (\x,\y) {};}
\end{scope}
\end{tikzpicture}
\begin{tikzpicture}[scale=.35,yscale=-1]
\draw[help lines] (0,0) grid (14,18);
\draw (2,-1) -- (6,-1); \draw (8,-1) -- (12,-1);
\draw[double,->] (6,-1) -- (6.8,-1); \draw[double,->] (8,-1) -- (7.2,-1);
\filldraw[fill=white] (7,-1) circle (2mm) node[above=1mm] {$\scriptstyle 4$};
\foreach \x in {1,2,3} {
\filldraw[fill=white] (2*\x,-1) circle (2mm) node[above=1mm] {$\scriptstyle\x$}; 
\filldraw[fill=white] (2*7-2*\x,-1) circle (2mm) node[above=1mm] {$\scriptstyle\x$}; }
\foreach \y in {0,2,4,6,8,10,12,14,16,18} {\node at (-1,\y) {$\scriptstyle\y$};}
\begin{scope}[yshift=3cm,xscale=-1,xshift=-14cm]
\draw[thick] (7,0) circle (2mm);
\foreach \x/\y in {7/2,7/6,7/10,7/14,8/5,8/9,8/13,10/7,10/11,12/9,10/15,12/13} 
{\node[regular polygon, regular polygon sides=4,draw,fill=white,inner sep=.3mm] at (\x,\y) {};}
\foreach \x/\y in {7/2,8/5,10/7,12/9} 
{\node[regular polygon, regular polygon sides=4,draw,fill=black,inner sep=.3mm] at (\x,\y) {};}
\end{scope}
\end{tikzpicture}
\nn\ee
\end{figure}
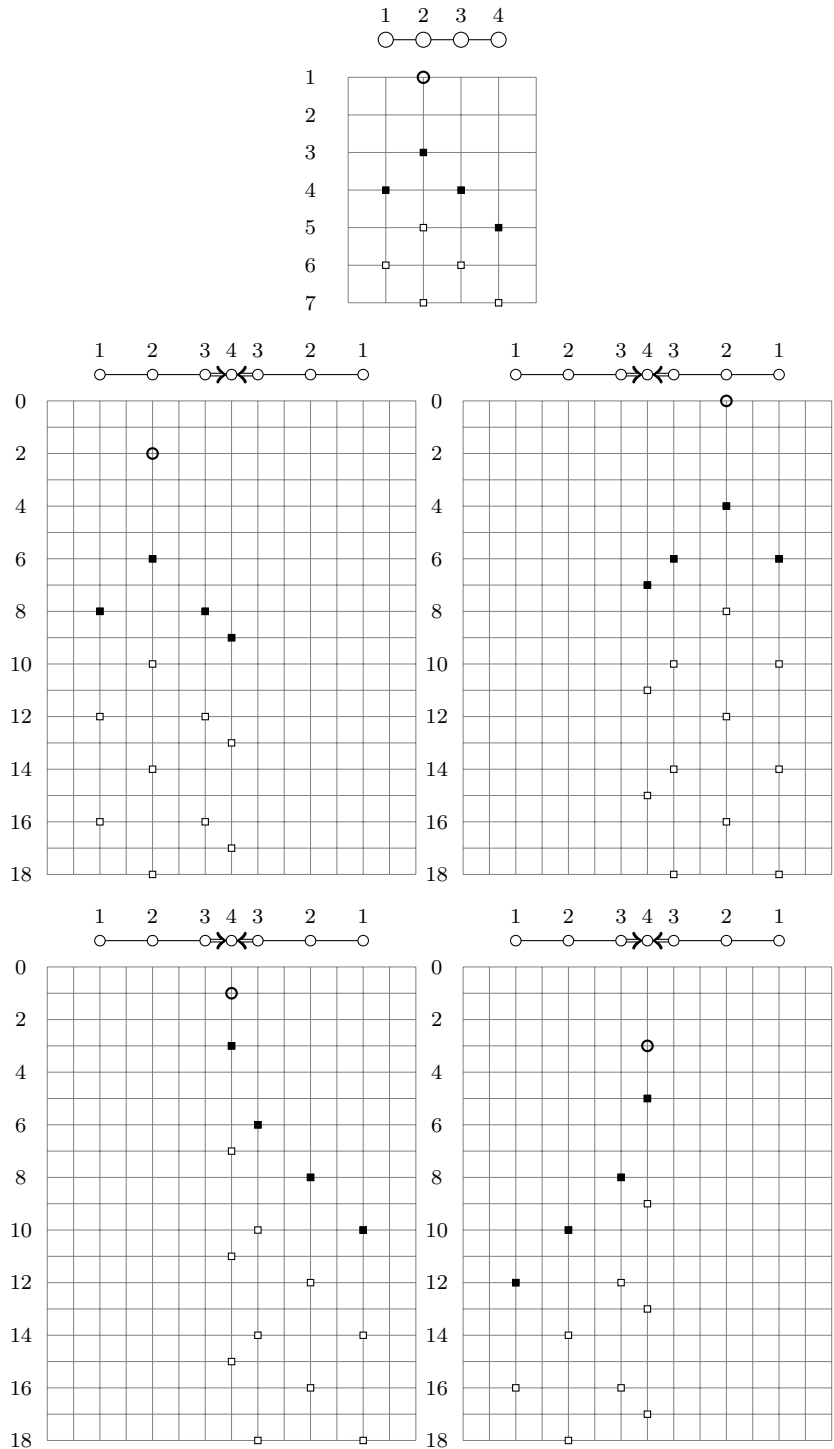

\section{Snake modules in types A and B}\label{sec:snakes}
In this section we introduce the class of modules to which the $q$-character formula of Theorem \ref{snakechar} below will pertain.
We specialize to types A and B: henceforth, $\g$ is either $\mf a_N$ or $\mf b_N$. 

\subsection{Notation, and the subring $\Z[Y^{\pm 1}_{i,k}]_{(i,k)\in\It}$} 
We define a subset $\It\subset I\times \Z$ as follows. 
\begin{enumerate}[Type A:]
\item Let $\It := \{ (i,k) \in I \times \Z : i-k\equiv 1 \mod 2 \}$.
\item Let $\It := \{(N,2k+1):k\in \Z\} \sqcup \{(i,k) \in I\times \Z : i<N \text{ and } k\equiv 0 \mod 2\}$.
\end{enumerate}
For the remainder of this paper, we pick and fix once and for all an $a\in \Cx$, and work solely with representations whose $q$-characters lie in the subring $\Z[Y_{i,aq^k}^{\pm 1}]_{(i,k)\in \It}$.
It is helpful to define also \be \Iw := \{(i,k) : (i,k-r_i)\in \It \}\ee
for we have, as a refinement of (\ref{imchiq}), that for all $m_+\in \Z[Y_{i,aq^k}]_{(i,k)\in \It}$
\be\nn \mchiq(\L(m_+))\subset m_+ \Z[ A_{i,aq^k}^{-1}]_{(i,k)\in \Iw}.\ee

\begin{rem}\label{nolossremark} 
For all $m\in \P^+$, $\L(m)$ is isomorphic to a tensor product of simple $\uqgh$-modules
such that for each simple factor $\L(m')$ there exists an $a\in \Cx$ such that $\chi_q(\L(m')) \in  \Z[Y^{\pm 1}_{i,aq^k}]_{(i,k)\in \It}$. So there is no loss of generality in restricting our attention to $\Z[Y_{i,aq^k}^{\pm 1}]_{(i,k)\in \It}$.
\end{rem}

From now on it is convenient to write, by an abuse of notation,
\be\nn Y_{i,k} := Y_{i,aq^k}, \quad A_{i,k} := A_{i,aq^k}\label{Yikdef},\quad u_{i,k} := u_{i,aq^k}\ee 
for all $(i,k)\in I \times \Z$ (\confer (\ref{udef}) for the definition of $u_{i,aq^k}$).  

\subsection{Snake position and minimal snake position}\label{snakepos} Let $(i,k)\in \It$. Let us say that a point $(i',k') \in \It$ is in \emph{in snake position} with respect to $(i,k)$ if and only if
\begin{enumerate}[Type A:]
\item $k'-k\geq |i'-i|+2$. 
\item  $ $
  \begin{align*}
   &i=i'=N :  &k'-k  &\geq 2 & \text{ and }\quad k'-k&\equiv 2 \mod 4\\
  &i<i'=N \text{ or } i'< i=N :  &k'-k&\geq 2|i'-i| +3 & \text{ and }\quad k'-k &\equiv 2|i'-i| - 1 \mod 4\\
  &i<N \text{ and } i'<N:  &k'-k&\geq 2|i'-i|+4 & \text{ and }\quad k'-k &\equiv 2|i'-i| \mod 4. \end{align*}
\end{enumerate} 
The point $(i',k')$ is in \emph{minimal} snake position to $(i,k)$ if and only if $k'-k$ is equal to the given lower bound. 
The meaning of snake position is illustrated in Figure \ref{snakeposfig}. Here, and subsequently, we draw the images of points in $\It$ under the  injective map $\iota:\It \to \Z \times \Z$ defined as follows.
\begin{align} \text{Type B:}&\qquad \iota:  (i,k)  \mapsto 
\begin{cases} (2i,k) & i<N\text{ and } 2N+k-2i \equiv 2\mod 4\\
             (4N-2-2i,k) & i<N\text{ and } 2N+k-2i \equiv 0\mod 4\\
                                 (2N-1,k)   & i = N .\end{cases}\label{iotadef}\\
  \text{Type A:}& \qquad \iota: (i,k)\mapsto (i,k).\nn\end{align} 
(We define $\iota$ in type A purely in order to make certain proofs more uniform in what follows.)
\subsection{Snakes and snake representations} We call a finite sequence $(i_t,k_t)$, $1\leq t \leq T$, $T\in \Z_{\geq 1}$, of points in $\It$ a \emph{(minimal) snake} if and only if for all $2\leq t \leq T$, $(i_t,p_t)$ is in (minimal) snake position with respect to $(i_{t-1},k_{t-1})$. 
We call the simple module $\L(m)$ a \emph{(minimal) snake module} if and only if $m=\prod_{t=1}^T \YY {i_t}{k_t}$ for some (minimal) snake $(i_t,k_t)_{1\leq t\leq T}$.

The notion of snake position includes the correct position for minimal affinizations, but is more general. For example, one sees that in type $A_4$, $(2,3)$ and $(3,4)$ are both in snake position to $(2,1)$, but so too is $(3,6)$.  
$\L(\YY21\YY23)$ and $\L(\YY21\YY34)$ are minimal affinizations of $V(2\omega_2)$ and $V(\omega_2+\omega_3)$ respectively, but $\L(\YY21\YY36)$ is not a minimal affinization. Also, $\L(\YY21\YY34\YY27)$ is a minimal snake module but not a minimal affinization. In fact, by comparing it with the classification of minimal affinizations in types A and B \cite{Cminaffrank2,CPminaffBCFG}, one can verify

\begin{prop}\label{minaffprop}
Let $(i_t,k_t)\in \It$, $1\leq t \leq T$, $T\in \Z_{\geq 1}$. The following are equivalent:
\begin{enumerate}
\item $\L(\prod_{t=1}^T Y_{i_t,k_t})$ is a minimal affinization
\item $(i_t,k_t)_{1\leq t \leq T}$ is a minimal snake and the sequence $(i_t)_{1\leq t\leq T}$ is monotonic. \qed 
\end{enumerate} 
\end{prop}

Thus we have the following nested classes of representations, in order of increasing generality:
\be 
\textbf{KR modules} \subset \textbf{minimal affinizations} \subset \textbf{minimal snake modules} \subset \textbf{snake modules.} \nn\ee

\begin{rem*}There is no upper bound on the gap $k'-k$ in the definition of snake position. Suppose  $(i_t,k_t)\in\It$, $1\leq t\leq T$ is a snake of length $T\in \Z_{\geq 1}$. If $k_{s+1}-k_s$ is sufficiently large, for some $1\leq s<T$, then $\L(\prod_{t=1}^T Y_{i_t,k_t}) \cong \L(\prod_{t=1}^s Y_{i_t,k_t}) \otimes \L(\prod_{t=s+1}^T Y_{i_t,k_t})$.  Thus snake modules need not be \emph{prime}.
We shall see that minimal snakes are prime.
\end{rem*}

\section{Paths and moves}\label{sec:paths}

\subsection{Paths and corners} 
For each $(i,k)\in \It$ we shall define a set  $\scr P_{i,k}$ of paths. For us, a \emph{path} is a  finite sequence of points in the plane $\R^2$. 
We write $(j,\ell)\in p$ if $(j,\ell)$ is a point of the path $p$. In our diagrams, we connect consecutive points of the path by line segments, for illustrative purposes only.

\subsubsection*{Paths of Type A} For all $(i,k)\in \It$, let
\bea \scr P_{i,k} := \Big\{ \big( (0,y_0), (1,y_1), \dots, (N+1,y_{N+1}) \big) &:& y_0 = i+k,\,\, y_{N+1} = N+1-i+k,\nn\\ &\text{and}& y_{i+1}-y_i \in \{ 1, -1 \} \,\,\forall 0\leq i \leq N \Big\} . \nn\eea
We define the sets $C_{p,\pm}$ of upper and lower \emph{corners} of a path $p=\big( (r,y_r) \big)_{0\leq r\leq N+1} \in \scr P_{i,k}$ to be (\confer Figure \ref{Afundfig})
\bea C_{p,+} &:=& \left\{ (r,y_r) \in p:  r\in I,\,y_{r-1} = y_r+1  = y_{r+1}\right\} \nn\\
   C_{p,-} &:=& \left\{ (r,y_r) \in p:  r\in I,\,y_{r-1} = y_r-1  = y_{r+1}\right\}  \nn.\eea

\subsubsection*{Paths of Type B}  
Pick and fix an $\eps$, $1/2>\eps>0$. We first define $\scr P_{N,\ell}$ for all $\ell\in 2\Z+1$ as follows. 
\begin{itemize}\item
For all $\ell\equiv 3\!\!\mod 4$,
\bea \scr P_{N,\ell} &:=& \Big\{\big( (0,y_0), (2,y_1), \dots, (2N -4,y_{N-2}), (2N-2,y_{N-1}), (2N-1,y_{N})\big) \nn\\ &&\quad: \quad y_0 = \ell+2N-1,  y_{i+1}-y_i \in\{2,-2\} \,\,\forall 0\leq i\leq N-2  \nn \\ &&\qquad\qquad\qquad\quad\text{and}\quad y_{N}-y_{N-1} \in \{1+\eps,-1-\eps \} \Big\}.\nn\eea
\item
For all $\ell\equiv 1\!\!\mod 4$,
\bea \scr P_{N,\ell} &:=& \Big\{\big( (4N-2,y_0), (4N-4,y_1), \dots, (2N +2,y_{N-2}), (2N,y_{N-1}), (2N-1,y_{N})\big) \nn\\ &&\quad: \quad y_0 = \ell+2N-1,  y_{i+1}-y_i \in\{2,-2\} \,\,\forall 0\leq i\leq N-2  \nn \\ &&\qquad\qquad\qquad\quad\text{and}\quad y_{N}-y_{N-1} \in \{1+\eps,-1-\eps \} \Big\}.\nn \eea
\end{itemize}
Next we define $\scr P_{i,k}$ for all $(i,k)\in \It$, $i<N$, as follows.
\bea \scr P_{i,k} := \Big\{ (a_0,a_1, \dots, a_{N}, \bar a_{N} , \dots ,\bar a_1,\bar a_0) &:& (a_0,a_1,\dots,a_N) \in \scr P_{N,k-(2N-2i-1)},\nn\\&& (\bar a_0,\bar a_1,\dots,\bar a_N) \in \scr P_{N,k+(2N-2i-1)},\label{bvectpathdef1}\\
&\text{and}& a_N-\bar a_N = (0,y) \quad\text{where}\quad y>0 \Big\} .\nn\eea
These definitions are illustrated in Figures \ref{Bimage}, \ref{Bspinfig} and \ref{Bfundfig}.

For all $(i,k)\in \It$, we define the sets of upper and lower \emph{corners} $C_{p,\pm}$ of a path $p= \big( (j_r,\ell_r) \big)_{0\leq t\leq |p|-1} \in \scr P_{i,k}$ as follows:
\bea C_{p,+} &:= & \iota^{-1}\Big\{ (j_r,\ell_r) \in p:  j_r\notin \{0,2N-1,4N-2\}, \,\ell_{r-1} > \ell_r,\,\,  \ell_{r+1}>\ell_r \Big\}\nn \\ 
           && {}\sqcup \{(N,\ell)\in \It: (2N-1,\ell-\eps)\in p 
                        \text{ and } (2N-1,\ell+\eps) \notin p \}\nn ,\eea
\bea C_{p,-} &:= & \iota^{-1}\Big\{ (j_r,\ell_r) \in p:   j_r\notin \{0,2N-1,4N-2\},\,\ell_{r-1} < \ell_r,\,\, \ell_{r+1}< \ell_r\Big\}\nn \\ 
           && {}\sqcup \{(N,\ell)\in \It:  (2N-1,\ell+\eps)\in p 
                          \text{ and } (2N-1,\ell-\eps) \notin p \}\nn .\eea
where $\iota$ is the map defined in (\ref{iotadef}). Note that $C_{p,\pm}$ is a subset of $\It$.

We define a map $\mon$ sending paths to monomials, as follows:
\be \mon : \bigsqcup_{(i,k)\in \It} \scr P_{i,k} \longrightarrow \Z\left[Y_{j,\ell}^{\pm 1}\right]_{(j,\ell)\in \It}\,\,;\quad
 p \mapsto \mon(p) := \prod_{(j,\ell)\in C_{p,+} } \YY j \ell \prod_{(j,\ell)\in C_{p,-} } \MM j \ell.\label{mondef}\ee

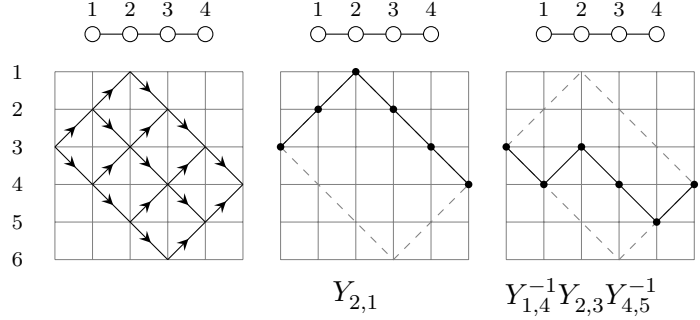
\begin{figure} 
\caption{In type $A_4$: left, illustration of the paths in $\scr P_{2,1}$; centre and right, the paths corresponding to two monomials of $\chi_q(\L(\YY 2 1))$.\label{Afundfig}}  
\be \begin{tikzpicture}[scale=.5,yscale=-1,decoration={
markings,
mark=at position .55 with {\arrow[black,line width=.3mm]{stealth}}}]
\draw[help lines] (0,0) grid (5,5);
\draw (1,-1) -- (4,-1);
\foreach \x in {1,2,3,4} {\filldraw[fill=white] (\x,-1) circle (2mm) node[above=1mm] {$\scriptstyle\x$}; }
\foreach \y in {1,2,3,4,5,6} {\node at (-7,\y-1) {$\scriptstyle\y$};}
\draw[gray,dashed] (0,2) -- (2,0) -- (5,3)--(3,5)--cycle;
\begin{scope}[every node/.style={minimum size=.1cm,inner sep=0mm,fill,circle}]
\draw (0,2) node{} -- ++(1,-1) node{} -- ++(1,-1) node{}-- ++(1,1) node{} -- ++(1,1) node{} -- ++(1,1) node{};
\end{scope}
\node[fill=gray!0] at (2,6) { $\YY21$};
\begin{scope}[xshift =- 6cm]
\draw[help lines] (0,0) grid (5,5);
\draw (1,-1) -- (4,-1);
\foreach \x in {1,2,3,4} {\filldraw[fill=white] (\x,-1) circle (2mm) node[above=1mm] {$\scriptstyle\x$}; }
\foreach \x/\y in {0/2,1/1,1/3,2/2,2/4,3/3} {\draw[postaction={decorate}] (\x,\y) -- ++(1,1); \draw[postaction={decorate}] (\x,\y) --++(1,-1);}
\foreach \x/\y in {2/0,3/1,4/2} {\draw[postaction={decorate}] (\x,\y) -- ++(1,1);}
\foreach \x/\y in {3/5,4/4} {\draw[postaction={decorate}] (\x,\y) -- ++(1,-1);}
\end{scope}
\begin{scope}[xshift = 6cm]
\draw[help lines] (0,0) grid (5,5);
\draw (1,-1) -- (4,-1);
\foreach \x in {1,2,3,4} {\filldraw[fill=white] (\x,-1) circle (2mm) node[above=1mm] {$\scriptstyle\x$}; }
\draw[gray,dashed] (0,2) -- (2,0) -- (5,3)--(3,5)--cycle;
\begin{scope}[every node/.style={minimum size=.1cm,inner sep=0mm,fill,circle}]
\draw (0,2) node{} -- ++(1,1) node{} -- ++(1,-1) node{}-- ++(1,1) node{} -- ++(1,1) node{} -- ++(1,-1) node{};
\end{scope}
\node[fill=gray!0] at (2,6) {$ \MM 1 4  \YY 2 3\MM 4 5$};
\end{scope}\end{tikzpicture}\nn\ee\end{figure}
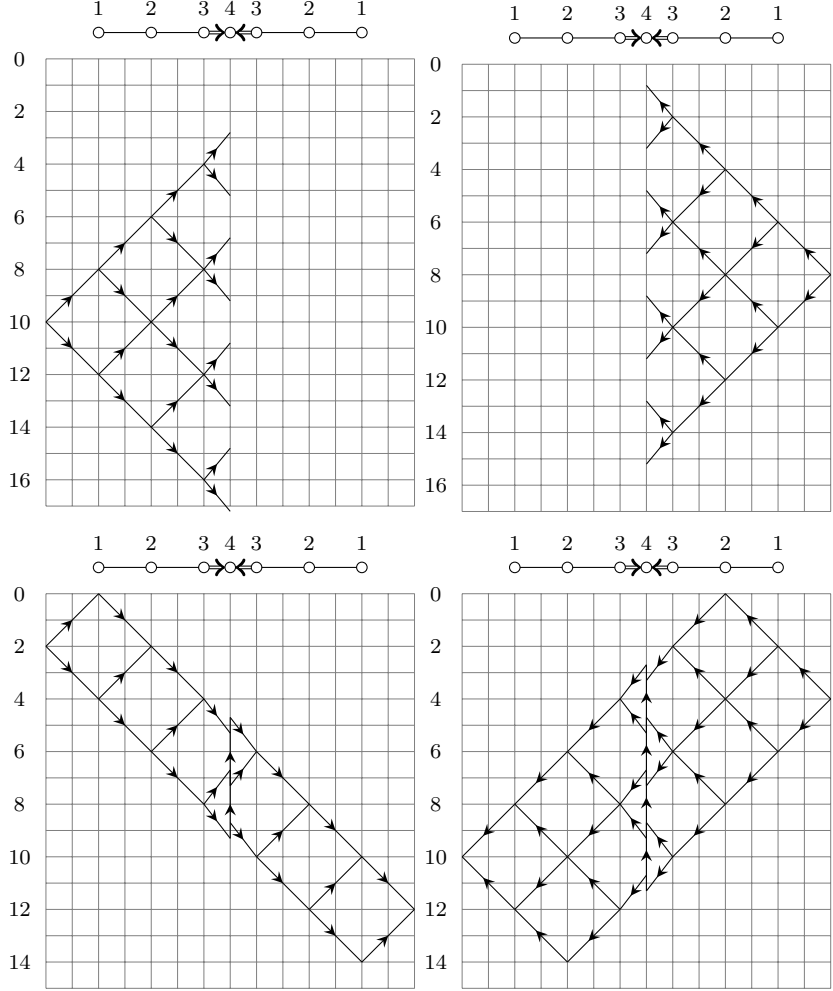
\begin{figure}
\caption{\label{Bimage} In type $B_4$, illustration of the paths in: $\scr P_{4,3}$ (top left), $\scr P_{4,1}$ (top right),  $\scr P_{1,0}$ (bottom left), and $\scr P_{2,0}$ (bottom right).}
\be
\begin{tikzpicture}[scale=.35,yscale=-1,decoration={
markings,
mark=at position .5 with {\arrow[black,line width=.3mm]{stealth}}}
]
\draw[help lines] (0,0) grid (14,17);
\draw (2,-1) -- (6,-1); \draw (8,-1) -- (12,-1);
\draw[double,->] (6,-1) -- (6.8,-1); \draw[double,->] (8,-1) -- (7.2,-1);
\filldraw[fill=white] (7,-1) circle (2mm) node[above=1mm] {$\scriptstyle 4$};
\foreach \x in {1,2,3} {
\filldraw[fill=white] (2*\x,-1) circle (2mm) node[above=1mm] {$\scriptstyle\x$}; 
\filldraw[fill=white] (2*7-2*\x,-1) circle (2mm) node[above=1mm] {$\scriptstyle\x$}; }
\foreach \y in {0,2,4,6,8,10,12,14,16} {\node at (-1,\y) {$\scriptstyle\y$};}
\begin{scope}[yshift=2cm]
\foreach \x/\y in {0/8,2/6,2/10,4/4,4/8,4/12} {\draw[postaction={decorate}] (\x,\y) -- ++(2,2); \draw[postaction={decorate}] (\x,\y) --++(2,-2);}
\foreach \y in {2,6,10,14}  {\draw[postaction={decorate}] (6,\y) -- ++(1,1.2); \draw[postaction={decorate}] (6,\y) --++(1,-1.2);}
\end{scope}

\end{tikzpicture}
\begin{tikzpicture}[scale=.35,yscale=-1,decoration={
markings,
mark=at position .5 with {\arrow[black,line width=.3mm]{stealth}}}]
\draw[help lines] (0,0) grid (14,17);
\draw (2,-1) -- (6,-1); \draw (8,-1) -- (12,-1);
\draw[double,->] (6,-1) -- (6.8,-1); \draw[double,->] (8,-1) -- (7.2,-1);
\filldraw[fill=white] (7,-1) circle (2mm) node[above=1mm] {$\scriptstyle 4$};
\foreach \x in {1,2,3} {
\filldraw[fill=white] (2*\x,-1) circle (2mm) node[above=1mm] {$\scriptstyle\x$}; 
\filldraw[fill=white] (2*7-2*\x,-1) circle (2mm) node[above=1mm] {$\scriptstyle\x$}; }
\foreach \y in {0,2,4,6,8,10,12,14,16} {\node at (-1,\y) {$\scriptstyle\y$};}
\begin{scope}[xshift=14cm,xscale=-1]
\foreach \x/\y in {0/8,2/6,2/10,4/4,4/8,4/12} {\draw[postaction={decorate}] (\x,\y) -- ++(2,2); \draw (\x,\y)[postaction={decorate}] --++(2,-2);}
\foreach \y in {2,6,10,14}  {\draw[postaction={decorate}] (6,\y) -- ++(1,1.2); \draw[postaction={decorate}] (6,\y) --++(1,-1.2);}
\end{scope}
\end{tikzpicture}
\nn\ee
\be \begin{tikzpicture}[scale=.35,yscale=-1,decoration={
markings,
mark=at position .5 with {\arrow[black,line width=.3mm]{stealth}}}]
\draw[help lines] (0,0) grid (14,15);
\draw (2,-1) -- (6,-1); \draw (8,-1) -- (12,-1);
\draw[double,->] (6,-1) -- (6.8,-1); \draw[double,->] (8,-1) -- (7.2,-1);
\filldraw[fill=white] (7,-1) circle (2mm) node[above=1mm] {$\scriptstyle 4$};
\foreach \x in {1,2,3} {
\filldraw[fill=white] (2*\x,-1) circle (2mm) node[above=1mm] {$\scriptstyle\x$}; 
\filldraw[fill=white] (2*7-2*\x,-1) circle (2mm) node[above=1mm] {$\scriptstyle\x$}; }
\foreach \y in {0,2,4,6,8,10,12,14} {\node at (-1,\y) {$\scriptstyle\y$};}
\foreach \x/\y in {0/2,2/4,4/6,8/10,10/12} {\draw[postaction={decorate}] (\x,\y) -- ++(2,2); \draw (\x,\y)[postaction={decorate}] --++(2,-2);}
\foreach \x/\y in {2/0,4/2,8/6,10/8,12/10} {\draw[postaction={decorate}] (\x,\y) -- ++(2,2);}
\foreach \x/\y in {12/14} {\draw[postaction={decorate}] (\x,\y) -- ++(2,-2);}
\foreach \x/\y in {6/4,6/8} {\draw[postaction={decorate}] (\x,\y) -- ++(1,1.3);
\draw[postaction={decorate}] (\x,\y)++(1,0.7) -- ++(1,1.3); }
\foreach \x/\y in {6/8} {\draw[postaction={decorate}] (\x,\y) -- ++(1,-1.3);
\draw[postaction={decorate}] (\x,\y)++(1,-0.7) -- ++(1,-1.3); }
\foreach \x/\y in {7/9,7/7} {\draw[postaction={decorate}] (\x,\y)++(0,.3) -- ++(0,-2.6); }

\end{tikzpicture}
\begin{tikzpicture}[scale=.35,yscale=-1,decoration={
markings,
mark=at position .6 with {\arrow[black,line width=.3mm]{stealth}}}]
\draw[help lines] (0,0) grid (14,15);
\draw (2,-1) -- (6,-1); \draw (8,-1) -- (12,-1);
\draw[double,->] (6,-1) -- (6.8,-1); \draw[double,->] (8,-1) -- (7.2,-1);
\filldraw[fill=white] (7,-1) circle (2mm) node[above=1mm] {$\scriptstyle 4$};
\foreach \x in {1,2,3} {
\filldraw[fill=white] (2*\x,-1) circle (2mm) node[above=1mm] {$\scriptstyle\x$}; 
\filldraw[fill=white] (2*7-2*\x,-1) circle (2mm) node[above=1mm] {$\scriptstyle\x$}; }
\foreach \y in {0,2,4,6,8,10,12,14} {\node at (-1,\y) {$\scriptstyle\y$};}
\begin{scope}[xshift=14cm,xscale=-1]
\foreach \x/\y in {0/4,2/6,4/8,8/12,2/2,4/4,8/8,10/10} {\draw[postaction={decorate}] (\x,\y) -- ++(2,2); \draw (\x,\y)[postaction={decorate}] --++(2,-2);}
\foreach \x/\y in {4/0,8/4,10/6,12/8} {\draw[postaction={decorate}] (\x,\y) -- ++(2,2);}
\foreach \x/\y in {10/14,12/12} {\draw[postaction={decorate}] (\x,\y) -- ++(2,-2);}
\foreach \x/\y in {6/2,6/6,6/10} {\draw[postaction={decorate}] (\x,\y) -- ++(1,1.3);
\draw[postaction={decorate}] (\x,\y)++(1,0.7) -- ++(1,1.3); }
\foreach \x/\y in {6/6,6/10} {\draw[postaction={decorate}] (\x,\y) -- ++(1,-1.3);
\draw[postaction={decorate}] (\x,\y)++(1,-0.7) -- ++(1,-1.3); }
\foreach \x/\y in {7/11,7/9,7/7,7/5} {\draw[postaction={decorate}] (\x,\y)++(0,.3) -- ++(0,-2.6); }
\end{scope}
\end{tikzpicture}\nn
\ee

\end{figure}

\begin{figure}
\caption{\label{Bspinfig} In type $B_4$, the paths corresponding to three monomials of $\chi_q(\L(Y_{4,1}))$}
\be \begin{tikzpicture}[scale=.35,yscale=-1]
\draw[help lines] (0,0) grid (14,15);
\draw (2,-1) -- (6,-1); \draw (8,-1) -- (12,-1);
\draw[double,->] (6,-1) -- (6.8,-1); \draw[double,->] (8,-1) -- (7.2,-1);
\filldraw[fill=white] (7,-1) circle (2mm) node[above=1mm] {$\scriptstyle 4$};
\foreach \x in {1,2,3} {
\filldraw[fill=white] (2*\x,-1) circle (2mm) node[above=1mm] {$\scriptstyle\x$}; 
\filldraw[fill=white] (2*7-2*\x,-1) circle (2mm) node[above=1mm] {$\scriptstyle\x$}; }
\foreach \y in {0,2,4,6,8,10,12,14} {\node at (-1,\y) {$\scriptstyle\y$};}
\begin{scope}[yshift=1cm]
\draw[gray,dashed] (7,-.3) -- ++(1,1.3) -- (14,7)--(7,14.3);
\begin{scope}[every node/.style={minimum size=.1cm,inner sep=0mm,fill,circle}]
\draw  (7,-.3) node{} -- ++(1,1.3) node{}
 -- ++(2,2) node{} -- ++(2,2) node{}--  ++(2,2) node{};
\end{scope}

\end{scope}
\node[fill=gray!0] at (4,13) {$Y_{4,1}$};
\end{tikzpicture} 
\begin{tikzpicture}[scale=.35,yscale=-1]
\draw[help lines] (0,0) grid (14,15);
\draw (2,-1) -- (6,-1); \draw (8,-1) -- (12,-1);
\draw[double,->] (6,-1) -- (6.8,-1); \draw[double,->] (8,-1) -- (7.2,-1);
\filldraw[fill=white] (7,-1) circle (2mm) node[above=1mm] {$\scriptstyle 4$};
\foreach \x in {1,2,3} {
\filldraw[fill=white] (2*\x,-1) circle (2mm) node[above=1mm] {$\scriptstyle\x$}; 
\filldraw[fill=white] (2*7-2*\x,-1) circle (2mm) node[above=1mm] {$\scriptstyle\x$}; }
\foreach \y in {0,2,4,6,8,10,12,14} {\node at (-1,\y) {$\scriptstyle\y$};}
\begin{scope}[yshift=1cm]
\draw[gray,dashed] (7,-.3) -- ++(1,1.3) -- (14,7)--(7,14.3);
\begin{scope}[every node/.style={minimum size=.1cm,inner sep=0mm,fill,circle}]
\draw  (7,2.3) node{} -- ++(1,-1.3) node{}
 -- ++(2,2) node{} -- ++(2,2) node{}--  ++(2,2) node{};
\end{scope}
\end{scope}
\node[fill=gray!0] at (4,13) {$Y_{4,3}^{-1}Y_{3,2}$};
\end{tikzpicture}
\begin{tikzpicture}[scale=.35,yscale=-1]
\draw[help lines] (0,0) grid (14,15);
\draw (2,-1) -- (6,-1); \draw (8,-1) -- (12,-1);
\draw[double,->] (6,-1) -- (6.8,-1); \draw[double,->] (8,-1) -- (7.2,-1);
\filldraw[fill=white] (7,-1) circle (2mm) node[above=1mm] {$\scriptstyle 4$};
\foreach \x in {1,2,3} {
\filldraw[fill=white] (2*\x,-1) circle (2mm) node[above=1mm] {$\scriptstyle\x$}; 
\filldraw[fill=white] (2*7-2*\x,-1) circle (2mm) node[above=1mm] {$\scriptstyle\x$}; }
\foreach \y in {0,2,4,6,8,10,12,14} {\node at (-1,\y) {$\scriptstyle\y$};}
\begin{scope}[yshift=1cm]
\draw[gray,dashed] (7,-.3) -- ++(1,1.3) -- (14,7)--(7,14.3);
\begin{scope}[every node/.style={minimum size=.1cm,inner sep=0mm,fill,circle}]
\draw  (7,6.3) node{} -- ++(1,-1.3) node{}
 -- ++(2,2) node{} -- ++(2,-2) node{}--  ++(2,2) node{};
\end{scope}
\end{scope}
\node[fill=gray!0] at (4,13) {$Y_{4,7}^{-1}Y_{3,6} Y_{2,8}^{-1} Y_{1,6}$};
\end{tikzpicture}\nn\ee
\end{figure}

\begin{figure}
\caption{\label{Bfundfig}In type $B_4$, the paths corresponding to the monomials of $\chi_q(\L(Y_{1,0}))$.}
\be \begin{tikzpicture}[scale=.35,yscale=-1]
\draw[help lines] (0,0) grid (14,15);
\draw (2,-1) -- (6,-1); \draw (8,-1) -- (12,-1);
\draw[double,->] (6,-1) -- (6.8,-1); \draw[double,->] (8,-1) -- (7.2,-1);
\filldraw[fill=white] (7,-1) circle (2mm) node[above=1mm] {$\scriptstyle 4$};
\foreach \x in {1,2,3} {
\filldraw[fill=white] (2*\x,-1) circle (2mm) node[above=1mm] {$\scriptstyle\x$}; 
\filldraw[fill=white] (2*7-2*\x,-1) circle (2mm) node[above=1mm] {$\scriptstyle\x$}; }
\foreach \y in {0,2,4,6,8,10,12,14} {\node at (-1,\y) {$\scriptstyle\y$};}
\node[fill=gray!0] at (4,13) { $\YY10$};
\draw[gray,dashed] (0,2) -- (2,0) -- (14,12)--(12,14)--cycle;
\begin{scope}[every node/.style={minimum size=.1cm,inner sep=0mm,fill,circle}]
\draw (0,2) node{}  -- ++(2,-2) node{}   -- ++(2,2) node{}  -- ++(2,2) node{}  -- ++(1,1.3) node{}  -- ++(0,-.6) node{}  -- ++(1,1.3) node{}  -- ++(2,2) node{}  -- ++(2,2) node{}  -- ++(2,2) node{} ;
\end{scope}
\end{tikzpicture}
\begin{tikzpicture}[scale=.35,yscale=-1]
\draw[help lines] (0,0) grid (14,15);
\draw (2,-1) -- (6,-1); \draw (8,-1) -- (12,-1);
\draw[double,->] (6,-1) -- (6.8,-1); \draw[double,->] (8,-1) -- (7.2,-1);
\filldraw[fill=white] (7,-1) circle (2mm) node[above=1mm] {$\scriptstyle 4$};
\foreach \x in {1,2,3} {
\filldraw[fill=white] (2*\x,-1) circle (2mm) node[above=1mm] {$\scriptstyle\x$}; 
\filldraw[fill=white] (2*7-2*\x,-1) circle (2mm) node[above=1mm] {$\scriptstyle\x$}; }
\foreach \y in {0,2,4,6,8,10,12,14} {\node at (-1,\y) {$\scriptstyle\y$};}
\draw[gray,dashed] (0,2) -- (2,0) -- (14,12)--(12,14)--cycle;
\begin{scope}[every node/.style={minimum size=.1cm,inner sep=0mm,fill,circle}]
\draw (0,2) node{}  -- ++(2,2) node{}   -- ++(2,-2) node{}  -- ++(2,2) node{}  -- ++(1,1.3) node{}  -- ++(0,-.6) node{}  -- ++(1,1.3) node{}  -- ++(2,2) node{}  -- ++(2,2) node{}  -- ++(2,2) node{} ;
\end{scope}
\node[fill=gray!0] at (4,13) {$Y_{2,2}Y_{1,4}^{-1}$};
\end{tikzpicture}
\begin{tikzpicture}[scale=.35,yscale=-1]
\draw[help lines] (0,0) grid (14,15);
\draw (2,-1) -- (6,-1); \draw (8,-1) -- (12,-1);
\draw[double,->] (6,-1) -- (6.8,-1); \draw[double,->] (8,-1) -- (7.2,-1);
\filldraw[fill=white] (7,-1) circle (2mm) node[above=1mm] {$\scriptstyle 4$};
\foreach \x in {1,2,3} {
\filldraw[fill=white] (2*\x,-1) circle (2mm) node[above=1mm] {$\scriptstyle\x$}; 
\filldraw[fill=white] (2*7-2*\x,-1) circle (2mm) node[above=1mm] {$\scriptstyle\x$}; }
\foreach \y in {0,2,4,6,8,10,12,14} {\node at (-1,\y) {$\scriptstyle\y$};}
\draw[gray,dashed] (0,2) -- (2,0) -- (14,12)--(12,14)--cycle;
\begin{scope}[every node/.style={minimum size=.1cm,inner sep=0mm,fill,circle}]
\draw (0,2) node{}  -- ++(2,2) node{}   -- ++(2,2) node{}  -- ++(2,-2) node{}  -- ++(1,1.3) node{}  -- ++(0,-.6) node{}  -- ++(1,1.3) node{}  -- ++(2,2) node{}  -- ++(2,2) node{}  -- ++(2,2) node{} ;
\end{scope}
\node[fill=gray!0] at (4,13) {$Y_{3,4}Y_{2,6}^{-1}$};
\end{tikzpicture}\nn\ee
$ $
\be \begin{tikzpicture}[scale=.35,yscale=-1]
\draw[help lines] (0,0) grid (14,15);
\draw (2,-1) -- (6,-1); \draw (8,-1) -- (12,-1);
\draw[double,->] (6,-1) -- (6.8,-1); \draw[double,->] (8,-1) -- (7.2,-1);
\filldraw[fill=white] (7,-1) circle (2mm) node[above=1mm] {$\scriptstyle 4$};
\foreach \x in {1,2,3} {
\filldraw[fill=white] (2*\x,-1) circle (2mm) node[above=1mm] {$\scriptstyle\x$}; 
\filldraw[fill=white] (2*7-2*\x,-1) circle (2mm) node[above=1mm] {$\scriptstyle\x$}; }
\foreach \y in {0,2,4,6,8,10,12,14} {\node at (-1,\y) {$\scriptstyle\y$};}
\draw[gray,dashed] (0,2) -- (2,0) -- (14,12)--(12,14)--cycle;
\begin{scope}[every node/.style={minimum size=.1cm,inner sep=0mm,fill,circle}]
\draw (0,2) node{}  -- ++(2,2) node{}   -- ++(2,2) node{}  -- ++(2,2) node{}  -- ++(1,-1.3) node{}  -- ++(0,-2) node{}  -- ++(1,1.3) node{}  -- ++(2,2) node{}  -- ++(2,2) node{}  -- ++(2,2) node{} ;
\end{scope}
\node[fill=gray!0] at (4,13) {$Y_{4,5}Y_{4,7}Y_{3,8}^{-1}$};
\end{tikzpicture}
\begin{tikzpicture}[scale=.35,yscale=-1]
\draw[help lines] (0,0) grid (14,15);
\draw (2,-1) -- (6,-1); \draw (8,-1) -- (12,-1);
\draw[double,->] (6,-1) -- (6.8,-1); \draw[double,->] (8,-1) -- (7.2,-1);
\filldraw[fill=white] (7,-1) circle (2mm) node[above=1mm] {$\scriptstyle 4$};
\foreach \x in {1,2,3} {
\filldraw[fill=white] (2*\x,-1) circle (2mm) node[above=1mm] {$\scriptstyle\x$}; 
\filldraw[fill=white] (2*7-2*\x,-1) circle (2mm) node[above=1mm] {$\scriptstyle\x$}; }
\foreach \y in {0,2,4,6,8,10,12,14} {\node at (-1,\y) {$\scriptstyle\y$};}
\draw[gray,dashed] (0,2) -- (2,0) -- (14,12)--(12,14)--cycle;
\begin{scope}[every node/.style={minimum size=.1cm,inner sep=0mm,fill,circle}]
\draw (0,2) node{}  -- ++(2,2) node{}   -- ++(2,2) node{}  -- ++(2,2) node{}  -- ++(1,1.3) node{}  -- ++(0,-4.6) node{}  -- ++(1,1.3) node{}  -- ++(2,2) node{}  -- ++(2,2) node{}  -- ++(2,2) node{} ;
\end{scope}
\node[fill=gray!0] at (4,13) {$Y_{4,5}Y_{4,9}^{-1}$};
\end{tikzpicture}
\begin{tikzpicture}[scale=.35,yscale=-1]
\draw[help lines] (0,0) grid (14,15);
\draw (2,-1) -- (6,-1); \draw (8,-1) -- (12,-1);
\draw[double,->] (6,-1) -- (6.8,-1); \draw[double,->] (8,-1) -- (7.2,-1);
\filldraw[fill=white] (7,-1) circle (2mm) node[above=1mm] {$\scriptstyle 4$};
\foreach \x in {1,2,3} {
\filldraw[fill=white] (2*\x,-1) circle (2mm) node[above=1mm] {$\scriptstyle\x$}; 
\filldraw[fill=white] (2*7-2*\x,-1) circle (2mm) node[above=1mm] {$\scriptstyle\x$}; }
\foreach \y in {0,2,4,6,8,10,12,14} {\node at (-1,\y) {$\scriptstyle\y$};}
\draw[gray,dashed] (0,2) -- (2,0) -- (14,12)--(12,14)--cycle;
\begin{scope}[every node/.style={minimum size=.1cm,inner sep=0mm,fill,circle}]
\draw (0,2) node{}  -- ++(2,2) node{}   -- ++(2,2) node{}  -- ++(2,2) node{}  -- ++(1,1.3) node{}  -- ++(0,-2) node{}  -- ++(1,-1.3) node{}  -- ++(2,2) node{}  -- ++(2,2) node{}  -- ++(2,2) node{} ;
\end{scope}
\node[fill=gray!0] at (4,13) {$Y_{4,7}^{-1}Y_{4,9}^{-1}Y_{3,6}$};
\end{tikzpicture}\nn\ee
$ $
\be \begin{tikzpicture}[scale=.35,yscale=-1]
\draw[help lines] (0,0) grid (14,15);
\draw (2,-1) -- (6,-1); \draw (8,-1) -- (12,-1);
\draw[double,->] (6,-1) -- (6.8,-1); \draw[double,->] (8,-1) -- (7.2,-1);
\filldraw[fill=white] (7,-1) circle (2mm) node[above=1mm] {$\scriptstyle 4$};
\foreach \x in {1,2,3} {
\filldraw[fill=white] (2*\x,-1) circle (2mm) node[above=1mm] {$\scriptstyle\x$}; 
\filldraw[fill=white] (2*7-2*\x,-1) circle (2mm) node[above=1mm] {$\scriptstyle\x$}; }
\foreach \y in {0,2,4,6,8,10,12,14} {\node at (-1,\y) {$\scriptstyle\y$};}
\node[fill=gray!0] at (4,13) { $Y_{3,10}^{-1}Y_{2,8}$};
\draw[gray,dashed] (0,2) -- (2,0) -- (14,12)--(12,14)--cycle;
\begin{scope}[every node/.style={minimum size=.1cm,inner sep=0mm,fill,circle}]
\draw (0,2) node{}  -- ++(2,2) node{}   -- ++(2,2) node{}  -- ++(2,2) node{}  -- ++(1,1.3) node{}  -- ++(0,-.6) node{}  -- ++(1,1.3) node{}  -- ++(2,-2) node{}  -- ++(2,2) node{}  -- ++(2,2) node{} ;
\end{scope}
\end{tikzpicture}
\begin{tikzpicture}[scale=.35,yscale=-1]
\draw[help lines] (0,0) grid (14,15);
\draw (2,-1) -- (6,-1); \draw (8,-1) -- (12,-1);
\draw[double,->] (6,-1) -- (6.8,-1); \draw[double,->] (8,-1) -- (7.2,-1);
\filldraw[fill=white] (7,-1) circle (2mm) node[above=1mm] {$\scriptstyle 4$};
\foreach \x in {1,2,3} {
\filldraw[fill=white] (2*\x,-1) circle (2mm) node[above=1mm] {$\scriptstyle\x$}; 
\filldraw[fill=white] (2*7-2*\x,-1) circle (2mm) node[above=1mm] {$\scriptstyle\x$}; }
\foreach \y in {0,2,4,6,8,10,12,14} {\node at (-1,\y) {$\scriptstyle\y$};}
\draw[gray,dashed] (0,2) -- (2,0) -- (14,12)--(12,14)--cycle;
\begin{scope}[every node/.style={minimum size=.1cm,inner sep=0mm,fill,circle}]
\draw (0,2) node{}  -- ++(2,2) node{}   -- ++(2,2) node{}  -- ++(2,2) node{}  -- ++(1,1.3) node{}  -- ++(0,-.6) node{}  -- ++(1,1.3) node{}  -- ++(2,2) node{}  -- ++(2,-2) node{}  -- ++(2,2) node{} ;
\end{scope}
\node[fill=gray!0] at (4,13) {$Y_{2,12}^{-1}Y_{1,10}$};
\end{tikzpicture}
\begin{tikzpicture}[scale=.35,yscale=-1]
\draw[help lines] (0,0) grid (14,15);
\draw (2,-1) -- (6,-1); \draw (8,-1) -- (12,-1);
\draw[double,->] (6,-1) -- (6.8,-1); \draw[double,->] (8,-1) -- (7.2,-1);
\filldraw[fill=white] (7,-1) circle (2mm) node[above=1mm] {$\scriptstyle 4$};
\foreach \x in {1,2,3} {
\filldraw[fill=white] (2*\x,-1) circle (2mm) node[above=1mm] {$\scriptstyle\x$}; 
\filldraw[fill=white] (2*7-2*\x,-1) circle (2mm) node[above=1mm] {$\scriptstyle\x$}; }
\foreach \y in {0,2,4,6,8,10,12,14} {\node at (-1,\y) {$\scriptstyle\y$};}
\draw[gray,dashed] (0,2) -- (2,0) -- (14,12)--(12,14)--cycle;
\begin{scope}[every node/.style={minimum size=.1cm,inner sep=0mm,fill,circle}]
\draw (0,2) node{}  -- ++(2,2) node{}   -- ++(2,2) node{}  -- ++(2,2) node{}  -- ++(1,1.3) node{}  -- ++(0,-.6) node{}  -- ++(1,1.3) node{}  -- ++(2,2) node{}  -- ++(2,2) node{}  -- ++(2,-2) node{} ;
\end{scope}
\node[fill=gray!0] at (4,13) {$Y_{1,14}^{-1}$};
\end{tikzpicture}\nn\ee
\end{figure}

\subsection{Lowering moves}\label{lowdef}
Let $(i,k)\in \It$ and $(j,\ell)\in \Iw$. We say a path $p\in \scr P_{i,k}$ can be \emph{lowered}  at $(j,\ell)$ if and only if $(j,\ell-r_j)\in C_{p,+}$ and $(j,\ell+r_j)\notin C_{p,+}$. If so, we define a \emph{lowering move} on $p$ at $(j,\ell)$, resulting in another path in $\scr P_{i,k}$ which we write as $p\scr A_{j,\ell}^{-1}$, as follows.

\subsubsection*{Moves of Type A} Let $p\in \scr P_{i,k}$. Then $p=\big( (i,y_i) \big)_{0\leq i\leq N+1}$ where $(y_i)_{0\leq i\leq N+1}\in \R^{N+2}$. The upper corners of $p$ are the points $(j,\ell-1)\in p$ such that $\ell= y_{j-1}=y_{j}+1 =y_{j+1}$. The second condition, that $(j,\ell+1)\notin C_{p,+}$, follows automatically and is thus redundant in type A.
For each such upper corner, we define 
$p\scr A_{j,\ell}^{-1}  := \big( (0,y_0),\dots, (j-1,y_{j-1}) , (j,y_j+2), (j+1,y_{j+1}), \dots, (N+1,y_{N+1}) \big)$, which is again a path in $\scr P_{i,k}$. 
Pictorially, this is the move
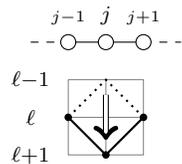
\be\nn\begin{tikzpicture}[baseline=0cm,scale=.5,yscale=-1]
\draw[help lines] (-1,-1) grid (1,1);
\draw (-1,-2) -- (1,-2);\draw[dashed] (-2,-2) -- (-1,-2);\draw[dashed] (1,-2) -- (2,-2);
\foreach \x in {-1,+1} {\filldraw[fill=white] (\x,-2) circle (2mm) node[above=1mm] {$\scriptscriptstyle j\x$}; }
\filldraw[fill=white] (0,-2) circle (2mm) node[above=1mm] {$\scriptstyle j$};
\foreach \y in {-1,+1} {\node at (-2,\y) {$\scriptstyle \ell\y$};}
\node at (-2,0) {$\scriptstyle \ell$};
\begin{scope}[every node/.style={minimum size=.1cm,inner sep=0mm,fill,circle}]
\draw[thick] (-1,0) node {} -- (0,1) node {} -- (1,0)node {} ;
\draw [thick,dotted] (1,0) -- (0,-1) -- (-1,0);
\end{scope}
\draw[double distance=.5mm,->,shorten <= 2mm,shorten >= 2mm] (0,-1) -- (0,1);
\end{tikzpicture}.\ee
\subsubsection*{Moves of Type B} 
We first define the lowering moves on paths in $\scr P_{N,k}$, $k\equiv 3\!\!\mod 4$. Note that, for all $p\in \scr P_{N,k}$, if $(j,\ell-r_j)\in C_{p,+}$ then $(j,\ell+r_j)\notin C_{p,+}$, so, as in type A, the latter condition is redundant in this case. For any $p\in P_{N,k}$, we have $p= \big( (0,y_0), (2,y_1), \dots, (2N-4,y_{N-2}), (2N-2,y_{N-1}), (2N-1,y_N) \big)$ for some $(y_i)_{0\leq i\leq N}\in \R^{N+1}$. We define the lowering moves case-by-case:
\begin{enumerate}[i)]
\item If $(j,\ell-2)\in C_{p,+}$  for some $j<N-1$, we have $\ell=y_{j-1} = y_j+2=y_{j+1}$ and we define $p\scr A_{j,\ell}^{-1}  := \big( (0,y_0),\dots, (2j-2,y_{j-1}) , (2j,y_j+4), (2j+2,y_{j+1}), \dots, (2N-2,y_{N-1}), (2N-1,y_N) \big)$.
\item If $(N-1,\ell-2)\in C_{p,+}$, we have $\ell = y_{N-2} = y_{N-1}+2 = y_{N}+1-\eps$ and we define $p\scr A_{j,\ell}^{-1}  := \big( (0,y_0),\dots , (2N-4,y_{N-2}),(2N-2,y_{N-1}+4), (2N-1,y_N+2-2\eps) \big)$.
\item If $(N,\ell-1)\in C_{p,+}$, we have $\ell = y_{N-1} = y_N+1+\eps$ and we define $p\scr A_{j,\ell}^{-1}  := \big( (0,y_0),\dots , (2N-4,y_{N-2}),(2N-2,y_{N-1}), (2N-1,y_N+2+2\eps) \big)$. 
\end{enumerate}
In each case, $p\scr A_{j,\ell}^{-1}\in\scr P_{N,k}$.
Pictorially, these moves are:
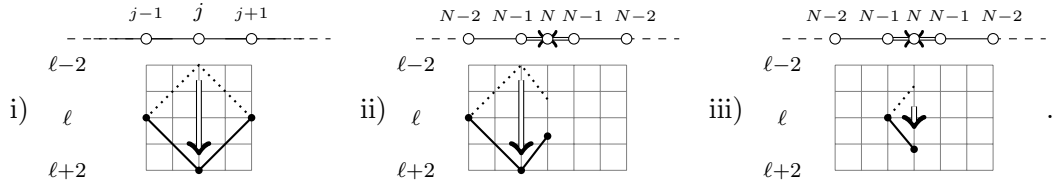
\be \text{i) } \begin{tikzpicture}[baseline=0cm,scale=.35,yscale=-1]
\draw[help lines] (-2,-2) grid (2,2);
\draw (-3,-3) -- (-1,-3);\draw (1,-3) -- (3,-3);
\draw[dashed] (-5,-3) -- (-3,-3);\draw[dashed] (3,-3) -- (5,-3);
\foreach \y in {-2,+2} {\node at (-5,\y) {$\scriptstyle \ell\y$};}
\node at (-5,0) {$\scriptstyle \ell$};
\begin{scope}[xscale=1]
\draw (-2,-3) -- (2,-3);
\draw[dashed] (-4,-3) -- (-2,-3);\draw[dashed] (2,-3) -- (4,-3);
\foreach \x in {-1,+1} {\filldraw[fill=white] (2*\x,-3) circle (2mm) node[above=1mm] {$\scriptscriptstyle j\x$}; }
\filldraw[fill=white] (0,-3) circle (2mm) node[above=1mm] {$\scriptstyle j$};
\end{scope}
\begin{scope}[every node/.style={minimum size=.1cm,inner sep=0mm,fill,circle}]
\draw[thick] (-2,0) node {} -- (0,2) node{} -- (2,0) node{};
\end{scope}
\draw [thick,dotted] (2,0) -- (0,-2) -- (-2,0);
\draw[double distance=.5mm,->,shorten <= 2mm,shorten >= 2mm] (0,-2) -- (0,2);
\end{tikzpicture}\quad\text{ii)}
\begin{tikzpicture}[baseline=0cm,scale=.35,yscale=-1]
\draw[help lines] (-3,-2) grid (3,2);
\draw[double,->] (-1,-3) -- (0,-3);
\draw[double,->] (1,-3) -- (0,-3);
\draw (-3,-3) -- (-1,-3);\draw (1,-3) -- (3,-3);
\draw[dashed] (-5,-3) -- (-3,-3);\draw[dashed] (3,-3) -- (5,-3);
\filldraw[fill=white] (-3,-3) circle (2mm) node[above=3mm, left=-3mm] {$\scriptscriptstyle N-2$};
\filldraw[fill=white] (-1,-3) circle (2mm) node[above=3mm,left=-3mm] {$\scriptscriptstyle N-1$};
\filldraw[fill=white] (0,-3) circle (2mm) node[above=1mm] {$\scriptscriptstyle N$};
\filldraw[fill=white] (1,-3) circle (2mm) node[above=3mm,right=-3mm] {$\scriptscriptstyle N-1$};
\filldraw[fill=white] (3,-3) circle (2mm) node[above=3mm,right=-3mm] {$\scriptscriptstyle N-2$};
\foreach \y in {-2,+2} {\node at (-5,\y) {$\scriptstyle \ell\y$};}
\node at (-5,0) {$\scriptstyle \ell$};
\begin{scope}[xscale=-1]
\begin{scope}[every node/.style={minimum size=.1cm,inner sep=0mm,fill,circle}]
\draw[thick]   (0,.7) node {}  -- (1,2)  node {} -- (3,0) node {}  ;
\end{scope}
\draw [thick,dotted] (0,-.7) -- (1,-2) -- (3,0);
\draw[double distance=.5mm,->,shorten <= 2mm,shorten >= 2mm] (1,-2) -- (1,2);\end{scope}
\end{tikzpicture} \quad\text{iii) }
\begin{tikzpicture}[baseline=0cm,scale=.35,yscale=-1]
\draw[help lines] (-3,-2) grid (3,2);
\draw[double,->] (-1,-3) -- (0,-3);
\draw[double,->] (1,-3) -- (0,-3);
\draw (-3,-3) -- (-1,-3);\draw (1,-3) -- (3,-3);
\draw[dashed] (-5,-3) -- (-3,-3);\draw[dashed] (3,-3) -- (5,-3);
\filldraw[fill=white] (-3,-3) circle (2mm) node[above=3mm, left=-3mm] {$\scriptscriptstyle N-2$};
\filldraw[fill=white] (-1,-3) circle (2mm) node[above=3mm,left=-3mm] {$\scriptscriptstyle N-1$};
\filldraw[fill=white] (0,-3) circle (2mm) node[above=1mm] {$\scriptscriptstyle N$};
\filldraw[fill=white] (1,-3) circle (2mm) node[above=3mm,right=-3mm] {$\scriptscriptstyle N-1$};
\filldraw[fill=white] (3,-3) circle (2mm) node[above=3mm,right=-3mm] {$\scriptscriptstyle N-2$};
\foreach \y in {-2,+2} {\node at (-5,\y) {$\scriptstyle \ell\y$};}
\node at (-5,0) {$\scriptstyle \ell$};
\begin{scope}[xscale=-1]
\begin{scope}[every node/.style={minimum size=.1cm,inner sep=0mm,fill,circle}]
\draw[thick]  (0,1.2) node {}-- (1,0)node {} ;
\end{scope}
\draw [thick,dotted] (0,-1.2) -- (1,0);
\draw[double distance=.5mm,->,shorten <= 2mm,shorten >= 2mm] (0,-1) -- (0,1);\end{scope}
\end{tikzpicture}\nn.\ee

Next we define the lowering moves on $\scr P_{N,k}$ with $k\equiv 1\!\!\mod 4$. Informally, these are simply the mirror images of the moves above. 
For any $p\in \scr P_{N,k}$, we have $p= \big( (4N-2,y_0), (4N-4,y_1), \dots, (2N +2,y_{N-2}), (2N,y_{N-1}), (2N-1,y_{N})\big)$ for some $(y_i)_{0\leq i\leq N}\in \R^{N+1}$. We define the lowering moves case-by-case:
\begin{enumerate}[i)]
\item If $(j,\ell-2)\in C_{p,+}$ for some $j<N-1$, we have $\ell=y_{j+1} = y_j+2=y_{j-1}$ and we define $p\scr A_{j,\ell}^{-1}  :=  \big((4N-2,y_0),  (4N-4,y_1),    \dots, (2j+2,y_{j-1}) , (2j,y_j+4), (2j-2,y_{j+1}), \dots, (2N+2,y_{N-2}),  (2N,y_{N-1}), (2N-1,y_N)  \big)$.
\item If $(N-1,\ell-2)\in C_{p,+}$, we have $\ell = y_{N}+1-\eps = y_{N-1}+2 = y_{N-2}$ and we define $p\scr A_{j,\ell}^{-1}  :=   \big( (4N-2,y_0), (4N-4,y_1), \dots, (2N +2,y_{N-2}), (2N,y_{N-1}+4), (2N-1,y_{N}+2-2\eps)\big)$.
\item If $(N,\ell-1)\in C_{p,+}$, we have $\ell = y_N+1+\eps=y_{N-1}$ and we define $p\scr A_{j,\ell}^{-1}  := \big( (4N-2,y_0), (4N-4,y_1), \dots, (2N +2,y_{N-2}), (2N,y_{N-1}), (2N-1,y_{N}+2+2\eps)\big)$. 
\end{enumerate}
In each case, $p\scr A_{j,\ell}^{-1}\in\scr P_{N,k}$. 
That is, pictorially,
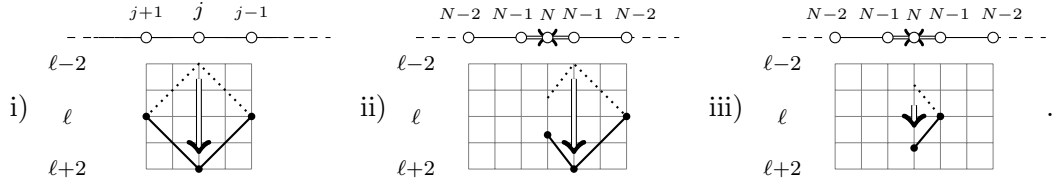
\be \text{i) } \begin{tikzpicture}[baseline=0cm,scale=.35,yscale=-1]
\draw[help lines] (-2,-2) grid (2,2);
\draw (-3,-3) -- (-1,-3);\draw (1,-3) -- (3,-3);
\draw[dashed] (-5,-3) -- (-3,-3);\draw[dashed] (3,-3) -- (5,-3);
\foreach \y in {-2,+2} {\node at (-5,\y) {$\scriptstyle \ell\y$};}
\node at (-5,0) {$\scriptstyle \ell$};
\begin{scope}[xscale=-1]
\draw (-2,-3) -- (2,-3);
\draw[dashed] (-4,-3) -- (-2,-3);\draw[dashed] (2,-3) -- (4,-3);
\foreach \x in {-1,+1} {\filldraw[fill=white] (2*\x,-3) circle (2mm) node[above=1mm] {$\scriptscriptstyle j\x$}; }
\filldraw[fill=white] (0,-3) circle (2mm) node[above=1mm] {$\scriptstyle j$};
\end{scope}
\begin{scope}[every node/.style={minimum size=.1cm,inner sep=0mm,fill,circle}]
\draw[thick] (-2,0) node {} -- (0,2) node{} -- (2,0) node{};
\end{scope}
\draw [thick,dotted] (2,0) -- (0,-2) -- (-2,0);
\draw[double distance=.5mm,->,shorten <= 2mm,shorten >= 2mm] (0,-2) -- (0,2);
\end{tikzpicture}\quad\text{ii)}
\begin{tikzpicture}[baseline=0cm,scale=.35,yscale=-1]
\draw[help lines] (-3,-2) grid (3,2);
\draw[double,->] (-1,-3) -- (0,-3);
\draw[double,->] (1,-3) -- (0,-3);
\draw (-3,-3) -- (-1,-3);\draw (1,-3) -- (3,-3);
\draw[dashed] (-5,-3) -- (-3,-3);\draw[dashed] (3,-3) -- (5,-3);
\filldraw[fill=white] (-3,-3) circle (2mm) node[above=3mm, left=-3mm] {$\scriptscriptstyle N-2$};
\filldraw[fill=white] (-1,-3) circle (2mm) node[above=3mm,left=-3mm] {$\scriptscriptstyle N-1$};
\filldraw[fill=white] (0,-3) circle (2mm) node[above=1mm] {$\scriptscriptstyle N$};
\filldraw[fill=white] (1,-3) circle (2mm) node[above=3mm,right=-3mm] {$\scriptscriptstyle N-1$};
\filldraw[fill=white] (3,-3) circle (2mm) node[above=3mm,right=-3mm] {$\scriptscriptstyle N-2$};
\foreach \y in {-2,+2} {\node at (-5,\y) {$\scriptstyle \ell\y$};}
\node at (-5,0) {$\scriptstyle \ell$};
\begin{scope}[xscale=1]
\begin{scope}[every node/.style={minimum size=.1cm,inner sep=0mm,fill,circle}]
\draw[thick]   (0,.7) node {}  -- (1,2)  node {} -- (3,0) node {}  ;
\end{scope}
\draw [thick,dotted] (0,-.7) -- (1,-2) -- (3,0);
\draw[double distance=.5mm,->,shorten <= 2mm,shorten >= 2mm] (1,-2) -- (1,2);\end{scope}
\end{tikzpicture} \quad\text{iii) }
\begin{tikzpicture}[baseline=0cm,scale=.35,yscale=-1]
\draw[help lines] (-3,-2) grid (3,2);
\draw[double,->] (-1,-3) -- (0,-3);
\draw[double,->] (1,-3) -- (0,-3);
\draw (-3,-3) -- (-1,-3);\draw (1,-3) -- (3,-3);
\draw[dashed] (-5,-3) -- (-3,-3);\draw[dashed] (3,-3) -- (5,-3);
\filldraw[fill=white] (-3,-3) circle (2mm) node[above=3mm, left=-3mm] {$\scriptscriptstyle N-2$};
\filldraw[fill=white] (-1,-3) circle (2mm) node[above=3mm,left=-3mm] {$\scriptscriptstyle N-1$};
\filldraw[fill=white] (0,-3) circle (2mm) node[above=1mm] {$\scriptscriptstyle N$};
\filldraw[fill=white] (1,-3) circle (2mm) node[above=3mm,right=-3mm] {$\scriptscriptstyle N-1$};
\filldraw[fill=white] (3,-3) circle (2mm) node[above=3mm,right=-3mm] {$\scriptscriptstyle N-2$};
\foreach \y in {-2,+2} {\node at (-5,\y) {$\scriptstyle \ell\y$};}
\node at (-5,0) {$\scriptstyle \ell$};
\begin{scope}[xscale=1]
\begin{scope}[every node/.style={minimum size=.1cm,inner sep=0mm,fill,circle}]
\draw[thick]  (0,1.2) node {}-- (1,0)node {} ;
\end{scope}
\draw [thick,dotted] (0,-1.2) -- (1,0);
\draw[double distance=.5mm,->,shorten <= 2mm,shorten >= 2mm] (0,-1) -- (0,1);\end{scope}
\end{tikzpicture}\nn.\ee

Finally, for all $(i,k)\in \It$ with $i<N$, one has (\confer \ref{bvectpathdef1}) that every path $p\in\scr P_{i,k}$ is equivalent to a pair $(a,b)$ where $a\in \scr P_{N,k-(2N-2i-1)}$ and $b\in \scr P_{N,k+(2N-2i-1)}$. By construction, $a$ and $b$ can share no upper corners. Thus the set of upper corners of $p$ is a subset of the \emph{disjoint} union of the sets of upper corners of $a$ and $b$. If $p$ can be lowered at $(j,\ell)$, exactly one of $a$ and $b$ can be lowered at $(j,\ell)$, and we define  $p\scr A_{j,\ell}^{-1}$ to be given by whichever of the pairs $\left( a \scr A_{j,\ell}^{-1},  b\right)$ and $\left( a , b\scr A_{j,\ell}^{-1} \right)$ is defined.  
It is here that the second criteria for it to be possible to lower $p$ at $(j,\ell)$, namely that $(j,\ell+r_j)\notin C_{p,+}$, is used, for it ensures that the condition on $a_N-\bar a_N$ in (\ref{bvectpathdef1}) is preserved.   

\subsection{Raising moves} 
Let $(i,k)\in \It$ and $(j,\ell)\in \Iw$. We say a path $p\in \scr P_{i,k}$ can be \emph{raised at} $(j,\ell)$ if and only if $p=p'\scr A_{j,\ell}^{-1}$ for some $p'\in \scr P_{i,k}$. If $p'$ exists it is unique, and we define $p\scr A_{j,\ell}:=p'$. It is straightforward to verify that $p$ can be raised at $(j,\ell)$ if and only if  $(j,\ell+r_j)\in C_{p,-}$ and $(j,\ell-r_j)\notin C_{p,-}$. 

\subsection{The highest/lowest path}\label{sec:highestpath} For all $(i,k)\in \It$, define $\phigh_{i,k}$, the \emph{highest path} to be the unique path in $\scr P_{i,k}$ with no lower corners. 
Equivalently, $\phigh_{i,k}$ is the unique path such that:
\begin{align}\text{Type A}& :  &(i,k) &\in \phigh_{i,k}  \nn\\
\text{Type B},&\,\, i<N:  &\iota(i,k) &\in \phigh_{i,k} \nn\\
\text{Type B},&\,\, i=N:  &(2N-1,k) - (0,\eps)&\in \phigh_{N,k} .\nn\end{align} 
Define $\plow_{i,k}$, the \emph{lowest path}, to be the unique path in $\scr P_{i,k}$ with no upper corners. 
Equivalently, $\plow_{i,k}$ is the unique path such that:
\begin{align}\text{Type A}& :  &(N+1-i,k+N+1) &\in \plow_{i,k}  \nn\\
\text{Type B},&\,\, i<N:  &\iota(i,k+4N-2) &\in \plow_{i,k} \nn\\
\text{Type B},&\,\, i=N:  &(2N-1,k+4N-2) + (0,\eps)&\in \plow_{N,k} .\nn\end{align}  

\subsection{Non-overlapping paths}\label{overlapdef}
Let $p,p'$ be paths. We say $p$ is \emph{strictly above} $p'$, and $p'$ is \emph{strictly below} $p$, if and only if
\be  (x,y)\in p \text{ and } (x,z)\in p'  \implies  y< z .\nn\ee
We say a $T$-tuple of paths $(p_1,\dots,p_T)$ is \emph{non-overlapping} if and only if $p_s$ is strictly above $p_t$ for all $s<t$. Otherwise, for some $s<t$ there exist $(x,y)\in p_s$ and $(x,z)\in p_t$ such that $y\geq z$, and we say \emph{ $p_s$ overlaps $p_t$ in column $x$}. 

For any snake $(i_t,k_t)\in \It$, $1\leq t \leq T$, $T\in \Z_{\geq 1}$, let us define 
\be \nops := \left\{(p_1,\dots,p_T): p_t\in \scr P_{i_t,k_t}, 1\leq t\leq T\,,  (p_1,\dots,p_T)\text{ is non-overlapping } \right\}. \nn\ee
\begin{lem}\label{disjcorners}
Let $(i_t,k_t)\in \It$, $1\leq t \leq T$, be a snake of length $T\in \Z_{\geq 1}$ and $(p_1,\dots,p_T) \in \nops$.
Suppose  $(i,k)\in C_{p_t,\pm}$ for some $t$, $1\leq t\leq T$. Then
\begin{enumerate}[(i)]
\item $(i,k)\notin C_{p_s,\pm}$ for any $s\neq t$, $1\leq s \leq T$, and
\item 
  (Type {\rm A}) $(i,k)\notin C_{p_s,\mp}$ for any $s$, $1\leq s \leq T$ \\
  (Type {\rm B}) if $(i,k)\in C_{p_s,\mp}$ for some $s$, $1\leq s\leq T$, then $s=t\pm1$ and $i=N$.
\end{enumerate}
\end{lem}
\begin{proof}
This follows from the definition of non-overlapping paths. Examples of the last part (pairs non-overlapping paths in type B which share a common corner) are shown in Figure \ref{nocrossfigureBprime}.
\end{proof}
Thus, for any $\ps\in \nops$ and any $(j,\ell)\in \Iw$,  at most one of the paths can be lowered at $(j,\ell)$ and at most one of the paths can be raised at $(j,\ell)$.  
We can therefore speak without ambiguity of performing a raising or lowering move at $(j,\ell)$ on a non-overlapping tuple of paths $\ps\in\nops$, to yield a new tuple $\ps \scr A_{j,\ell}^{\pm 1}$.

Lowering moves on tuples of non-overlapping paths can introduce overlaps. However, we have the following lemma which gives information about how such overlaps arise.
\begin{lem}\label{firstXlem}
Let $(i_t,k_t)\in \It$, $1\leq t\leq T$, be a snake of length $T\in \Z_{\geq 1}$ and $\ps\in\nops$. Let $t$, $1\leq t\leq T$, and  $(j,\ell)\in \Iw$ be such that the path $p_t$ can be lowered at $(j,\ell)$. This move introduces an overlap if and only if there is an $s$, $t<s\leq T$, such that $p_s$ has an upper corner at $(j,\ell+r_j)$ or a lower corner at $(j,\ell-r_j)$.

Similarly, let $t$, $1\leq t\leq T$, and  $(j,\ell)\in \Iw$ be such that the path $p_t$ can be raised at $(j,\ell)$. This move introduces an overlap if and only if there is an $s$, $1\leq s<t$, such that $p_s$ has a lower corner at $(j,\ell-r_j)$ or an upper corner at $(j,\ell+r_j)$. 
\end{lem}
\begin{proof}
This is seen by inspection of the definitions above of paths and moves. We sketch the distinct cases, up to symmetry, in type B. In most cases the overlap occurs at an upper corner $(j,\ell+r_j)$ of $p_s$:
\be\nn
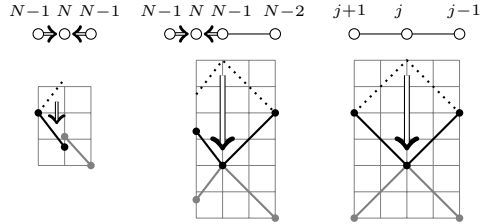
\begin{tikzpicture}[baseline=0cm,scale=.35,yscale=-1]
\begin{scope}[xshift=12cm]
\draw[help lines] (-1,0) grid (3,6);
\draw (-1,-1) -- (3,-1);
\filldraw[fill=white] (-1,-1) circle (2mm) node[above=3mm,left=-3mm] {$\scriptscriptstyle j+1$};
\filldraw[fill=white] (1,-1) circle (2mm) node[above=3mm,right=-3mm] {$\scriptscriptstyle j$};
\filldraw[fill=white] (3,-1) circle (2mm) node[above=3mm,right=-3mm] {$\scriptscriptstyle j-1$};
\draw[thick,dotted]  (-1,2) -- (1,0) -- (3,2) ;
\begin{scope}[every node/.style={minimum size=.1cm,inner sep=0mm,fill,circle}]
\draw[thick,gray] (-1,6)  node {} -- (1,4) node {}  -- (3,6)node {} ;
\draw[thick]  (-1,2) node {} -- (1,4) node {} -- (3,2) node {}  ;
\end{scope}

\draw[double distance=.5mm,->,shorten <= 2mm,shorten >= 2mm] (1,0) -- (1,4);

\end{scope}
\begin{scope}[xshift=5cm]
\draw[help lines] (0,0) grid (3,6);
\draw[double distance =.2mm,->,shorten >= 1mm] (-1,-1) -- (0,-1);
\draw[double distance =.2mm,->,shorten >= 1mm] (1,-1) -- (0,-1);
\draw (1,-1) -- (3,-1);
\filldraw[fill=white] (-1,-1) circle (2mm) node[above=3mm,left=-3mm] {$\scriptscriptstyle N-1$};
\filldraw[fill=white] (0,-1) circle (2mm) node[above=1mm] {$\scriptscriptstyle N$};
\filldraw[fill=white] (1,-1) circle (2mm) node[above=3mm,right=-3mm] {$\scriptscriptstyle N-1$};
\filldraw[fill=white] (3,-1) circle (2mm) node[above=3mm,right=-3mm] {$\scriptscriptstyle N-2$};
\draw[thick,dotted]  (0,1.3) -- (1,0) -- (3,2) ;
\begin{scope}[every node/.style={minimum size=.1cm,inner sep=0mm,fill,circle}]
\draw[thick,gray] (0,5.3) node {} -- (1,4) node {}   -- (3,6)node {} ;
\draw[thick]  (0,2.7)node {}  -- (1,4)node {}  -- (3,2)node {}  ;
\end{scope}
\draw[double distance=.5mm,->,shorten <= 2mm,shorten >= 2mm] (1,0) -- (1,4);
\end{scope}
\begin{scope}[xshift=0cm]
\draw[help lines] (-1,1) grid (1,4);
\draw[double distance =.2mm,->,shorten >= 1mm] (-1,-1) -- (0,-1);
\draw[double distance =.2mm,->,shorten >= 1mm] (1,-1) -- (0,-1);
\filldraw[fill=white] (-1,-1) circle (2mm) node[above=3mm,left=-3mm] {$\scriptscriptstyle N-1$};
\filldraw[fill=white] (0,-1) circle (2mm) node[above=1mm] {$\scriptscriptstyle N$};
\filldraw[fill=white] (1,-1) circle (2mm) node[above=3mm,right=-3mm] {$\scriptscriptstyle N-1$};
\draw[thick,dotted]  (-1,2) -- (0,.7) ;
\begin{scope}[every node/.style={minimum size=.1cm,inner sep=0mm,fill,circle}]
\draw[thick,gray]  (0,2.9)node {}  -- (1,4)node {}  ;
\draw[thick]  (-1,2) node {} -- (0,3.3)node {}  ;
\end{scope}
\draw[double,->,shorten <= 2mm,shorten >= 2mm] (-.3,1) -- (-.3,3);
\end{scope}
\end{tikzpicture}
\nn.\ee
The exception is when the upper corner $(j,\ell-r_j)$ of $p_t$ is also a lower corner of $p_s$, which can happen only when $j=N$ -- \confer Lemma \ref{disjcorners} and Figure \ref{nocrossfigureBprime}: 
\be
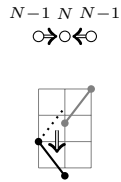
\begin{tikzpicture}[baseline=0cm,scale=.35,yscale=-1]
\begin{scope}[xshift=0cm]
\draw[help lines] (-1,1) grid (1,4);
\draw[double distance =.2mm,->,shorten >= 1mm] (-1,-1) -- (0,-1);
\draw[double distance =.2mm,->,shorten >= 1mm] (1,-1) -- (0,-1);
\filldraw[fill=white] (-1,-1) circle (2mm) node[above=3mm,left=-3mm] {$\scriptscriptstyle N-1$};
\filldraw[fill=white] (0,-1) circle (2mm) node[above=1mm] {$\scriptscriptstyle N$};
\filldraw[fill=white] (1,-1) circle (2mm) node[above=3mm,right=-3mm] {$\scriptscriptstyle N-1$};
\draw[thick,dotted]  (-1,3) -- (0,1.7) ;

\draw[double,->,shorten <= 2mm,shorten >= 2mm] (-.3,2) -- (-.3,4);

\begin{scope}[every node/.style={minimum size=.1cm,inner sep=0mm,fill,circle}]
\draw[thick,gray]  (0,2.3) node {} -- (1,1) node {} ;
\draw[thick]  (-1,3) node {} -- (0,4.3) node {} ;
\end{scope}

\end{scope}
\end{tikzpicture}.
\nn\ee

\end{proof}

\begin{rem}
This property of our tuples of paths is vital for the validity of Lemma \ref{inthinsimplelem}  and hence of Theorem \ref{snakechar} below. Informally speaking, it means that the first overlap between paths always corresponds, in the $q$-character, to an illegal lowering step in some $\uqslt$ evaluation module. That is, in the ball-and-box picture of thin $\uqslt$-modules (\S\ref{boxpic}) it always corresponds to trying to perform an illegal move $\bb\bb\to\eb\eb$ or $\eb\eb\to\wb\wb$. 
It should be noted that Lemma \ref{firstXlem} would not hold if the definition  of snake position (\S\ref{snakepos}) in type B were widened by dropping the conditions on $k-k'\mod 4$. For example, consider $((N,1), (N,5))$. The first overlap between in a path in $\scr P_{(N,1)}$ and a path in $\scr P_{(N,5)}$ is not necessarily of the type in Lemma \ref{firstXlem}, \be\nn
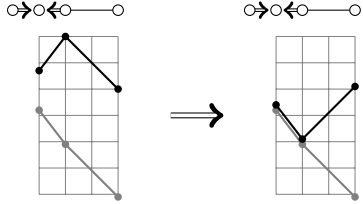
\begin{tikzpicture}[baseline=0cm,scale=.35,yscale=-1,xscale=-1]
\draw[help lines] (-3,-2) grid (0,4);
\draw[double distance =.2mm,->,shorten >= 1mm] (-1,-3) -- (0,-3);
\draw[double distance =.2mm,->,shorten >= 1mm] (1,-3) -- (0,-3);
\draw (-3,-3) -- (-1,-3);
\filldraw[fill=white] (-3,-3) circle (2mm);
\filldraw[fill=white] (-1,-3) circle (2mm);
\filldraw[fill=white] (0,-3) circle (2mm) ;
\filldraw[fill=white] (1,-3) circle (2mm) ;
\begin{scope}[every node/.style={minimum size=.1cm,inner sep=0mm,fill,circle}]
\draw[thick,gray]  (-3,4.1)node {}  -- (-1,2.1) node {} -- (0,.8)node {} ;
\draw [thick] (-3,0)node {}  -- (-1,-2)node {}  -- (0,-.7)node {} ;
\end{scope}
\draw[double distance=.5mm,->] (-5,1) -- (-7,1);
\begin{scope}[xshift=-9cm]
\draw[help lines] (-3,-2) grid (0,4);
\draw[double distance =.2mm,->,shorten >= 1mm] (-1,-3) -- (0,-3);
\draw[double distance =.2mm,->,shorten >= 1mm] (1,-3) -- (0,-3);
\draw (-3,-3) -- (-1,-3);
\filldraw[fill=white] (-3,-3) circle (2mm);
\filldraw[fill=white] (-1,-3) circle (2mm);
\filldraw[fill=white] (0,-3) circle (2mm) ;
\filldraw[fill=white] (1,-3) circle (2mm) ;


\begin{scope}[every node/.style={minimum size=.1cm,inner sep=0mm,fill,circle}]
\draw[thick,gray]  (-3,4.1) node {} -- (-1,2.1)node {}  -- (0,.8)node {} ;
\draw[thick]  (-3,-.1)node {}  -- (-1,1.9)node {}  -- (0,.6)node {}  ;
\end{scope}

\end{scope}
\end{tikzpicture}\ee
and could thus correspond to a legal lowering step $\bb\eb \to \eb\wb$ in the $q$-character. \end{rem}

Our definitions of paths and moves are so constructed that we have 
\begin{lem}\label{steplemma}
Let $(i_t,k_t)\in \It$, $1\leq t\leq T$, be a snake of length $T\in \Z_{\geq 1}$ and $\ps\in\nops$. If $\pps = \ps \scr A_{j,\ell}^{\pm 1}$, where $(j,\ell)\in\Iw$ is any point at which $\ps$ can be raised/lowered, then $\prod_{t=1}^T \mon(p'_t) = A_{j,\ell}^{\pm1}\prod_{t=1}^T \mon(p_t)$. 
\qed\end{lem}

\subsection{Properties of moves and paths} 
The rest of this section is concerned with establishing some facts about the relationship between paths and monomials.

Given any two paths $p=(x_r,y_r)_{1\leq r\leq n}$ and $p'=(x_r,y'_r)_{1\leq r\leq n}$, $n\in\Z_{>0}$, in $\scr P_{i,k}$ we say $p$ is \emph{weakly above} (resp. \emph{weakly below}) $p'$ if and only if $y_r\leq y'_r$ (resp. $y_r\geq y'_r$) for all $1\leq r\leq n$. We also define 
\be \topp(p,p') :=  (x_r, \min(y_r,y'_r))_{1\leq r\leq n}.\ee

The three following lemmas are proved by inspection. 
\begin{lem}\label{lcl} Any path $p\in \scr P_{i,k}$ is uniquely defined by its set of lower corners. \qed\end{lem}

\begin{lem}\label{wbr} Suppose $p$ and $p'$ are paths in $\scr P_{i,k}$ such that $p$ is weakly below $p'$. If $p\neq p'$ then there is a $(j,\ell)\in \Iw$ such that $p$ can be raised at $(j,\ell)$ and $p \scr A_{j,\ell}$ is weakly below $p'$. 
\qed\end{lem}
\begin{lem} For all $p,p'\in \scr P_{i,k}$, $\topp(p,p')\in \scr P_{i,k}$ and $\topp(p,p')$ is weakly above both $p$ and $p'$.
\qed\end{lem}

Now we have
\begin{lem}\label{movelemmaA} Let $p$ and $p'$ be paths in $\scr P_{i,k}$. Then
$p$ can be obtained from $p'$ by a sequence of moves containing no inverse pair of raising/lowering moves. 
\end{lem}
\begin{proof}
By applying Lemma \ref{wbr} a finite number of times we construct a sequence $s$ of distinct points in $\Iw$ such that, starting with $p$, performing raising moves at these points, in order, yields $\topp(p,p')$. Similarly, we construct a sequence $s'$ of raising moves taking $p'$ to $\topp(p,p')$; by reversing this sequence, we have a sequence of lowering moves taking $\topp(p,p')$ to $p'$. It is enough to show that $s$ and $s'$ have no element in common. 

Suppose for a contradiction the point $(j,\ell)\in \Iw$ occurs in both $s$ and $s'$. Let $\tilde p$ be the path obtained by performing, on $p$, raising moves at the points in $s$ preceding $(j,\ell)$, in order. Similarly, let $\tilde p'$ be the path obtained by performing, on $p'$, raising moves at the points in $s'$ preceding $(j,\ell)$, in order. 
Then $\topp(p,p') = \topp(\tilde p,\tilde p')$. But both $\tilde p$ and $\tilde p'$ have a lower corner at $(j,\ell+r_j)$. Therefore $\topp(\tilde p,\tilde p')$ has a lower corner at $(j,\ell+r_j)$, while $\topp(p,p')$ does not: a contradiction. 
\end{proof}


\begin{lem}\label{movelemmaB} Let $(i_t,k_t)\in \It$, $1\leq t\leq T$, be a snake of length $T\in \Z_{\geq 1}$. Let $\ps$ and $\pps$ be tuples of non-overlapping paths in $\nops$. Then $\ps$ can be obtained from $\pps$ by a sequence of moves containing no inverse pair of raising/lowering moves,  such that no move introduces any overlaps.  
\end{lem}
\begin{proof}
For each $t$, $1\leq t\leq T$, let $s_t$ (resp. $s'_t$) be the finite sequence of points at which we must perform raising moves on $p_t$ (resp. $p'_t$) to reach $\topp(p_t,p'_t)$, by repeated application of Lemma \ref{wbr}.  Then the following is a sequence of moves obeying the requirements of the lemma. We first perform raising moves (at the points in $s_1$, in order) on $p_1$ to reach $\topp(p_1,p_1')$, then on $p_2$ (at the points in $s_2$, in order) to reach $\topp(p_2,p_2')$ and so on until we raise $p_T$ to $\topp(p_T,p'_T)$. Then we perform lowering moves (at the points in $s_T'$, in reverse order) on $\topp(p_T,p'_T)$ to reach $p'_T$, then on $\topp(p_{T-1},p_{T-1}')$ (at the points in $s_{T-1}'$, in reverse order) to $p_{T-1}'$ and so on until we lower $\topp(p_1,p_1')$ to $p_1$. This ordering of the moves ensures that at no step do any two paths overlap. By Lemma \ref{movelemmaA}, $s_t\cap s_t'=\emptyset$ for all $t$, $1\leq t\leq T$. It remains to check that $s_t\cap s_u'=\emptyset$ for all distinct $1\leq t \neq u \leq T$. Suppose for a contradiction that $(j,\ell)\in s_t\cap s_u'$, for some $1\leq t\neq  u\leq T$. Without loss of generality suppose $t<u$. Let $(x,y):= \iota(j,\ell+r_j)$. Then $p'_u$ has a lower corner at $(j,\ell+r_j)$ and $\topp(p_u,p'_u)$ does not. So $\topp(p_u,p_u')$, and hence $p_u$, contains a point $(x,y')$, $y'<y$. But $p_t$ also has a lower corner at $(j,\ell+r_j)$, i.e. it contains the point $(x,y)$. So $p_u$ and $p_t$ overlap in column $x$: a contradiction since $p_t$ and $p_u$ are non-overlapping.   
\end{proof}

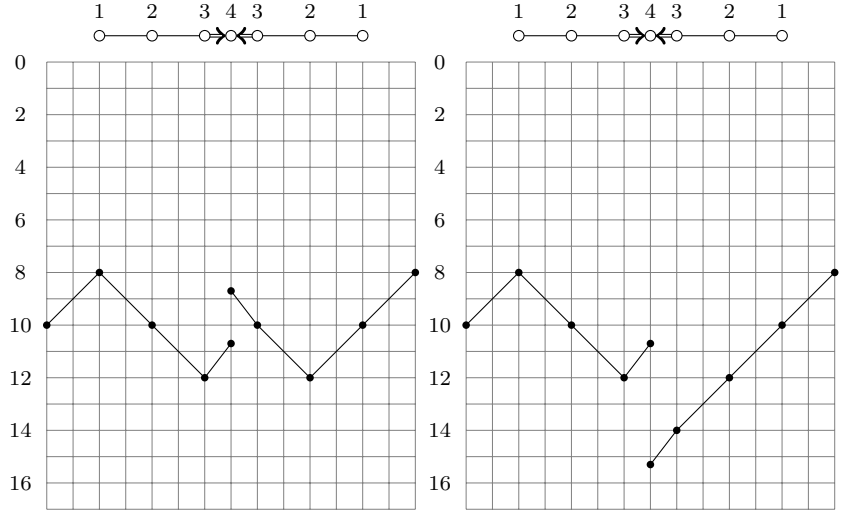
\begin{figure}
\caption{Illustration of the definition, \S\ref{overlapdef}, of overlapping paths in type $B_4$. By Theorem \ref{snakechar}, $\L (Y_{4,1} Y_{4,3})$ contains the monomial $(Y_{1,8} Y_{3,12}^{-1} Y_{4,11})(Y_{4,9}Y_{2,12}^{-1})$ (left) but \emph{not} the monomial  $(Y_{1,8} Y_{3,12}^{-1} Y_{4,11})(Y_{4,15}^{-1})$ (right). Note that the latter is disallowed despite the fact that no paths intersect, because the would-be upper path contains, in the column corresponding to the short node, a point below a point of the would-be lower path.\label{nocrossfigureB}}
\be\begin{tikzpicture}[scale=.35,yscale=-1]
\draw[help lines] (0,0) grid (14,17);
\draw (2,-1) -- (6,-1); \draw (8,-1) -- (12,-1);
\draw[double,->] (6,-1) -- (6.8,-1); \draw[double,->] (8,-1) -- (7.2,-1);
\filldraw[fill=white] (7,-1) circle (2mm) node[above=1mm] {$\scriptstyle 4$};
\foreach \x in {1,2,3} {
\filldraw[fill=white] (2*\x,-1) circle (2mm) node[above=1mm] {$\scriptstyle\x$}; 
\filldraw[fill=white] (2*7-2*\x,-1) circle (2mm) node[above=1mm] {$\scriptstyle\x$}; }
\foreach \y in {0,2,4,6,8,10,12,14,16} {\node at (-1,\y) {$\scriptstyle\y$};}
\begin{scope}[every node/.style={minimum size=.1cm,inner sep=0mm,fill,circle}]
\draw (7,8.7) node{} --(8,10) node{}  -- (10,12) node{} -- (12,10) node{}  -- (14,8) node{} ;
\draw (7,10.7) node{}  -- (6,12) node{} -- (4,10) node{}  -- (2,8) node{}  -- (0,10) node{} ;
\end{scope}
\end{tikzpicture}
\begin{tikzpicture}[scale=.35,yscale=-1]
\draw[help lines] (0,0) grid (14,17);
\draw (2,-1) -- (6,-1); \draw (8,-1) -- (12,-1);
\draw[double,->] (6,-1) -- (6.8,-1); \draw[double,->] (8,-1) -- (7.2,-1);
\filldraw[fill=white] (7,-1) circle (2mm) node[above=1mm] {$\scriptstyle 4$};
\foreach \x in {1,2,3} {
\filldraw[fill=white] (2*\x,-1) circle (2mm) node[above=1mm] {$\scriptstyle\x$}; 
\filldraw[fill=white] (2*7-2*\x,-1) circle (2mm) node[above=1mm] {$\scriptstyle\x$}; }
\foreach \y in {0,2,4,6,8,10,12,14,16} {\node at (-1,\y) {$\scriptstyle\y$};}
\begin{scope}[every node/.style={minimum size=.1cm,inner sep=0mm,fill,circle}]
\draw (7,15.3) node{}  -- (8,14) node{}  -- (10,12) node{}-- (12,10) node{}  -- (14,8) node{} ;
\draw (7,10.7) node{}  -- (6,12) node{} -- (4,10) node{}  -- (2,8) node{}  -- (0,10) node{} ;
\end{scope}
\end{tikzpicture}
\nn\ee
\end{figure}
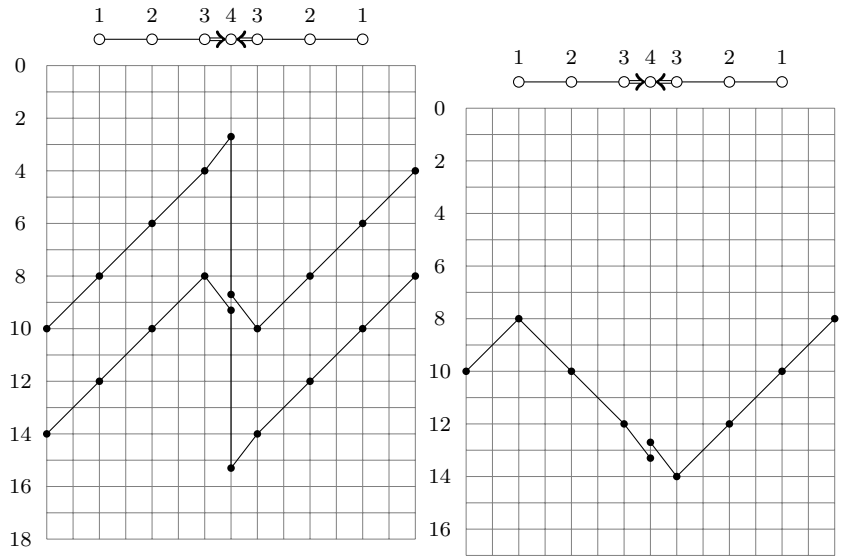
\begin{figure}
\caption{Illustration of Lemma \ref{disjcorners}. By Theorem \ref{snakechar}: $\L (Y_{2,0} Y_{2,4})$ contains the monomial $(Y_{4,3}Y_{4,9}Y_{3,10}^{-1})(Y_{3,8} Y_{4,9}^{-1}Y_{4,15}^{-1})=Y_{4,3}Y_{3,10}^{-1}Y_{3,8}Y_{4,15}^{-1}$ (left) and $\L (Y_{4,1} Y_{4,3})$ contains the monomial $(Y_{3,14}^{-1}Y_{4,13})(Y_{1,8}Y_{4,13}^{-1})=Y_{1,8}Y_{3,14}^{-1}$ (right).
In such cases it is possible for a factor $Y^{-1}$ from a given path cancel a factor $Y$ in a path strictly above. But in no case can a $Y^{-1}$ from a given path cancel a $Y$ from a path strictly below. 
\label{nocrossfigureBprime}}
\be\nn\begin{tikzpicture}[scale=.35,yscale=-1]
\draw[help lines] (0,0) grid (14,18);
\draw (2,-1) -- (6,-1); \draw (8,-1) -- (12,-1);
\draw[double,->] (6,-1) -- (6.8,-1); \draw[double,->] (8,-1) -- (7.2,-1);
\filldraw[fill=white] (7,-1) circle (2mm) node[above=1mm] {$\scriptstyle 4$};
\foreach \x in {1,2,3} {
\filldraw[fill=white] (2*\x,-1) circle (2mm) node[above=1mm] {$\scriptstyle\x$}; 
\filldraw[fill=white] (2*7-2*\x,-1) circle (2mm) node[above=1mm] {$\scriptstyle\x$}; }
\foreach \y in {0,2,4,6,8,10,12,14,16,18} {\node at (-1,\y) {$\scriptstyle\y$};}

\begin{scope}[every node/.style={minimum size=.1cm,inner sep=0mm,fill,circle}]
\draw (14,4) node{}  -- (12,6) node{}  -- (10,8) node{}  -- (8,10) node{}  -- (7,8.7) node{}  --(7,2.7) node{}  -- (6,4) node{}  -- (4,6) node{}  -- (2,8) node{} -- (0,10) node{} ;
\draw (14,8) node{}  -- (12,10) node{}  -- (10,12) node{}  -- (8,14) node{}  -- (7,15.3) node{}  --(7,9.3) node{}  -- (6,8) node{}  -- (4,10) node{}  -- (2,12) node{}  -- (0,14) node{} ;
\end{scope}

\end{tikzpicture}
\begin{tikzpicture}[scale=.35,yscale=-1]
\draw[help lines] (0,0) grid (14,17);
\draw (2,-1) -- (6,-1); \draw (8,-1) -- (12,-1);
\draw[double,->] (6,-1) -- (6.8,-1); \draw[double,->] (8,-1) -- (7.2,-1);
\filldraw[fill=white] (7,-1) circle (2mm) node[above=1mm] {$\scriptstyle 4$};
\foreach \x in {1,2,3} {
\filldraw[fill=white] (2*\x,-1) circle (2mm) node[above=1mm] {$\scriptstyle\x$}; 
\filldraw[fill=white] (2*7-2*\x,-1) circle (2mm) node[above=1mm] {$\scriptstyle\x$}; }
\foreach \y in {0,2,4,6,8,10,12,14,16} {\node at (-1,\y) {$\scriptstyle\y$};}

\begin{scope}[every node/.style={minimum size=.1cm,inner sep=0mm,fill,circle}]
\draw (7,12.7) node{}  -- (8,14) node{}  -- (10,12) node{}  -- (12,10) node{}  -- (14,8) node{} ;
\draw (7,13.3) node{}  -- (6,12) node{}  -- (4,10) node{}  -- (2,8) node{}  -- (0,10) node{} ;
\end{scope}

\end{tikzpicture}
\ee
\end{figure}


\begin{lem}\label{movelemma}
Let $(i_t,k_t)\in \It$, $1\leq t\leq T$, be a snake of length $T\in \Z_{\geq 1}$ and $(j_r,\ell_r)$, $1\leq r\leq R$,  a sequence of $R\in\Z_{\geq 0}$ points in $\Iw$. For all $\ps\in\nops$ and $\pps\in\nops$, the following are equivalent: 
\begin{enumerate}[(i)] 
\item $\prod_{t=1}^T \mon(p'_t) =  \prod_{t=1}^T \mon(p_t) \cdot \prod_{r=1}^R A_{j_r,\ell_r}^{-1}$ 
\item there is a permutation $\sigma \in S_R$ such that $\big( (j_{\sigma(1)},\ell_{\sigma(1)}),\dots, (j_{\sigma(R)},\ell_{\sigma(R)}) \big)$ is a sequence of lowering moves that can be performed on $\ps$, without ever introducing overlaps, to yield $\pps$. 
\end{enumerate}
\end{lem}
\begin{proof}
That (ii) implies (i) follows from Lemma \ref{steplemma}. To see that (i) implies (ii), note that since $\ps\in\nops$ and $\pps\in\nops$, Lemma \ref{movelemmaB} states that there is a sequence of moves that takes $\ps$ to $\pps$ without introducing overlaps and without ever performing a move and its inverse. By Lemma \ref{steplemma} and the fact that the $(A_{j,\ell})_{(j,\ell)\in \Iw}$ are algebraically independent, these moves must indeed be lowering moves at the points $(j_r,\ell_r)_{1\leq r\leq R}$ arranged in some order. 
\end{proof}
\begin{cor}\label{injmap} Let $(i_t,k_t)\in \It$, $1\leq t\leq T$, be a snake of length $T\in \Z_{\geq 1}$. The map
\be\nn\nops \to  \Z[Y_{j,\ell}^{\pm 1}]_{(j,\ell)\in \It};\qquad \ps \mapsto \prod_{t=1}^T \mon(p_t)\ee
is injective.
\end{cor}
\begin{proof}
This is the case $R=0$ in Lemma \ref{movelemma}.
\end{proof}

\begin{lem}\label{updownlem} 
Let $(i_t,k_t)\in \It$, $1\leq t\leq T$, be a snake of length $T\in \Z_{\geq 1}$ and $\ps\in\nops$. Let $m= \prod_{t=1}^T \mon(p_t)$. If $mA_{i,k}^{-1}$ not of the form $\prod_{t=1}^T\mon(p'_t)$ for any $\pps\in\nops$, then neither is $mA_{i,k}^{-1}A_{j,\ell}$ unless $(j,l)=(i,k)$.
\end{lem}
\begin{proof} 
By Lemma \ref{movelemma}, if $mA_{i,k}^{-1}A_{j,\ell}=\prod_{t=1}^T \mon(p_t')$ for some $\pps\in\nops$ then $\pps$ is can be obtained from $\ps$, without ever introducing overlaps, by either
\begin{enumerate}[(i)]
\item lowering at $(i,k)$ and then raising at $(j,\ell)$, or 
\item raising at $(j,\ell)$ and then lowering at $(i,k)$.
\end{enumerate}
We assume that $mA_{i,k}^{-1}$ is not of the form $\prod_{t=1}^T\mon(p'_t)$ for any $\pps\in\nops$. So, by Lemma \ref{steplemma}, either (a) it is not possible to lower $\ps$ at $(i,k)$, or else (b) this is a valid lowering move but one which introduces an overlap. Hence (i) is impossible. Now suppose (ii) is possible: suppose the raising move at $(j,\ell)$ is on $p_s$ and the lowering move at $(i,k)$ is on $p_t$, $1\leq s,t\leq T$. If $s\neq t$ then this requires that $p_t$ can be lowered at $(i,k)$, so we must be in case (b): i.e. there is an $r$ such that when $p_t$ is lowered at $(i,k)$ it overlaps with $p_r$. If $s\notin\{t,r\}$ then, after the raising move at $(j,\ell)$ on $p_s$, it is still true that when $p_t$ is lowered at $(i,k)$ it overlaps with $p_r$. If $s=r$ then note that $p_t$ lowered at $(i,k)$ overlaps with $p_r$ raised at $(j,\ell)$. 
Therefore in fact $s=t$, i.e. both moves must be on the same path. 
%
One then sees, on inspection, that it is necessary that $(j,\ell) = (i,k)$, as required. 
\end{proof}

\begin{lem}\label{inthinsimplelem} 
Let $(i_t,k_t)\in \It$, $1\leq t\leq T$, be a snake of length $T\in \Z_{\geq 1}$ and $\ps\in\nops$. Pick and fix an $i\in I$.
Let $\overline{\scr P} \subset \nops$ be the set of those non-overlapping tuples of paths that can be obtained from $\ps$ by performing a sequence of raising or lowering moves at points of the form $(i,\ell)\in\Iw$. Then  $\beta_i(\prod_{t=1}^T \mon(p_t))$ is the $q$-character of a thin simple  finite-dimensional $\uqslt$-module.
\end{lem}
\begin{proof}
Let $m:=\prod_{t=1}^T \mon(p_t)$. The monomial $\beta_i(m)$ satisfies the conditions given in Lemma \ref{thinsimplesl2lem} to be part of a thin simple finite-dimensional $\uqsl i$-module, let us call it $V$: the condition $|u_{i,\ell}(m)| \leq 1$ is by Lemma \ref{disjcorners}; the condition $u_{i,\ell-r_i}(m) - u_{i,\ell+r_i}(m) \neq 2$ holds because there cannot be an upper corner at $(i,\ell-r_i)$ and a lower corner at $(i,\ell+r_i)$ without paths overlapping. (So we have that, in the language of the ball-and-box model of \S\ref{boxpic},  the pattern $\bb\wb$ never occurs.) 

Now we argue that the elements of $\overline{\scr P}$ are in bijection with the set $\mchiq(V)$ of monomials of $\chi_q(V)$.

By the definition of lowering moves and Lemma \ref{firstXlem}, it is possible to lower $\ps$ at $(i,\ell)$ if and only if 
\begin{enumerate}[(i)]
\item $(i,\ell-r_i)$ is an upper corner of some path in $\ps$ \emph{and} 
\item $(i,\ell-r_i)$ is not a lower corner of any path in $\ps$, \emph{and} 
\item $(i,\ell+r_i)$ is not an upper corner of any path in $\ps$.\end{enumerate}
Lemma \ref{movelemma} states that the only way to produce a factor $A_{i,\ell}^{-1}$ is to do the lowering move at $(i,\ell)$. 
Hence
\be mA_{i,\ell}^{-1}\in\sum_{\pps \in \overline{ \scr P}} \prod_{t=1}^T \mon(p'_t)\ee if and only if (i), (ii) and (iii) hold; that is, \confer (\ref{mondef}), if and only if $u_{i,\ell-r_i}(m) = 1$ and $u_{i,\ell+r_i}(m)=0$. As Lemma \ref{thinsimplesl2lem} states, these are precisely the conditions under which
\be \beta_i(mA_{i,\ell}^{-1}) \in \mchiq(V).\ee
Similar statements hold for raising moves.
As we noted following Lemma \ref{thinsimplesl2lem}, by a finite sequence of moves of this type we obtain every monomial of $\chi_q(V)$ and no others. We also, by definition, generate all the elements of $\overline{\scr P}$ and no other tuples of paths. 
\end{proof}
\begin{lem}\label{pathsspeciallem}
Let $(i_t,k_t)\in \It$, $1\leq t\leq T$, be a snake of length $T\in \Z_{\geq 1}$ and $\ps\in\nops$. 

If $\ps\neq\ptop$ then $\prod_{t=1}^T \mon(p_t)$ is not dominant. 

If $\ps\neq\pbot$ then $\prod_{t=1}^T \mon(p_t)$ is not anti-dominant.
\end{lem}
\begin{proof}
If $\ps\neq\ptop$ then there exists $t$, $1\leq t\leq T$, such that $p_t\neq \phigh_{i_t,k_t}$. Therefore $p_t$ has at least one lower corner at, \confer \S\ref{sec:highestpath}. Suppose it is at $(i,k)\in \It$. If no path in $\ps$ has an upper corner at $(i,k)$ then  $\prod_{t=1}^T \mon(p_t)$ is not dominant. Lemma \ref{disjcorners} states that the only way another path can have an upper corner at $(i,k)$ without there being any overlaps is if $i=N$ and $s=t-1$. Since $s$ precedes $t$ in the snake, $p_s$ cannot equal $\phigh_{i_s,k_s}$ and must have a lower corner at some $(i',k')\in \It$, $i'\neq N$. No path in $\ps$ can have an upper corner here, again by Lemma \ref{disjcorners}. So the resulting factor $Y_{i',k'}^{-1}$ cannot be cancelled and   $\prod_{t=1}^T \mon(p_t)$ is not dominant.
The argument for anti-dominant monomials is similar.
\end{proof}

\section{The path formula for $q$-characters of snake modules}\label{sec:snakechar}
We are now in a position to state the second main result of the paper.
\begin{thm}
\label{snakechar} Let $(i_t,k_t)\in \It$, $1\leq t\leq T$, be a snake of length $T\in \Z_{\geq 1}$. Then
\be\chi_q\left(\L\left(\prod_{t=1}^T Y_{i_t,k_t}\right)\right) = \sum_{\displaystyle \substack{(p_1,\dots,p_T)\in \nops}} 
\prod_{t=1}^T \mon(p_t).\label{sncht}\ee
The module $\L(\prod_{t=1}^T \YY {i_t}{k_t})$ is thin, special and anti-special.
\end{thm}
\begin{proof}
Let $m_+:=
\prod_{t=1}^Tm(\phigh_{i_t,k_t})$ be the highest monomial of  $\chi_q(\L(\prod_{t=1}^T \YY {i_t}{k_t}))$. Define
\be \mc M:= \{ \prod_{t=1}^T \mon(p_t) : \ps \in \nops \}.\nn\ee
We shall show that the conditions of Theorem \ref{thmA} apply to the pair $(m_+,\mc M)$. 
Indeed, Property (\ref{monlydom}) follows from Lemma \ref{pathsspeciallem},
Property (\ref{onewayback}) is Lemma \ref{updownlem}, and
Property (\ref{inthinsimple}) is Lemma \ref{inthinsimplelem}. Since, by Corollary \ref{injmap}, $\sum_{m\in \mc M} m = \sum_{\ps\in\nops}\prod_{t=1}^T \mon(p_t)$, the formula (\ref{sncht}) and the properties thin and special follow from Theorem \ref{thmA}. Anti-special then follows from Lemma \ref{pathsspeciallem}.  
\end{proof}

\begin{exmp}\label{notab}
In type $B_5$, consider the simple module $\L(Y_{4,0}Y_{5,5} Y_{4,10})$. This is a minimal snake module; its snake is shown on the left in Figure \ref{notabfig}. By Theorem \ref{snakechar}, its $q$-character includes the monomial
\be\nn Y_{2,12}^{-1}Y_{5,7}Y_{5,1} Y_{5,13} Y_{4,14}^{-1} Y_{2,10} Y_{1,22}^{-1} Y_{5,15} Y_{5,17} Y_{4,18}^{-1} Y_{2,14},\ee 
which corresponds to the non-overlapping tuple of paths shown on the right in Figure \ref{notabfig}.
\end{exmp}

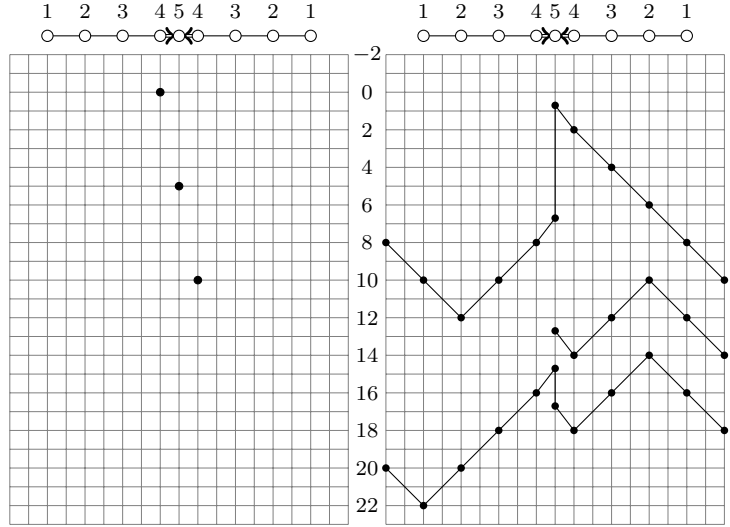
\begin{figure}
\caption{\label{notabfig} Illustration of Theorem \ref{snakechar} in type $B_5$. Left: a minimal snake. Right: a tuple of non-overlapping paths for this snake. See Example \ref{notab}.}
\be\nn
\begin{tikzpicture}[scale=.25,yscale=-1]
\draw[help lines] (0,0) grid (18,25);
\draw (2,-1) -- (8,-1); \draw (10,-1) -- (16,-1);
\draw[double,->] (8,-1) -- (8.8,-1); \draw[double,->] (10,-1) -- (9.2,-1);
\filldraw[fill=white] (9,-1) circle (3mm) node[above=1mm] {$\scriptstyle 5$};
\foreach \x in {1,2,3,4} {
\filldraw[fill=white] (2*\x,-1) circle (3mm) node[above=1mm] {$\scriptstyle\x$}; 
\filldraw[fill=white] (2*9-2*\x,-1) circle (3mm) node[above=1mm] {$\scriptstyle\x$}; }
\foreach \x/\y in {8/0,9/5,10/10}{\filldraw (\x,\y+2) circle (2mm);}
\begin{scope}[xshift=20cm]
\draw[help lines] (0,0) grid (18,25);
\draw (2,-1) -- (8,-1); \draw (10,-1) -- (16,-1);
\draw[double,->] (8,-1) -- (8.8,-1); \draw[double,->] (10,-1) -- (9.2,-1);
\filldraw[fill=white] (9,-1) circle (3mm) node[above=1mm] {$\scriptstyle 5$};
\foreach \x in {1,2,3,4} {
\filldraw[fill=white] (2*\x,-1) circle (3mm) node[above=1mm] {$\scriptstyle\x$}; 
\filldraw[fill=white] (2*9-2*\x,-1) circle (3mm) node[above=1mm] {$\scriptstyle\x$}; }
\foreach \y in {-2,0,2,4,6,8,10,12,14,16,18,20,22} {\node at (-1,\y+2) {$\scriptstyle\y$};}

\begin{scope}[every node/.style={minimum size=.1cm,inner sep=0mm,fill,circle}]
\draw (0,8+2) node{}  -- ++(2,2) node{}  -- ++(2,2) node{}  -- ++(2,-2) node{}  -- ++(2,-2) node{}  -- ++(1,-1.3) node{}  -- ++(0,-6) node{}  -- ++(1,1.3) node{}  --++(2,2) node{}  -- ++(2,2) node{}  -- ++(2,2) node{}  -- ++(2,2) node{}   ;
\draw (18,14+2) node{}  -- ++(-2,-2) node{}  -- ++(-2,-2) node{}  -- ++(-2,2) node{}  -- ++(-2,2) node{}  -- ++(-1,-1.3) node{}  ;
\draw (18,18+2) node{}  -- ++(-2,-2) node{}  -- ++(-2,-2) node{}  -- ++(-2,2) node{}  -- ++(-2,2) node{}  -- ++(-1,-1.3) node{}  -- ++(0,-2) node{}  -- ++(-1,1.3) node{}  -- ++(-2,2) node{}  -- ++(-2,2) node{}  -- ++(-2,2) node{}  -- ++(-2,-2) node{}   ;
\end{scope}

\end{scope}
\end{tikzpicture}
\ee
\end{figure}

\section{Existing results and skew tableaux}\label{sec:tableaux}
In types A and B, Theorem \ref{snakechar} gives a closed form, in terms of non-overlapping paths, for the $q$-character of the $\uqgh$-module $\L\left(\prod_{t=1}^T Y_{i_t,k_t}\right)$ where $(i_t,k_t)_{1\leq t\leq T}$ is any snake whose elements lie in the set $\It$, \confer \S\ref{sec:snakes}. 
Let us define a subset $\Iy\subseteq \It$ by:
\begin{enumerate}[Type A:]
\item $\Iy := \It$.
\item $\Iy := \{(i,k) \in I\times \Z : i<N \text{ and } 2i-k\equiv 2 \mod 4\}$.
\end{enumerate}
 
\begin{rem}\label{lrrem}In type B, the points of $\Iy$ all lie strictly to the left (resp. right) of the column corresponding to the short node, in the figures as drawn,  when $N\equiv 0\!\!\mod 2$ (resp. $N\equiv 1\!\!\mod 2$). \end{rem}

In the existing literature, closed forms have been found or conjectured 
for the $q$-character of $\L\left(\prod_{t=1}^T Y_{i_t,k_t}\right)$ whenever $(i_t,k_t)_{1\leq t\leq T}$ is a snake whose elements lie in $\Iy$.
In this section we describe how the existing closed forms can be recovered as special cases of Theorem \ref{snakechar}. The combinatorial objects used in these closed forms are (or are equivalent to) skew tableaux.

In this paper, let us say that a \emph{skew diagram} $\sd$ is finite subset $\sd \subset \Z\times\Z_{>0}$ such that 
\begin{enumerate}\item if $\sd\neq \emptyset$ then there is a $j$ such that $(j,1)\in\sd$, and
\item if $(i,j)\notin \sd$ then either $\forall i'\geq i, \forall j'\geq j$, $(i',j') \notin \sd$ or $\forall i'\leq i, \forall j'\leq j$, $(i',j') \notin \sd$.
\end{enumerate} 
If $(i,j)\in \sd$ we say $\sd$ has a \emph{box} in \emph{row} $i$, \emph{column} $j$. By definition the left-most (i.e. lowest-numbered) column containing a box is numbered 1. 
However, as a matter of convenience, we do not insist that the top-most (i.e. lowest-numbered) row containing a box is numbered 1, and we do not regard skew diagrams related by a vertical shift as identical.

Define a set $\Alp$ (the \emph{alphabet}) equipped with a total ordering $<$ (\emph{alphabetical ordering}) as follows:
\begin{enumerate}[Type A:]
\item $\Alp := I = \{1,2,\dots N,N+1\}$, $1<2<\dots <N<N+1$.
\item $\Alp := \{1,2,\dots,N,0,\bar N,\dots, \bar 2,\bar 1 \}$, $1<2<\dots<N<0<\bar N<\dots< \bar 2<\bar 1$. 
\end{enumerate}
A \emph{skew tableaux} $\T$ with shape $\sd$ is then any map $\T:\sd \to \Alp$ that obeys the following horizontal rule (H) and vertical rule (V):
\begin{enumerate}[Type A:]
\item \begin{enumerate} \item[(H)] $\T(i,j) \leq \T(i,j+1)$ \item[(V)] $\T(i,j)< \T(i+1,j)$.\end{enumerate}
\item \begin{enumerate} \item[(H)] $\T(i,j) \leq \T(i,j+1)$ and $(\T(i,j),\T(i,j+1)) \neq (0,0)$ 
                      \item[(V)] $\T(i,j) < \T (i+1,j)$ or $(\T(i,j),\T(i+1,j)) = (0,0).$ \end{enumerate} 
\end{enumerate}
In type A these are the usual rules for a semi-standard skew tableaux. In type B, see \cite{KOS}.

Let $\Tab\sd$ be the set of skew tableaux with shape $\sd$.

\begin{defn}\label{sddef}
Given a snake $(i_t,k_t)_{1\leq t\leq T}$ whose elements lie in $\Iy$, define a set
\be \sd_{\sn} := \left\{ (s,t)\in \Z\times \Z_{\geq 0} : 1\leq t\leq T,\,\, t-\frac{k_t}{2r^\vee}-\frac{i_t-1}{2} \leq s \leq t-\frac{k_t}{2r^\vee}+\frac{i_t-1}{2}\right\}.\nn\ee
\end{defn}
Note that the bounds on $s$ are indeed integers and that the top-most (i.e. lowest numbered) row containing a box is row $T-\frac{k_T}{2r^\vee}-\frac{i_T-1}{2}$. (Recall $r^\vee=1$ in type A and $r^\vee = 2$ in type B.)

\begin{prop}\label{snake2skew}  The map $\sn \mapsto \sd_\sn$ is a bijection from the set of snakes of length $T$ with elements in $\Iy$ to the set of skew diagrams whose non-empty columns are $1,2,\dots,T$  and none of whose columns has more than $N$ (resp. $N-1$) boxes in type A (resp. type B). 
\end{prop}
\begin{proof}
By definition $\sd_\sn$ has, in each column $t$, $1\leq t\leq T$, a stack of $i_t$ adjacent boxes, the top-most of which is in row $t-\frac{k_t}{2r^\vee}-\frac{i_t-1}{2}$, and no boxes elsewhere. It is clear that $\sn$ can be reconstructed from this set of boxes, so the map injective. To check $\sd_\sn$ is a skew diagram we must check that for each $t$, $1\leq t \leq T-1$, the bottom box in column $t$ is not above the bottom box in column $t+1$, and the top box in column $t$ is not above the top box in column $t+1$. These conditions are
\be k_{t+1}-k_t \geq (2+ i_{t+1}-i_t)r^\vee \quad\text{and}\quad k_{t+1}-k_t \geq (2+ i_t-i_{t+1})r^\vee\nn\ee
respectively. This is exactly the definition of snake position for points $(i_t,k_t)$ and $(i_{t+1},k_{t+1})$ in $\Iy$, \confer Definition \ref{snakepos}. So the map is onto the stated set of skew diagrams.
\end{proof}
\begin{rem}\label{sdrem}\begin{enumerate}[(i)]
  \item Minimal snakes (\S\ref{snakepos}) in $\Iy$ correspond to skew diagrams in which, for each occupied column except the last, either the top-most box in that column is in the same row as the top-most box in the next column, or the bottom-most box in that column is in the same row as the bottom-most box in the next column. If the top- or bottom-most boxes in \emph{all} the columns are aligned in one row, then the corresponding module is a minimal affinization. 
\item If a skew diagram has an empty column $j$ then it describes the same representation as the skew diagram in which each box to the right of column $j$ is translated one step up and to the left, \confer (\ref{mtab}). So there is no loss of generality in considering only skew diagrams that have no empty columns between occupied columns.
\end{enumerate}
\end{rem}

To each $\T\in\Tab\sd$ we associate a monomial in $\Z[Y_{i,k}^{\pm 1}]_{(i,k)\in \It}$ as follows:
\be M(\T) := \prod_{(i,j)\in \lambda/\mu} \boxed{\T(i,j)}_{\,2r^\vee(j-i)} \label{mtab}\ee
where
\begin{enumerate}[Type A:] 
\item  $\boxed{i}_k:= Y_{i-1,i+k}^{-1} Y_{i,i-1+k}^{\phantom{+1}}$,  $1\leq i\leq N+1$.
\item  $\boxed{i}_k:= Y_{i-1,2i+k}^{-1} Y^{\phantom{+1}}_{i,2i-2+k} $ ,  $1\leq i\leq N-1$\\ 
$\boxed{N}_k := Y_{N-1,2N+k}^{-1}  Y_{N,2N-3+k}Y^{\phantom{+1}}_{N,2N-1+k}$\\ 
$\boxed{0}_k := Y_{N,2N+1+k}^{-1} Y^{\phantom{+1}}_{N,2N-3+k}$\\
$\boxed{\bar N}_k :=  Y_{N,2N-1+k}^{-1}Y_{N,2N+1+k}^{-1} Y^{\phantom{+1}}_{N-1,2N-2+k}$  \\ 
$\boxed{\bar i}_k := Y^{-1}_{i,4N-2i+k} Y^{\phantom{+1}}_{i-1,4N-2-2i+k}$  , $1\leq i\leq N-1$
\end{enumerate}
with, by convention, $Y_{0,k}:= 1$ and $Y_{N+1,k}:=1$ for all $k\in\Z$. Note that, in both types A and B, $\{\boxed{i}_0\}_{i\in \Alp}$ are the monomials of $q$-character of the first fundamental representation, $\L(Y_{1,0})$. Compare for example Figure \ref{Bfundfig} in type $B_4$.

\begin{prop}
Let $\sn$ be a snake whose elements are in $\Iy$ and $\sd$ the corresponding skew diagram as in Definition \ref{sddef}. There is a bijection $\nops \to\Tab\sd; \ps \mapsto \T_\ps$ between non-overlapping tuples of paths and skew tableaux such that
\be M(\T_\ps) = \prod_{t=1}^T \mon(p_t) .\label{eqmon}\ee
\end{prop}
\begin{proof}
The required bijection is as follows:
\begin{enumerate}[Type A:]
\item For each $t$, $1\leq t\leq T$, $p_t$ is of the form $p_t= \big( (r,y^{(t)}_r) \big)_{0\leq r\leq N+1}$. 
Let \be \nn S_t:= \{ r\in \Alp :  y^{(t)}_{r}-y^{(t)}_{r-1} = -1 \}.\ee By definition of $\scr P_{i_t,k_t}$, $S_t$ consists of exactly $i_t$ letters from the alphabet $\Alp$. We enter these letters in the boxes in column $t$, in alphabetical order, starting at the top box and working downwards.
\item For each $t$, $1\leq t\leq T$, the path $p_t$ is given by a pair $(a_t,\bar a_t)$ where $a_t\in \scr P_{N,k_t-(2N-2i_t-1)}$ and $\bar a_t \in \scr P_{N,k+(2N-2i-1)}$. 
Let $a_t =: \big( (x_r^{(t)},y^{(t)}_r) \big)_{0\leq r \leq N}$ and  $\bar a_t =: \big( (\bar x_r^{(t)},\bar y^{(t)}_r) \big)_{0\leq r \leq N}$, and define 
\begin{align*}   S_t &:= \left\{ r\in \Alp :   1\leq r\leq N,\,\, y^{(t)}_r-y^{(t)}_{r-1} < 0 \right\} \\
 \bar S_t &:= \left\{\bar r\in \Alp : 1\leq r\leq N,\,\, \bar y^{(t)}_r- \bar y^{(t)}_{r-1} >0 \right\}. \end{align*} 
Starting from the top box in column $t$ and working downwards, we enter the letters in $S_t$ in alphabetical order. Then, starting from the bottom box in column $t$ and working upwards, we enter the letters in $\bar S_t$ in reverse alphabetical order. The condition $y^{(t)}_N>\bar y^{(t)}_N$ in the definition of $\scr P_{i_t,k_t}$ ensures that, in so doing, no box is filled twice. Finally we enter the letter $0$ into any boxes in (the middle of) column $t$ that remain unfilled.
\end{enumerate}
For each $t$, $1\leq t\leq T$, this prescription places the paths in $\scr P_{i_t,k_t}$ in bijection with the set of fillings of the column $t$ that obey the vertical rule (V) in the definition of skew tableaux above. It does so in such a way that $\mon(p_t) = \prod_{i:(i,t)\in\sd} \boxed{\T_\ps(i,t)}_{2r^\vee(t-i)}$, from which (\ref{eqmon}) follows.  

Finally, it is seen by inspection that the non-overlapping condition (\S\ref{overlapdef}) on the paths is equivalent, in the special case that the points $\sn$ lie in $\Iy$, to the horizontal rule (H) in the definition of skew tableaux.
\end{proof}

\begin{cor}\label{tabchar} Let $(i_t,k_t)\in\Iy$, $1\leq t\leq T$, be a snake.  Then
\be\label{tchr} \chi_q\left(\L\left(\prod_{t=1}^T Y_{i_t,k_t}\right)\right) = \sum_{\T\in\Tab\sd} M(\T)\ee
where $\sd$ is the corresponding skew diagram according to Definition \ref{sddef}.
\end{cor}
\begin{proof}
Given the preceding proposition, this is immediate from Theorem \ref{snakechar}.
\end{proof}

In type A, Corollary \ref{tabchar} was previously known, at least in the Yangian case  \cite{NT,CherednikGT}. 
In type B, Corollary \ref{tabchar} was previously known for those skew diagrams giving minimal affinizations \cite{HernandezMinAff}; \confer Remark \ref{sdrem} (i).

The right-hand side of (\ref{tchr}) is the tableaux expression for the \emph{Jacobi-Trudi determinant} in type A or B. Corollary \ref{tabchar} thus confirms the conjecture made in \cite{KOS} (and, specifically in the language of $q$-characters, in \cite{NN1}, Conjecture 2.2, part 1) that the type B Jacobi-Trudi determinant gives the $q$-character of an irreducible representation of the quantum affine algebra.

\begin{rem}\begin{enumerate}[(i)]
\item 
The paths in the present paper and those of \cite{NN1} are a priori unrelated, and are being used to solve distinct problems. Our paths are constructed to correspond to $q$-characters of fundamental representations, and the notion of ``non-overlapping'' is designed to pick out the $q$-character of the simple quotient of a suitably ordered tensor product of fundamental factors. By contrast, the ``non-intersecting'' paths of \cite{NN1} are used in deriving the skew tableaux expression for the Jacobi-Trudi determinant, and originate in the Gessel-Viennot path method for determinants \cite{GVpaths}. 
\item For simplicity in stating Definition \ref{sddef}, we defined $\Iy$ in type B to contain no points $(N,k)$, $k\in \Z$. In fact the tableaux description extends to snakes in which such points occur in neighbouring pairs: each such pair corresponds to a column with $>N-1$ boxes.  For example $\L(Y_{N,1} Y_{N,3+4k})$, $k\in \Z_{\geq 0}$ corresponds to a single column of $N+k$ boxes.  Thus, one could replace ``strictly'' by ``weakly'' in Remark \ref{lrrem}. Furthermore, \cite{KOS} define ``hatched'' tableaux, which allow one to delete, from such a snake, exactly one point $(N,k)$, thereby describing some spin-odd representations.
\end{enumerate}\end{rem}

\subsection{Examples}\label{sec:examples}
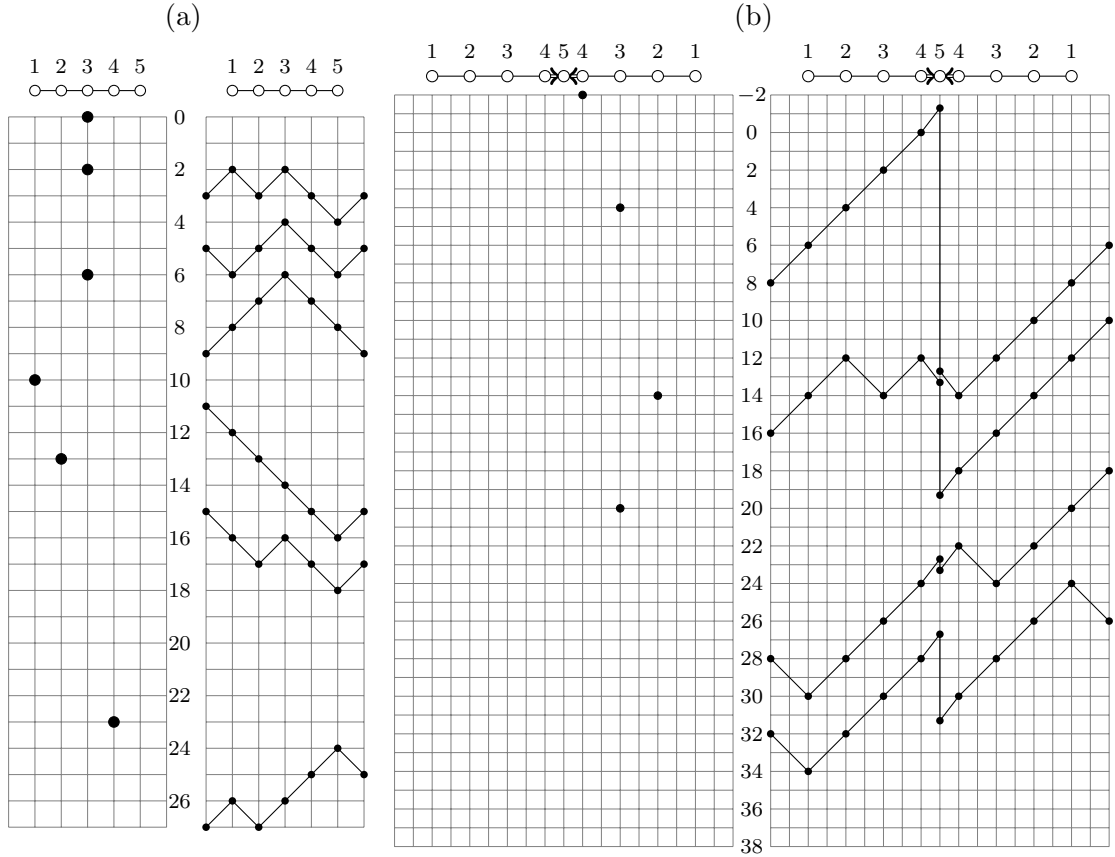
\begin{figure} \caption{\label{examplesnakefig} See \S\ref{sec:examples}. In types $A_5$, (a),  and $B_5$, (b): a snake whose elements lie in $\Iy$ and which therefore corresponds to a skew diagram (left), and a non-overlapping tuple of paths for that snake (right).}
\be\nn \begin{matrix} \text{(a)} & \text{(b)} \\ 
\begin{tikzpicture}[scale=.35,yscale=-1]
  \draw[help lines] (0,0) grid (6,27);
\draw (1,-1) -- (5,-1);
\foreach \x in {1,2,3,4,5} {\filldraw[fill=white] (\x,-1) circle (2mm) node[above=1mm] {$\scriptstyle\x$}; }
\foreach \x/\y in {3/0,3/2,3/6,1/10,2/13,4/23}{\filldraw (\x,\y) circle (2mm);}
\begin{scope}[xshift=7.5cm]
\draw[help lines] (0,0) grid (6,27);
\foreach \y in {0,2,4,6,8,10,12,14,16,18,20,22,24,26} {\node at (-1,\y) {$\scriptstyle\y$};}
\draw (1,-1) -- (5,-1);
\foreach \x in {1,2,3,4,5} {\filldraw[fill=white] (\x,-1) circle (2mm) node[above=1mm] {$\scriptstyle\x$}; }

\begin{scope}[every node/.style={minimum size=.1cm,inner sep=0mm,fill,circle}]
\draw (0,3) node{}  -- ++(1,-1) node{}  -- ++(1,1) node{}  -- ++(1,-1) node{}  -- ++(1,1) node{}  -- ++(1,1) node{}  -- ++(1,-1) node{} ;
\draw (0,5) node{}  -- ++(1,1) node{}  -- ++(1,-1) node{}  -- ++(1,-1) node{}  -- ++(1,1) node{}  -- ++(1,1) node{}  -- ++(1,-1) node{} ;
\draw (0,9) node{}  -- ++(1,-1) node{}  -- ++(1,-1) node{}  -- ++(1,-1) node{}  -- ++(1,1) node{}  -- ++(1,1) node{}  -- ++(1,1) node{} ;
\draw (0,11) node{}  -- ++(1,1) node{}  -- ++(1,1) node{}  -- ++(1,1) node{}  -- ++(1,1) node{}  -- ++(1,1) node{}  -- ++(1,-1) node{} ;
\draw (0,15) node{}  -- ++(1,1) node{}  -- ++(1,1) node{}  -- ++(1,-1) node{}  -- ++(1,1) node{}  -- ++(1,1) node{}  -- ++(1,-1) node{} ;
\draw (0,27) node{}  -- ++(1,-1) node{}  -- ++(1,1) node{}  -- ++(1,-1) node{}  -- ++(1,-1) node{}  -- ++(1,-1) node{}  -- ++(1,1) node{} ;
\end{scope}

\end{scope}
\node at (0,28) {$$};
\end{tikzpicture} &
\begin{tikzpicture}[scale=.25,yscale=-1]
\draw[help lines] (0,0) grid (18,40);
\draw (2,-1) -- (8,-1); \draw (10,-1) -- (16,-1);
\draw[double,->] (8,-1) -- (8.8,-1); \draw[double,->] (10,-1) -- (9.2,-1);
\filldraw[fill=white] (9,-1) circle (3mm) node[above=1mm] {$\scriptstyle 5$};
\foreach \x in {1,2,3,4} {
\filldraw[fill=white] (2*\x,-1) circle (3mm) node[above=1mm] {$\scriptstyle\x$}; 
\filldraw[fill=white] (2*9-2*\x,-1) circle (3mm) node[above=1mm] {$\scriptstyle\x$}; }
\foreach \x/\y in {10/0,12/6,14/16,12/22}{\filldraw (\x,\y) circle (2mm);}
\begin{scope}[xshift=20cm]
\draw[help lines] (0,0) grid (18,40);
\draw (2,-1) -- (8,-1); \draw (10,-1) -- (16,-1);
\draw[double,->] (8,-1) -- (8.8,-1); \draw[double,->] (10,-1) -- (9.2,-1);
\filldraw[fill=white] (9,-1) circle (3mm) node[above=1mm] {$\scriptstyle 5$};
\foreach \x in {1,2,3,4} {
\filldraw[fill=white] (2*\x,-1) circle (3mm) node[above=1mm] {$\scriptstyle\x$}; 
\filldraw[fill=white] (2*9-2*\x,-1) circle (3mm) node[above=1mm] {$\scriptstyle\x$}; }
\foreach \y in {-2,0,2,4,6,8,10,12,14,16,18,20,22,24,26,28,30,32,34,36,38} {\node at (-1,\y+2) {$\scriptstyle\y$};}

\begin{scope}[every node/.style={minimum size=.1cm,inner sep=0mm,fill,circle}]
\draw (18,8) node{}  -- ++(-2,2) node{}  -- ++(-2,2) node{}  -- ++(-2,2) node{}  -- ++(-2,2) node{}  -- ++(-1,-1.3) node{}  -- (9,.7) node{}  -- ++(-1,1.3) node{}  --++(-2,2) node{}  -- ++(-2,2) node{}  -- ++(-2,2) node{}  -- ++(-2,2) node{}   ;
\draw (18,12) node{}  -- ++(-2,2) node{}  -- ++(-2,2) node{}  -- ++(-2,2) node{}  -- ++(-2,2) node{}  -- ++(-1,1.3) node{}  -- (9,15.3) node{}  -- ++(-1,-1.3) node{}  --++(-2,2) node{}  -- ++(-2,-2) node{}  -- ++(-2,2) node{}  -- ++(-2,2) node{}   ;
\draw (18,20) node{}  -- ++(-2,2) node{}  -- ++(-2,2) node{}  -- ++(-2,2) node{}  -- ++(-2,-2) node{}  -- ++(-1,1.3) node{}  -- ++(0,-.6) node{}  -- ++(-1,1.3) node{}  -- ++(-2,2) node{}  -- ++(-2,2) node{}  -- ++(-2,2) node{}  -- ++(-2,-2) node{}   ;
\draw (18,28) node{}  -- ++(-2,-2) node{}  -- ++(-2,2) node{}  -- ++(-2,2) node{}  -- ++(-2,2) node{}  -- ++(-1,1.3) node{}  -- ++ (0,-4.6) node{}  --++(-1,1.3) node{}  -- ++(-2,2) node{}  -- ++(-2,2) node{}  -- ++(-2,2) node{}  -- ++(-2,-2) node{}   ;
\end{scope}
\end{scope}
\end{tikzpicture} \end{matrix}\ee
\end{figure}

\begin{exmp}
In type $A_5$, $\L(Y_{3,0} Y_{3,2} Y_{3,6} Y_{1,10} Y_{2,13} Y_{4,23})$ is a snake module. Its snake is shown in Figure \ref{examplesnakefig}(a). Its $q$-character includes the monomial 
\be\label{monaex} Y_{1,2}Y_{2,3}^{-1} Y_{3,2} Y_{5,4}^{-1} Y_{1,6}^{-1} Y_{3,4} Y_{5,6}^{-1} Y_{3,6} Y_{5,16}^{-1} Y_{2,17}^{-1} Y_{3,16} Y_{5,18}^{-1} Y_{1,26} Y_{2,27}^{-1} Y_{5,24}.\ee The non-overlapping tuple of paths corresponding to this monomial is also shown in Figure \ref{examplesnakefig}(a). The skew diagram for this representation, and the skew tableau for the particular monomial (\ref{monaex}), are respectively
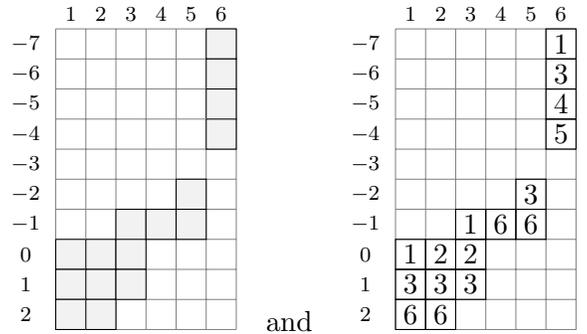
\be\nn\begin{tikzpicture}[scale=.4,yscale=-1]
\begin{scope}[xshift=-.5cm,yshift=-.5cm] \draw[help lines] (1,-7) grid (7,3); \end{scope}
\foreach \y in {-7,-6,-5,-4,-3,-2,-1,0,1,2} {\node at (-.5,\y) {$\scriptstyle\y$};}
\foreach \x in {1,2,3,4,5,6} {\node at (\x,-8) {$\scriptstyle\x$};}
\begin{scope}[every node/.style={minimum size=.4cm,inner sep=0mm,fill=gray!10,draw,rectangle}]
\foreach \x/\y in {1/0,1/1,1/2,2/0,2/1,2/2,3/-1,3/0,3/1,4/-1,5/-2,5/-1,6/-7,6/-6,6/-5,6/-4} {\node at (\x,\y) {};}
\end{scope}
\end{tikzpicture}\quad\text{and}\quad
\begin{tikzpicture}[scale=.4,yscale=-1]
\begin{scope}[xshift=-.5cm,yshift=-.5cm] \draw[help lines] (1,-7) grid (7,3); \end{scope}
\foreach \y in {-7,-6,-5,-4,-3,-2,-1,0,1,2} {\node at (-.5,\y) {$\scriptstyle\y$};}
\foreach \x in {1,2,3,4,5,6} {\node at (\x,-8) {$\scriptstyle\x$};}
\begin{scope}[every node/.style={minimum size=.4cm,inner sep=0mm,draw,rectangle}]
\foreach \x/\y/\c in {1/0/1,1/1/3,1/2/6,2/0/2,2/1/3,2/2/6,3/-1/1,3/0/2,3/1/3,4/-1/6,5/-2/3,5/-1/6,6/-7/1,6/-6/3,6/-5/4,6/-4/5} {\node at (\x,\y) {\c};}
\end{scope}
\end{tikzpicture}\,\,.\ee
\end{exmp}

\begin{exmp}
In type $B_5$, $\L(Y_{4,-2} Y_{3,4} Y_{2,14} Y_{3,20} )$ is a snake module. Its snake is shown in Figure \ref{examplesnakefig}(b).  Its $q$-character includes the monomial 
\be\label{monbex} Y_{5,-1} Y_{4,14}^{-1} Y_{2,12} Y_{3,14}^{-1} Y_{4,12} Y_{5,19}^{-1}Y_{1,30}^{-1} Y_{4,22} Y_{3,24}^{-1} Y_{1,34}^{-1} Y_{5,27} Y_{5,31}^{-1}Y_{1,24}.\ee The non-overlapping tuple of paths corresponding to this monomial is also shown in Figure \ref{examplesnakefig}(b). Since the points of this snake all lie in $\Iy$, the monomials can equivalently be described by skew tableaux. The skew diagram for this representation, and the skew tableau for the particular monomial (\ref{monbex}), are respectively
\be\nn \begin{tikzpicture}[scale=.4,yscale=-1]
\begin{scope}[xshift=-.5cm,yshift=-.5cm] \draw[help lines] (1,-2) grid (5,4); \end{scope}
\foreach \y in {-2,-1,0,1,2,3} {\node at (-.5,\y) {$\scriptstyle\y$};}
\foreach \x in {1,2,3,4} {\node at (\x,-3) {$\scriptstyle\x$};}
\begin{scope}[every node/.style={minimum size=.4cm,inner sep=0mm,fill=gray!10,draw,rectangle}]
\foreach \x/\y in {1/0,1/1,1/2,1/3,2/0,2/1,2/2,3/-1,3/0,4/-2,4/-1,4/0} {\node at (\x,\y) {};}
\end{scope}
\end{tikzpicture}
\quad\text{and}\quad
\begin{tikzpicture}[scale=.4,yscale=-1]
\begin{scope}[xshift=-.5cm,yshift=-.5cm] \draw[help lines] (1,-2) grid (5,4); \end{scope}
\foreach \y in {-2,-1,0,1,2,3} {\node at (-.5,\y) {$\scriptstyle\y$};}
\foreach \x in {1,2,3,4} {\node at (\x,-3) {$\scriptstyle\x$};}
\begin{scope}[every node/.style={minimum size=.4cm,inner sep=0mm,draw,rectangle}]
\foreach \x/\y/\c in {1/0/5,1/1/0,1/2/0,1/3/0,2/0/0,2/1/\bar 5,2/2/\bar 3,3/-1/4,3/0/\bar 1,4/-2/1,4/-1/0,4/0/\bar 1} {\node at (\x,\y) {$\c$};}
\end{scope}
\end{tikzpicture}\,\,.\ee
\end{exmp}



\def\cprime{$'$}
\providecommand{\bysame}{\leavevmode\hbox to3em{\hrulefill}\thinspace}
\providecommand{\MR}{\relax\ifhmode\unskip\space\fi MR }
\providecommand{\MRhref}[2]{%
  \href{http://www.ams.org/mathscinet-getitem?mr=#1}{#2}
}
\providecommand{\href}[2]{#2}

\end{document}